\documentclass[11pt, reqno]{amsart}

\usepackage{amsmath, amsfonts, amssymb, amsbsy, bigstrut, graphicx, enumerate,  upref}

\usepackage{pstricks}

\usepackage[noadjust,nosort]{cite}

\newcommand{\ep}{\epsilon}

\newcommand{\ms}{\mathcal{S}}

\newtheorem{thm}{Theorem}[section]
\newtheorem{lmm}[thm]{Lemma}
\newtheorem{cor}[thm]{Corollary}
\newtheorem{prop}[thm]{Proposition}

\newcommand{\bigavg}[1]{\biggl\langle #1 \biggr\rangle}

\newcommand{\ee}{\mathbb{E}}

\newcommand{\ma}{\mathcal{A}}

\newcommand{\cp}{\mathcal{P}}
\newcommand{\cs}{\mathcal{S}}

\newcommand{\mx}{\mathcal{X}}

\newcommand{\ra}{\rightarrow}
\newcommand{\rr}{\mathbb{R}}
\newcommand{\smallavg}[1]{\langle #1 \rangle}

\newcommand{\tr}{\operatorname{Tr}}

\newcommand{\zz}{\mathbb{Z}}

\newcommand{\fpar}[2]{\frac{\partial #1}{\partial #2}}
\newcommand{\spar}[2]{\frac{\partial^2 #1}{\partial #2^2}}
\newcommand{\mpar}[3]{\frac{\partial^2 #1}{\partial #2 \partial #3}}

\newcommand{\fs}{\mathbb{S}}
\newcommand{\fd}{\mathbb{D}}
\newcommand{\ftw}{\mathbb{T}}
\newcommand{\fst}{\mathbb{M}}

\numberwithin{equation}{section}

\usepackage{hyperref}

\begin{document}
%\title[Gauge-string duality]{Towards a mathematical understanding of gauge-string duality}
\title[Solution of $SO(N)$ lattice gauge theory]{Rigorous solution of strongly coupled $SO(N)$ lattice gauge theory in the large $N$ limit}
\author{Sourav Chatterjee}
\address{\newline Department of Statistics \newline Stanford University\newline Sequoia Hall, 390 Serra Mall \newline Stanford, CA 94305\newline \newline \textup{\tt souravc@stanford.edu}}
\thanks{Research partially supported by NSF grant DMS-1441513}
\keywords{Gauge-string duality, AdS/CFT, lattice gauge theory, Yang--Mills, string theory, matrix integral}
\subjclass[2010]{70S15, 81T13, 81T25, 82B20}

\begin{abstract}
%This paper proves an explicit formula for the Wilson loop expectations in the large $N$ limit of strongly coupled $SO(N)$ lattice gauge theory. The formula is given by a sum over trajectories in a kind of string theory on the lattice, thereby verifying the widely conjectured duality between gauge theories and string theories in this example.
The main result of this paper is a rigorous  computation of Wilson loop expectations in strongly coupled $SO(N)$ lattice gauge theory in the large $N$ limit, in any dimension.  The formula appears as  an absolutely convergent sum over trajectories in a kind of string theory on the lattice, demonstrating an explicit gauge-string duality. The generality of the proof technique  may allow it to be extended other gauge groups.%This may give a first step to a conceptual understanding of gauge-string duality, going beyond exactly solvable models. 
\end{abstract}

\maketitle

\tableofcontents

\section{Introduction}\label{intro}
Quantum gauge theories, also called quantum Yang--Mills theories, form the basic building blocks of the Standard Model of quantum mechanics. In an effort to gain a better theoretical understanding of quantum gauge theories, Wilson \cite{wilson74} introduced a discretized version of these theories in 1974. These have since come to be known as `lattice gauge theories' or `lattice Yang--Mills theories'. Lattice gauge theories are mathematically well-defined objects. However, proving theorems about them has turned out to be quite challenging. Indeed, proving the existence of a suitable scaling limit of lattice gauge theories would solve the Clay  millennium problem of Yang--Mills existence \cite{jaffewitten}. Partial progress in this direction was made by Brydges, Fr\"ohlich and Seiler \cite{brydgesetal79, brydgesetal80, brydgesetal81}, Ba\l aban \cite{balaban83, balaban85, balaban85b, balaban87} and Magnen, Rivasseau and S\'en\'eor~\cite{magnenetal93}.

Even dealing with lattice gauge theories by themselves, without taking  scaling limits, has not been easy. Wilson's original goal was to study  certain quantities that are now known as `Wilson loop expectations'. These are the basic objects of interest in lattice gauge theories. A `solution' of a lattice gauge theory would refer to a formula for Wilson loop expectations. Although a number of impressive mathematical advances have occurred over the years \cite{driver89, frohlichspencer82, gopfertmack82, grosswitten80, gross83, grossetal89, guth80,levy03,levy06,levy11, luscher10, osterwalderseiler78, sengupta93, sengupta97, tomboulisyaffe85, witten89, witten90, witten92}, no lattice gauge theory in any dimension higher than two has been rigorously solved in the above~sense. The monograph \cite{seiler82} gives a comprehensive introduction to lattice gauge theories and their connections with quantum physics for mathematically oriented readers.

A promising approach to solving lattice gauge theories was proposed by  't Hooft \cite{thooft74}, who suggested an approximate computation of the partition function of lattice gauge theories in a certain kind of limit. The solution posited by 't Hooft involved a connection between matrix integrals and the enumeration of planar maps, an idea that has had ramifications in various branches of physics and mathematics over the last forty years. A large part of this connection is now mathematically well-grounded \cite{collinsetal09, ercolanimclaughlin03, ercolanimclaughlin08, eynard03, eynard05, eynardorantin07, eynardorantin09,guionnet04, guionnet06, guionnetetal12, guionnetmaida05, guionnetnovak14, guionnetzeitouni02}. The paper \cite{lucinipanero13} contains a survey of the developments on the physics side. 
%due to the efforts of Guionnet, Eynard, Orantin, Ercolani, McLaughlin and their various coauthors~\cite{collinsetal09, ercolanimclaughlin03, ercolanimclaughlin08, eynard03, eynard05, eynardorantin07, eynardorantin09,guionnet04, guionnet06, guionnetetal12, guionnetmaida05, guionnetnovak14, guionnetzeitouni02}. 

Lattice gauge theories have a parameter called the coupling strength. When the coupling strength is large, the system is said to be at strong coupling. Otherwise, it is in the weak coupling regime. The map enumeration technique of 't Hooft is a perturbative expansion for weakly coupled lattice gauge theories in the 't Hooft limit. This paper, on the other hand, gives a formula for Wilson loop expectations in the 't Hooft limit of $SO(N)$ lattice gauge theory at strong coupling. The formula involves a sum over all possible trajectories in a kind of string theory on the lattice. This establishes an explicit `gauge-string duality'. Dualities between gauge theories and string theories have been conjectured in the physics literature for a long time. After the seminal contribution of Maldacena \cite{maldacena97} in 1997, the investigation of gauge-string duality --- and  a special case of it, the AdS/CFT duality --- has become one of the most popular areas of research in theoretical physics. 

Incidentally, the groups $SU(N)$ and $U(N)$ are the ones that are more relevant for physics than $SO(N)$. However, the method of this paper is somewhat easier to implement for $SO(N)$ than $SU(N)$ or $U(N)$, and that is why $SO(N)$ has been chosen as the first candidate for investigation. Using the same technique, analogous results for $SU(N)$ have been recently derived in a lengthier paper by Jafarov~\cite{jafarov16}.

The formula obtained in this manuscript (Theorem \ref{mainthm}) is the first rigorously proved verification of any kind of gauge-string duality in any dimension higher than two, except for some exactly solvable models like three dimensional Chern--Simons gauge theory  \cite{witten89, witten90, witten92}. However, two things need to clarified. First, this paper deals only with theories on lattices, whereas the physicists who think about gauge-string duality are interested mainly in continuum theories. But mathematicians do not even know how to define the continuum theories, so that is a moot point. Second, physicists who work with lattice gauge theories are aware of combinatorial formulas for recursively computing Wilson loop expectations in the 't Hooft limit. But as far as I understand, these formulas are not rigorously proved.

The method of proof, very briefly, is as follows. In the physics literature, Makeenko and Migdal~\cite{makeenkomigdal79} showed in 1979 that Wilson loops expectations in lattice gauge theories satisfy a complicated set of equations that are now known as the Makeenko--Migdal master loop equations. These are special cases of the class of equations that are broadly known as Schwinger--Dyson equations in the random matrix literature. The first step in this paper is to rigorously derive a version of the Makeenko--Migdal equations. This is different than the original, because Makeenko and Migdal assumed a certain mathematically unproved `factorization property' of Wilson loops, which is not assumed in this article but rather derived later as a corollary of the main theorem. The next step is to rigorously solve the modified Makeenko--Migdal equations, and then write the solution as an absolutely convergent sum indexed by trajectories of strings in a discrete string theory. Solving the Makeenko--Migdal equations is a challenge that has stood for the last thirty-five years (in spite of some famous attempts in~\cite{bhanotetal82, eguchikawai82, gonzalez83} and more recently in~\cite{kovtunetal,uy08,du16}), so it is interesting to see that it can actually be done.  Prior to this work, the only rigorously proved case was that of dimension two, where the Makeenko--Migdal equations were analyzed by L\'evy \cite{levy11}. %, and even more interesting that the solution involves a duality with a string theory. 

The approach outlined above is fairly general, and may apply to other models. It is not dependent on any magical integrability property of the system under consideration, or connections to other integrable systems. %The reason why this is done for $SO(N)$ in this paper is nothing in particular; the proof uses no special property of $SO(N)$ that is unavailable for other matrix groups such as $SU(N)$ or $U(N)$. 

The next section introduces some basic notation and terminology that will be used throughout this article, and also introduces a `lattice string theory'. The main result and some corollaries are  presented in Section \ref{results}.

%\section{Loops and operations on loops}\label{not1}

\section{Notation, terminology and a lattice string theory}\label{not1}
The following notation will be fixed throughout this manuscript. Let $d\ge 2$ be a positive integer. Let $\zz^d$ be the $d$-dimensional integer lattice. Let $E$ be the set of directed nearest neighbor edges of $\zz^d$. If $e\in E$, then $e$ has a beginning point, which will be denoted by $u(e)$, and an ending point, which will be denoted by $v(e)$. Furthermore, we will say that $e$ is positively oriented if the ending point is greater than the beginning point in the lexicographic ordering, and negatively oriented otherwise. Let $E^+$ be the set of positively oriented edges and $E^-$ be the set of negatively oriented edges. If $e = (u,v)\in E$, we will denote the edge $(v,u)$ by $e^{-1}$. Note that if $e\in E^+$ then $e^{-1}\in E^-$ and vice versa. 

\subsection{Loops and loop sequences}
A `path' $\rho$ in the lattice $\zz^d$ is defined to be a sequence of edges $e_1,e_2,\ldots, e_n\in E$ such that $v(e_i)=u(e_{i+1})$ for $i=1,2,\ldots, n-1$. We will usually write $\rho = e_1e_2\cdots e_n$ and say that $\rho$ has length $n$. The length of $\rho$ will be denoted by $|\rho|$. The path $\rho$ will be called `closed' if $v(e_n)= u(e_1)$. We also define a `null path' that has no edges. The null path is denoted by $\emptyset$ and is defined to have length zero. By definition, the null path is closed. The `inverse' of the path $\rho$ is defined as 
\[
\rho^{-1}:= e_n^{-1}e_{n-1}^{-1}\cdots e_1^{-1}\, .
\]
If $\rho' = e_1'\cdots e_m'$ is another path such that $v(e_n)=u(e_1')$, then the concatenated path $\rho\rho'$ is defined~as 
\[
\rho\rho' := e_1\cdots e_n e_1'\cdots e_m'\,.
\]
If $\rho=e_1\cdots e_n$ is a closed path, we will say that another closed path $\rho'$ is `cyclically equivalent' to $\rho$, and write $\rho \sim \rho'$, if 
\[
\rho' = e_i e_{i+1}\cdots e_n e_1e_2\cdots e_{i-1}
\]
for some $2\le i\le n$. The equivalence classes will be called `cycles'. The length of a cycle $l$, denoted by $|l|$, is defined to be equal to the length of any of the closed paths in $l$. The `null cycle' is defined to be the equivalence class containing only the null path, and its length is defined to be zero.

It will sometimes be useful to talk about edges at particular locations in paths and cycles. Locations in paths are easy to define: If $\rho = e_1\cdots e_n$ is a path, then $e_k$ is the edge at its $k^{\mathrm{th}}$ location. For cycles, we define the `first edge' of a cycle by some arbitrary rule. Once the first edge is defined, the definition of the $k^{\mathrm{th}}$  edge is automatic. The $k^{\mathrm{th}}$ edge will sometimes be alternately called the edge at `location $k$'. If $e_1\cdots e_n$ is a closed path and $l$ is the cycle containing this path, we will often forget about the distinction between the two and simply write $l = e_1\cdots e_n$. %If $e_k$ is the $k$th edge of a cycle $l$ of length $n$, the `representative path' of the $l$ is the path $e_1e_2\cdots e_n$. We will often write $l=e_1e_2\cdots e_n$.

We will say that a path $\rho = e_1e_2\cdots e_n$ has a `backtrack' at location $i$, where $1\le i\le n-1$, if $e_{i+1}=e_i^{-1}$. If $\rho$ is a closed path, then we will say that it has a backtrack at location $n$ if $e_1 = e_n^{-1}$. In a closed path of length $n$, a backtrack that occurs at some location $i\le n-1$ will be called an `interior backtrack', and a backtrack at location $n$ will be called a `terminal backtrack'.

Figure \ref{nobackfig} contrasts a closed path with a backtrack versus a closed path with no backtracks.

\begin{figure}[t]
%\begin{pdfpic}
\begin{pspicture}(1,.5)(11,3.5)
\psset{xunit=1cm,yunit=1cm}
\psline{->}(1,3)(1.6,3)
\psline{-}(1.6,3)(2,3)
\psline{->}(2,3)(2,2.4)
\psline{-}(2,2.4)(2,2)
\psline{->}(2,2)(2.6, 2)
\psline{-}(2.6,2)(3,2)
\psline{->}(3,2)(3.6,2)
\psline{-}(3.6,2)(4,2)
\psline{-}(4,2)(4, 1.95)
\psline{->}(4, 1.95)(3.4, 1.95)
\psline{-}(3.4, 1.95)(3, 1.95)
\psline{->}(3,1.95)(3,1.4)
\psline{-}(3,1.4)(3,1)
\psline{->}(3, 1)(2, 1)
\psline{-}(2, 1)(1, 1)
\psline{->}(1,1)(1, 2)
\psline{-}(1, 2)(1, 3)
\psline{->}(7,3)(7.6,3)
\psline{-}(7.6,3)(8,3)
\psline{->}(8,3)(8,2.4)
\psline{-}(8,2.4)(8,2)
\psline{->}(8,2)(8.6, 2)
\psline{-}(8.6,2)(9,2)
\psline{->}(9,2)(9.6,2)
\psline{-}(9.6,2)(10,2)
\psline{->}(10,2)(10, 2.6)
\psline{-}(10,2.6)(10,3)
\psline{->}(10,3)(10.6,3)
\psline{-}(10.6,3)(11,3)
\psline{->}(11,3)(11,2.4)
\psline{-}(11,2.4)(11,1.95)
\psline{->}(11,1.95)(10.4, 1.95)
\psline{-}(10.4, 1.95)(10, 1.95)
\psline{->}(10, 1.95)(9.4, 1.95)
\psline{-}(9.4, 1.95)(9, 1.95)
\psline{->}(9,1.95)(9,1.4)
\psline{-}(9,1.4)(9,1)
\psline{->}(9, 1)(8, 1)
\psline{-}(8, 1)(7, 1)
\psline{->}(7,1)(7, 2)
\psline{-}(7, 2)(7, 3)
%\psline[linestyle = dotted]{->}(5.5, 2)(7, 2)
%\rput(5.5,2){\large $\longrightarrow$}
%%
%\psframe(1,1)(2,2)
%\rput(0,-.2){\small $x$}
%\rput(8.9,-.2){\small $y$}
%\psline{*-*}(0,0)(9,0)
%\pscurve{-}(0,0)(2,2)(4.5,-1.1)(7,.5)(9,0)
%\psline[linestyle=dashed]{->}(2,1.2)(2,2)
%\rput(2,1){\small $D(x,y)$}
%\psline[linestyle=dashed]{->}(2, .8)(2,0)
%\rput(4.5,-.7){\small $G(x,y)$}
\end{pspicture}
%\end{pdfpic}
\caption{A closed path with a backtrack (on the left) versus a closed path with no backtracks (on the right). The closely spaced double lines denote the same edge traversed twice.}
\label{nobackfig}
\end{figure}
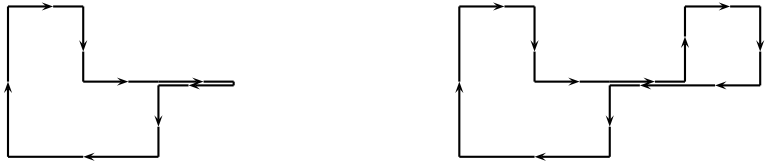

Given a path $\rho=e_1e_2\cdots e_n$ that has a backtrack at location $i$, the path obtained by erasing this backtrack is defined as
\[
\rho' := e_1e_2\cdots e_{i-1}e_{i+2}\cdots e_n\, .
\]
It is easy to check that $\rho'$ is indeed a path, and it is closed if $\rho$ is closed. If  $\rho$ is closed and has a backtrack at location $n$, then backtrack erasure at location $n$ results in
\[
\rho' := e_2e_3\cdots e_{n-1}\, .
\]
Again, it is easy to check that $\rho'$ is a closed path. Figure \ref{backfig} illustrates a single backtrack erasure.

\begin{figure}[t]
%\begin{pdfpic}
\begin{pspicture}(1,.5)(11,3.5)
\psset{xunit=1cm,yunit=1cm}
\psline{->}(1,3)(1.6,3)
\psline{-}(1.6,3)(2,3)
\psline{->}(2,3)(2,2.4)
\psline{-}(2,2.4)(2,2)
\psline{->}(2,2)(2.6, 2)
\psline{-}(2.6,2)(3,2)
\psline{->}(3,2)(3.6,2)
\psline{-}(3.6,2)(4,2)
\psline{-}(4,2)(4, 1.95)
\psline{->}(4, 1.95)(3.4, 1.95)
\psline{-}(3.4, 1.95)(3, 1.95)
\psline{->}(3,1.95)(3,1.4)
\psline{-}(3,1.4)(3,1)
\psline{->}(3, 1)(2, 1)
\psline{-}(2, 1)(1, 1)
\psline{->}(1,1)(1, 2)
\psline{-}(1, 2)(1, 3)
\psline[linestyle = dotted]{->}(5.5, 2)(7, 2)
%\rput(5.5,2){\large $\longrightarrow$}
%%
\psline{->}(8,3)(8.6,3)
\psline{-}(8.6,3)(9,3)
\psline{->}(9,3)(9,2.4)
\psline{-}(9,2.4)(9,2)
\psline{->}(9,2)(9.6, 2)
\psline{-}(9.6,2)(10,2)
\psline{->}(10,2)(10,1.4)
\psline{-}(10,1.4)(10,1)
\psline{->}(10, 1)(9, 1)
\psline{-}(9, 1)(8, 1)
\psline{->}(8,1)(8, 2)
\psline{-}(8, 2)(8, 3)
%\psframe(1,1)(2,2)
%\rput(0,-.2){\small $x$}
%\rput(8.9,-.2){\small $y$}
%\psline{*-*}(0,0)(9,0)
%\pscurve{-}(0,0)(2,2)(4.5,-1.1)(7,.5)(9,0)
%\psline[linestyle=dashed]{->}(2,1.2)(2,2)
%\rput(2,1){\small $D(x,y)$}
%\psline[linestyle=dashed]{->}(2, .8)(2,0)
%\rput(4.5,-.7){\small $G(x,y)$}
\end{pspicture}
%\end{pdfpic}
\caption{A single backtrack erasure.}
\label{backfig}
\end{figure}
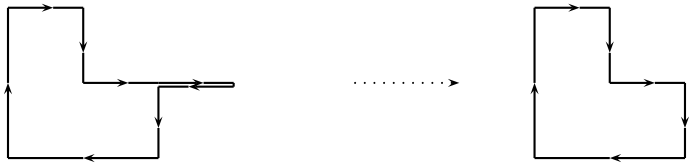

Note that backtrack erasures always decrease the length of the path  (although they can produce new backtracks that were not there in the original path, so the number of available backtracks may not decrease). Therefore, we cannot go on indefinitely erasing backtracks. If we start with a closed path $\rho$ and keep erasing backtracks successively, we must at some point end up with a closed path with no backtracks. Such a path will be called a `nonbacktracking closed path', and its cyclical equivalence class will be called a `nonbacktracking cycle', since each member of such an equivalence class is nonbacktracking. Nonbacktracking cycles will be called `loops'. The null cycle is by definition a loop, called the null loop. The following simple lemma shows that the nonbacktracking closed path obtained by successively erasing backtracks from a closed path is unique up to cyclical equivalence.
\begin{lmm}\label{core}
Let $\rho$ and $\rho'$ be  cyclically equivalent closed paths, and let $\rho_1$ and $\rho_2$ be two nonbacktracking closed paths obtained by successively erasing backtracks starting from $\rho$ and $\rho'$ respectively. Then $\rho_1$ and $\rho_2$ are cyclically equivalent. In particular, the case $\rho=\rho'$  implies that any two sequences of backtrack erasures starting from the same closed path must lead to cyclically equivalent nonbacktracking closed paths.
\end{lmm}
This lemma is proved in Section \ref{coresec}. The lemma allows us to define a `nonbacktracking core' of any cycle $l$ --- it is simply the unique loop obtained by successive backtrack erasures until there are no more backtracks. The nonbacktracking core of a cycle $l$ will be denoted by $[l]$. Figure \ref{corefig} illustrates a nonbacktracking core obtained by succesively erasing all backtracks. One can check in this figure that the order of erasing backtracks does not matter.

\begin{figure}[t]
%\begin{pdfpic}
\begin{pspicture}(1,.5)(11,4.5)
\psset{xunit=1cm,yunit=1cm}
\psline{->}(1,3)(1.6,3)
\psline{-}(1.6,3)(2,3)
\psline{->}(2,3)(2,2.4)
\psline{-}(2,2.4)(2,2)
\psline{->}(2,2)(2.6, 2)
\psline{-}(2.6,2)(3,2)
\psline{->}(3,2)(3,1.4)
\psline{-}(3,1.4)(3,1)
\psline{->}(3,1)(3.6,1)
\psline{-}(3.6,1)(4,1)
\psline{->}(4,1)(4,1.6)
\psline{-}(4,1.6)(4,2)
\psline{-}(4,2)(4.05,2)
\psline{->}(4.05,2)(4.05, 1.4)
\psline{-}(4.05, 1.4)(4.05, 1)
\psline{->}(4.05, 1)(4.6,1)
\psline{-}(4.6,1)(5,1)
\psline{-}(5,1)(5, .95)
\psline{->}(5, .95)(4.4, .95)
\psline{-}(4.4, .95)(4, .95)
\psline{->}(4, .95)(3.4, .95)
\psline{-}(3.4, .95)(3, .95)
\psline{->}(3, .95)(2, .95)
\psline{-}(2, .95)(.95, .95)
\psline{->}(.95,.95)(.95, 2)
\psline{-}(.95, 2)(.95, 3)
\psline{->}(.95,3)(.95, 3.6)
\psline{-}(.95, 3.6)(.95,4)
\psline{-}(.95,4)(1,4)
\psline{->}(1,4)(1,3.4)
\psline{-}(1,3.4)(1,3)
\psline[linestyle = dotted]{->}(5.5, 2)(7, 2)
%\rput(5.5,2){\large $\longrightarrow$}
%%
\psline{->}(8,3)(8.6,3)
\psline{-}(8.6,3)(9,3)
\psline{->}(9,3)(9,2.4)
\psline{-}(9,2.4)(9,2)
\psline{->}(9,2)(9.6, 2)
\psline{-}(9.6,2)(10,2)
\psline{->}(10,2)(10,1.4)
\psline{-}(10,1.4)(10,1)
\psline{->}(10, 1)(9, 1)
\psline{-}(9, 1)(8, 1)
\psline{->}(8,1)(8, 2)
\psline{-}(8, 2)(8, 3)
%\psframe(1,1)(2,2)
%\rput(0,-.2){\small $x$}
%\rput(8.9,-.2){\small $y$}
%\psline{*-*}(0,0)(9,0)
%\pscurve{-}(0,0)(2,2)(4.5,-1.1)(7,.5)(9,0)
%\psline[linestyle=dashed]{->}(2,1.2)(2,2)
%\rput(2,1){\small $D(x,y)$}
%\psline[linestyle=dashed]{->}(2, .8)(2,0)
%\rput(4.5,-.7){\small $G(x,y)$}
\end{pspicture}
%\end{pdfpic}
\caption{Nonbactracking core obtained by backtrack erasures.}
\label{corefig}
\end{figure}
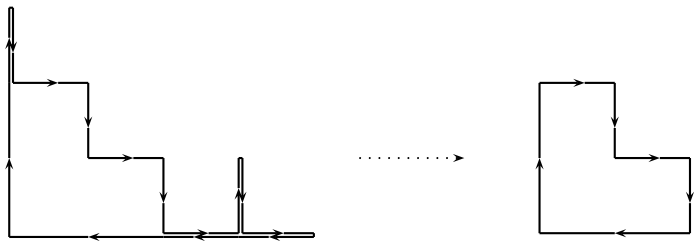

A nonbacktracking closed path of length four will be called a plaquette. Let $\cp$ be the set of plaquettes. Let $\cp^+$ be the subset of plaquettes $p=e_1e_2 e_3e_4$ such that $u(e_1)$ is lexicographically the smallest among all the vertices in $p$ and $v(e_1)$ is the second-smallest. It is not difficult to see that for any $p\in \cp$, either $p$ or $p^{-1}$ is cyclically equivalent to a unique element of $\cp^+$.  An element of $\cp^+$ will be called a `positively oriented plaquette'. 

If $\Lambda$ is a subset of $\zz^d$, let $E_\Lambda$ be the set of directed edges whose beginning and ending points are both in $\Lambda$, and define $E^+_\Lambda$ and $E^-_\Lambda$ analogously. Let $\cp_\Lambda$ be the set of plaquettes whose vertices are all in $\Lambda$, and define $\cp^+_\Lambda$ analogously.

Often, we will need to deal with collections of loops instead of a single loop. If $(l_1,\ldots, l_n)$ and $(l'_1,\ldots, l'_m)$ are two finite sequences of loops, we will say that they are equivalent if one can be obtained from the other by deleting and inserting null loops at various locations. Let $\ms$ denote the set of equivalence classes. Individual loops are also considered as members of $\ms$ by the natural inclusion. An element of $\ms$ will be called a `loop sequence'. The null loop sequence is the equivalence class of the null loop.  Any non-null loop sequence $s$ has a representative member $(l_1,\ldots, l_n)$ that has no null loops. This will be called the minimal representation of $s$. Often, we will write $s=(l_1,\ldots, l_n)$ even if $(l_1,\ldots, l_n)$ is not the minimal representation of $s$. 
 If $(l_1,\ldots, l_n)$ is the minimal representation of $s$, then the length of $s$ is defined as $|s| := |l_1|+\cdots +|l_n|$.  
The null loop sequence is defined to have length zero.

\subsection{A lattice string theory}
We will now define some operations on loops. These correspond to familiar operations on strings in the continuum setting. There are four kinds of operations in all, called `merger', `deformation', `splitting' and `twisting'. We begin by defining mergers. Let $l$ and $l'$ be two non-null loops. Let $x$ be a location in $l$ and $y$ be a location in $l'$. If $l$ contains an edge $e$ at location $x$ and $l'$  contains the same edge $e$ at location $y$, then write $l = aeb$  and $l' = ced$ (where $a,b,c,d$ are paths), and define the `positive merger' of $l$ and $l'$ at locations $x$ and $y$ as 
\[
l\oplus_{x,y} l' := [a e dceb]\,,
\]
where $[aedceb]$ is the nonbacktracking core of the cycle $aedceb$, and define the `negative merger' of $l$ and $l'$ at locations $x$ and $y$ as 
\[
l\ominus_{x,y} l' := [ac^{-1}d^{-1}b]\, .
\]
Figure \ref{posstfig} illustrates an example of positive merger. Negative merger is illustrated in Figure \ref{negstfig}. 

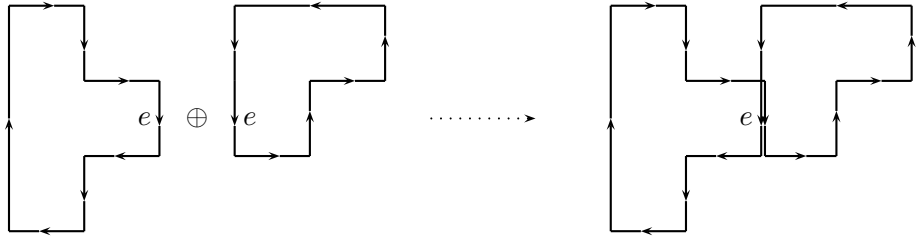
\begin{figure}[t]
%\begin{pdfpic}
\begin{pspicture}(1,1)(11,4.5)
\psset{xunit=1cm,yunit=1cm}
\psline{->}(1,4)(1.6, 4)
\psline{-}(1.6,4)(2,4)
\psline{->}(2,4)(2, 3.4)
\psline{-}(2,3.4)(2,3)
\psline{->}(2,3)(2.6, 3)
\psline{-}(2.6,3)(3,3)
\psline{->}(3,3)(3, 2.4)
\psline{-}(3, 2.4)(3,2)
\psline{->}(3,2)(2.4,2)
\psline{-}(2.4,2)(2,2)
\psline{->}(2,2)(2, 1.4)
\psline{-}(2,1.4)(2,1)
\psline{->}(2,1)(1.4,1)
\psline{-}(1.4,1)(1,1)
\psline{->}(1,1)(1,2.5)
\psline{-}(1,2.5)(1,4)
\psline{->}(4,3)(4,2.4)
\psline{-}(4,2.4)(4,2)
\psline{->}(4,2)(4.6,2)
\psline{-}(4.6,2)(5,2)
\psline{->}(5,2)(5,2.6)
\psline{-}(5,2.6)(5,3)
\psline{->}(5,3)(5.6,3)
\psline{-}(5.6,3)(6,3)
\psline{->}(6,3)(6, 3.6)
\psline{-}(6,3.6)(6,4)
\psline{->}(6,4)(5, 4)
\psline{-}(5,4)(4,4)
\psline{->}(4,4)(4,3.4)
\psline{-}(4,3.4)(4,3)
\psline[linestyle = dotted]{->}(6.5, 2.5)(8, 2.5)
\psline{->}(9,4)(9.6, 4)
\psline{-}(9.6,4)(10,4)
\psline{->}(10,4)(10, 3.4)
\psline{-}(10,3.4)(10,3)
\psline{->}(10,3)(10.6, 3)
\psline{-}(10.6,3)(11.05,3)
\psline{->}(11.05,3)(11.05, 2.4)
\psline{-}(11.05, 2.4)(11.05,2)
\psline{->}(11,2)(10.4,2)
\psline{-}(10.4,2)(10,2)
\psline{->}(10,2)(10, 1.4)
\psline{-}(10,1.4)(10,1)
\psline{->}(10,1)(9.4,1)
\psline{-}(9.4,1)(9,1)
\psline{->}(9,1)(9,2.5)
\psline{-}(9,2.5)(9,4)
\psline{->}(11,3)(11,2.4)
\psline{-}(11,2.4)(11,2)
\psline{->}(11.05,2)(11.6,2)
\psline{-}(11.6,2)(12,2)
\psline{->}(12,2)(12,2.6)
\psline{-}(12,2.6)(12,3)
\psline{->}(12,3)(12.6,3)
\psline{-}(12.6,3)(13,3)
\psline{->}(13,3)(13, 3.6)
\psline{-}(13,3.6)(13,4)
\psline{->}(13,4)(12, 4)
\psline{-}(12,4)(11,4)
\psline{->}(11,4)(11,3.4)
\psline{-}(11,3.4)(11,3)
\rput(3.5,2.5){$\oplus$}
\rput(2.8,2.5){$e$}
\rput(4.2,2.5){$e$}
\rput(10.8,2.5){$e$}
%\rput(5.5,2){\large $\longrightarrow$}
%%
%\psframe(1,1)(2,2)
%\rput(0,-.2){\small $x$}
%\rput(8.9,-.2){\small $y$}
%\psline{*-*}(0,0)(9,0)
%\pscurve{-}(0,0)(2,2)(4.5,-1.1)(7,.5)(9,0)
%\psline[linestyle=dashed]{->}(2,1.2)(2,2)
%\rput(2,1){\small $D(x,y)$}
%\psline[linestyle=dashed]{->}(2, .8)(2,0)
%\rput(4.5,-.7){\small $G(x,y)$}
\end{pspicture}
%\end{pdfpic}
\caption{Positive merger.}
\label{posstfig}
\end{figure}

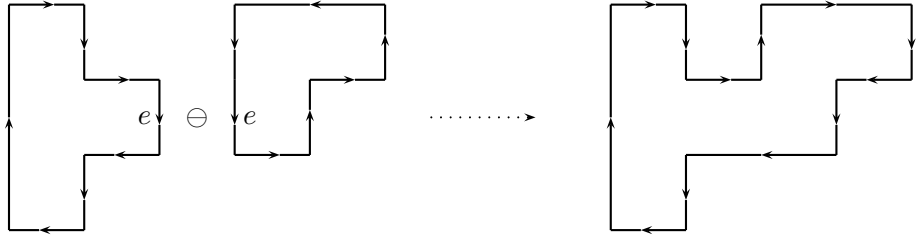
\begin{figure}[t]
%\begin{pdfpic}
\begin{pspicture}(1,1)(11,4.5)
\psset{xunit=1cm,yunit=1cm}
\psline{->}(1,4)(1.6, 4)
\psline{-}(1.6,4)(2,4)
\psline{->}(2,4)(2, 3.4)
\psline{-}(2,3.4)(2,3)
\psline{->}(2,3)(2.6, 3)
\psline{-}(2.6,3)(3,3)
\psline{->}(3,3)(3, 2.4)
\psline{-}(3, 2.4)(3,2)
\psline{->}(3,2)(2.4,2)
\psline{-}(2.4,2)(2,2)
\psline{->}(2,2)(2, 1.4)
\psline{-}(2,1.4)(2,1)
\psline{->}(2,1)(1.4,1)
\psline{-}(1.4,1)(1,1)
\psline{->}(1,1)(1,2.5)
\psline{-}(1,2.5)(1,4)
\psline{->}(4,3)(4,2.4)
\psline{-}(4,2.4)(4,2)
\psline{->}(4,2)(4.6,2)
\psline{-}(4.6,2)(5,2)
\psline{->}(5,2)(5,2.6)
\psline{-}(5,2.6)(5,3)
\psline{->}(5,3)(5.6,3)
\psline{-}(5.6,3)(6,3)
\psline{->}(6,3)(6, 3.6)
\psline{-}(6,3.6)(6,4)
\psline{->}(6,4)(5, 4)
\psline{-}(5,4)(4,4)
\psline{->}(4,4)(4,3.4)
\psline{-}(4,3.4)(4,3)
\psline[linestyle = dotted]{->}(6.5, 2.5)(8, 2.5)
\psline{->}(9,4)(9.6, 4)
\psline{-}(9.6,4)(10,4)
\psline{->}(10,4)(10, 3.4)
\psline{-}(10,3.4)(10,3)
\psline{->}(10,3)(10.6, 3)
\psline{-}(10.6,3)(11,3)
\psline{->}(10,2)(10, 1.4)
\psline{-}(10,1.4)(10,1)
\psline{->}(10,1)(9.4,1)
\psline{-}(9.4,1)(9,1)
\psline{->}(9,1)(9,2.5)
\psline{-}(9,2.5)(9,4)
\psline{->}(11,3)(11, 3.6)
\psline{-}(11,3.6)(11,4)
\psline{->}(11,4)(12,4)
\psline{-}(12,4)(13,4)
\psline{->}(13,4)(13, 3.4)
\psline{-}(13,3.4)(13,3)
\psline{->}(13,3)(12.4,3)
\psline{-}(12.4,3)(12,3)
\psline{->}(12,3)(12, 2.4)
\psline{-}(12,2.4)(12,2)
\psline{->}(12,2)(11,2)
\psline{-}(11,2)(10,2)
\rput(3.5,2.5){$\ominus$}
\rput(2.8,2.5){$e$}
\rput(4.2,2.5){$e$}
%\rput(5.5,2){\large $\longrightarrow$}
%%
%\psframe(1,1)(2,2)
%\rput(0,-.2){\small $x$}
%\rput(8.9,-.2){\small $y$}
%\psline{*-*}(0,0)(9,0)
%\pscurve{-}(0,0)(2,2)(4.5,-1.1)(7,.5)(9,0)
%\psline[linestyle=dashed]{->}(2,1.2)(2,2)
%\rput(2,1){\small $D(x,y)$}
%\psline[linestyle=dashed]{->}(2, .8)(2,0)
%\rput(4.5,-.7){\small $G(x,y)$}
\end{pspicture}
%\end{pdfpic}
\caption{Negative merger.}
\label{negstfig}
\end{figure}

If $l$ contains an edge $e$ at location $x$ and $l'$ contains the inverse edge $e^{-1}$ at location $y$, write $l = aeb$ and $l' = ce^{-1} d$ and define the positive merger of $l$ and $l'$ at locations $x$ and $y$ as
\[
l\oplus_{x,y} l' := [ae c^{-1} d^{-1} e b]\, ,
\]
and define the negative merger of $l$ and $l'$ at locations $x$ and $y$ as
\[
l\ominus_{x,y} l' := [adcb]\, ,
\]
It is easy to verify that in both cases the positive and negative mergers produce new loops. %If $l''$ is a loop produced by the positive stitching of two loops $l$ and $l'$ at two locations, then the length of $l''$ is the sum of the lengths of $l$ and $l'$. However, if $l''$ is produced by the negative stitching of $l$ and $l'$, then the length of $l''$ is 

If a loop $l'$ is produced by merging a plaquette with a loop $l$, we will say that $l'$ is a deformation of $l$. If the merger is positive we will say that $l'$ is a positive deformation, whereas if the merger is negative we will say that $l'$ is a negative deformation. Since a plaquette cannot contain an edge $e$ or it inverse at more than one locations, we will use the notations $l \oplus_x p$ and $l\ominus_x p$ to denote the loops obtained by merging $l$ and $p$ at locations $x$ and $y$, where $y$ is the unique location in $p$ where $e$ or $e^{-1}$ occurs, where $e$ is the edge occurring at location $x$ in $l$. Also, we will denote the set of positively oriented plaquettes containing an edge  $e$ or its inverse by $\cp^+(e)$. Figure \ref{posdefig} illustrates a positive deformation. A negative deformation is illustrated in Figure \ref{negdefig}.

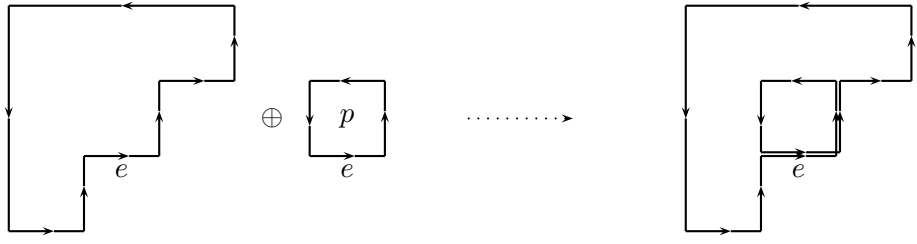
\begin{figure}[t]
%\begin{pdfpic}
\begin{pspicture}(1,1)(11,4.5)
\psset{xunit=1cm,yunit=1cm}
\psline{->}(0,1)(0.6,1)
\psline{-}(0.6,1)(1,1)
\psline{->}(1,1)(1,1.6)
\psline{-}(1,1.6)(1,2)
\psline{->}(1,2)(1.6,2)
\psline{-}(1.6,2)(2,2)
\psline{->}(2,2)(2,2.6)
\psline{-}(2,2.6)(2,3)
\psline{->}(2,3)(2.6,3)
\psline{-}(2.6,3)(3,3)
\psline{->}(3,3)(3,3.6)
\psline{-}(3,3.6)(3,4)
\psline{->}(3,4)(1.5, 4)
\psline{-}(1.5,4)(0,4)
\psline{->}(0,4)(0,2.5)
\psline{-}(0,2.5)(0,1)
\psline{->}(4,2)(4.6,2)
\psline{-}(4.6,2)(5,2)
\psline{->}(5,2)(5,2.6)
\psline{-}(5,2.6)(5,3)
\psline{->}(5,3)(4.4,3)
\psline{-}(4.4,3)(4,3)
\psline{->}(4,3)(4,2.4)
\psline{-}(4,2.4)(4,2)
\psline[linestyle = dotted]{->}(6, 2.5)(7.5, 2.5)
\psline{->}(9,1)(9.6,1)
\psline{-}(9.6,1)(10,1)
\psline{->}(10,1)(10,1.6)
\psline{-}(10,1.6)(10,2)
\psline{->}(10,2)(10.6,2)
\psline{-}(10.6,2)(11,2)
\psline{->}(11,2)(11,2.6)
\psline{-}(11,2.6)(11,3)
\psline{->}(11,3)(10.4,3)
\psline{-}(10.4,3)(10,3)
\psline{->}(10,3)(10,2.4)
\psline{-}(10,2.4)(10,2.05)
\psline{->}(10,2.05)(10.6, 2.05)
\psline{-}(10.6,2.05)(11.05, 2.05)
\psline{->}(11.05, 2.05)(11.05, 2.6)
\psline{-}(11.05, 2.6)(11.05,3)
\psline{->}(11.05,3)(11.6,3)
\psline{-}(11.6,3)(12,3)
\psline{->}(12,3)(12,3.6)
\psline{-}(12,3.6)(12,4)
\psline{->}(12,4)(10.5, 4)
\psline{-}(10.5,4)(9,4)
\psline{->}(9,4)(9,2.5)
\psline{-}(9,2.5)(9,1)
\rput(10.5, 1.8){$e$}
\rput(1.5, 1.8){$e$}
\rput(4.5, 1.8){$e$}
\rput(4.5, 2.5){$p$}
\rput(3.5, 2.5){$\oplus$}
%\rput(5.5,2){\large $\longrightarrow$}
%%
%\psframe(1,1)(2,2)
%\rput(0,-.2){\small $x$}
%\rput(8.9,-.2){\small $y$}
%\psline{*-*}(0,0)(9,0)
%\pscurve{-}(0,0)(2,2)(4.5,-1.1)(7,.5)(9,0)
%\psline[linestyle=dashed]{->}(2,1.2)(2,2)
%\rput(2,1){\small $D(x,y)$}
%\psline[linestyle=dashed]{->}(2, .8)(2,0)
%\rput(4.5,-.7){\small $G(x,y)$}
\end{pspicture}
%\end{pdfpic}
\caption{Positive deformation.}
\label{posdefig}
\end{figure}

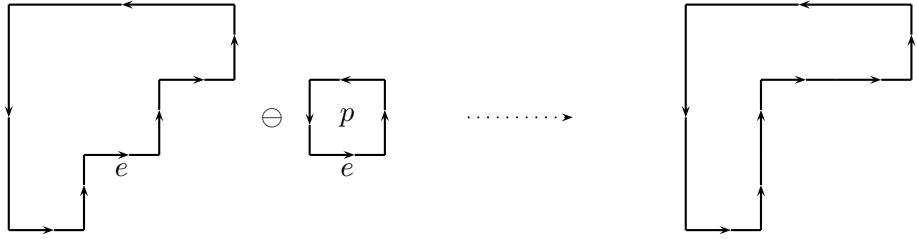
\begin{figure}[t]
%\begin{pdfpic}
\begin{pspicture}(1,1)(11,4.5)
\psset{xunit=1cm,yunit=1cm}
\psline{->}(0,1)(0.6,1)
\psline{-}(0.6,1)(1,1)
\psline{->}(1,1)(1,1.6)
\psline{-}(1,1.6)(1,2)
\psline{->}(1,2)(1.6,2)
\psline{-}(1.6,2)(2,2)
\psline{->}(2,2)(2,2.6)
\psline{-}(2,2.6)(2,3)
\psline{->}(2,3)(2.6,3)
\psline{-}(2.6,3)(3,3)
\psline{->}(3,3)(3,3.6)
\psline{-}(3,3.6)(3,4)
\psline{->}(3,4)(1.5, 4)
\psline{-}(1.5,4)(0,4)
\psline{->}(0,4)(0,2.5)
\psline{-}(0,2.5)(0,1)
\psline{->}(4,2)(4.6,2)
\psline{-}(4.6,2)(5,2)
\psline{->}(5,2)(5,2.6)
\psline{-}(5,2.6)(5,3)
\psline{->}(5,3)(4.4,3)
\psline{-}(4.4,3)(4,3)
\psline{->}(4,3)(4,2.4)
\psline{-}(4,2.4)(4,2)
\psline[linestyle = dotted]{->}(6, 2.5)(7.5, 2.5)
\psline{->}(9,1)(9.6,1)
\psline{-}(9.6,1)(10,1)
\psline{->}(10,1)(10,1.6)
\psline{-}(10,1.6)(10,2)
\psline{->}(10,2)(10,2.6)
\psline{-}(10,2.6)(10,3)
\psline{->}(10,3)(10.6,3)
\psline{-}(10.6,3)(11,3)
\psline{->}(11,3)(11.6,3)
\psline{-}(11.6,3)(12,3)
\psline{->}(12,3)(12,3.6)
\psline{-}(12,3.6)(12,4)
\psline{->}(12,4)(10.5, 4)
\psline{-}(10.5,4)(9,4)
\psline{->}(9,4)(9,2.5)
\psline{-}(9,2.5)(9,1)
\rput(1.5, 1.8){$e$}
\rput(4.5, 1.8){$e$}
\rput(4.5, 2.5){$p$}
\rput(3.5, 2.5){$\ominus$}
%\rput(10.5, 1.8){$e$}
%\rput(5.5,2){\large $\longrightarrow$}
%%
%\psframe(1,1)(2,2)
%\rput(0,-.2){\small $x$}
%\rput(8.9,-.2){\small $y$}
%\psline{*-*}(0,0)(9,0)
%\pscurve{-}(0,0)(2,2)(4.5,-1.1)(7,.5)(9,0)
%\psline[linestyle=dashed]{->}(2,1.2)(2,2)
%\rput(2,1){\small $D(x,y)$}
%\psline[linestyle=dashed]{->}(2, .8)(2,0)
%\rput(4.5,-.7){\small $G(x,y)$}
\end{pspicture}
%\end{pdfpic}
\caption{Negative deformation.}
\label{negdefig}
\end{figure}

Next, let $l$ be a non-null loop and let $x,y$ be distinct locations in $l$. If $l$ contains the same edge $e$ at $x$ and $y$, write $l = aebec$ and define the `positive splitting' of $l$ at $x$ and $y$ to be the pair of loops 
\[
\times^1_{x,y} l := [aec]\, , \ \ \times^2_{x,y} l := [be]\, .
\]
If $l$ contains $e$ at location $x$ and $e^{-1}$ at location $y$, write $l = aeb e^{-1} c$ and define the negative splitting of $l$ at $x$ and $y$ to be the pair of loops
\[
\times^1_{x,y} l := [ac]\, , \ \ \times^2_{x,y} l := [b]\, .
\]
%If $x>y$, let $\times^1_{x,y}l := \times^2_{y,x} l$ and $\times^2_{x,y} l := \times^1_{y,x} l$ in both of the above cases. 
Figure \ref{posspfig} illustrates an example of positive splitting, and Figure \ref{negspfig} illustrates negative splitting.

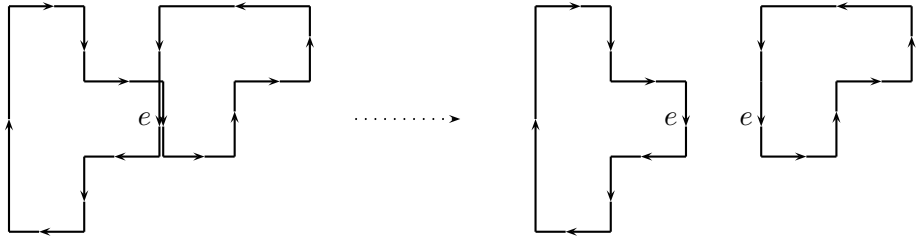
\begin{figure}[t]
%\begin{pdfpic}
\begin{pspicture}(1,1)(11,4.5)
\psset{xunit=1cm,yunit=1cm}
\psline{->}(8,4)(8.6, 4)
\psline{-}(8.6,4)(9,4)
\psline{->}(9,4)(9, 3.4)
\psline{-}(9,3.4)(9,3)
\psline{->}(9,3)(9.6, 3)
\psline{-}(9.6,3)(10,3)
\psline{->}(10,3)(10, 2.4)
\psline{-}(10, 2.4)(10,2)
\psline{->}(10,2)(9.4,2)
\psline{-}(9.4,2)(9,2)
\psline{->}(9,2)(9, 1.4)
\psline{-}(9,1.4)(9,1)
\psline{->}(9,1)(8.4,1)
\psline{-}(8.4,1)(8,1)
\psline{->}(8,1)(8,2.5)
\psline{-}(8,2.5)(8,4)
\psline{->}(11,3)(11,2.4)
\psline{-}(11,2.4)(11,2)
\psline{->}(11,2)(11.6,2)
\psline{-}(11.6,2)(12,2)
\psline{->}(12,2)(12,2.6)
\psline{-}(12,2.6)(12,3)
\psline{->}(12,3)(12.6,3)
\psline{-}(12.6,3)(13,3)
\psline{->}(13,3)(13, 3.6)
\psline{-}(13,3.6)(13,4)
\psline{->}(13,4)(12, 4)
\psline{-}(12,4)(11,4)
\psline{->}(11,4)(11,3.4)
\psline{-}(11,3.4)(11,3)
\psline[linestyle = dotted]{->}(5.5, 2.5)(7, 2.5)
\psline{->}(1,4)(1.6, 4)
\psline{-}(1.6,4)(2,4)
\psline{->}(2,4)(2, 3.4)
\psline{-}(2,3.4)(2,3)
\psline{->}(2,3)(2.6, 3)
\psline{-}(2.6,3)(3.05,3)
\psline{->}(3.05,3)(3.05, 2.4)
\psline{-}(3.05, 2.4)(3.05,2)
\psline{->}(3,2)(2.4,2)
\psline{-}(2.4,2)(2,2)
\psline{->}(2,2)(2, 1.4)
\psline{-}(2,1.4)(2,1)
\psline{->}(2,1)(1.4,1)
\psline{-}(1.4,1)(1,1)
\psline{->}(1,1)(1,2.5)
\psline{-}(1,2.5)(1,4)
\psline{->}(3,3)(3,2.4)
\psline{-}(3,2.4)(3,2)
\psline{->}(3.05,2)(3.6,2)
\psline{-}(3.6,2)(4,2)
\psline{->}(4,2)(4,2.6)
\psline{-}(4,2.6)(4,3)
\psline{->}(4,3)(4.6,3)
\psline{-}(4.6,3)(5,3)
\psline{->}(5,3)(5, 3.6)
\psline{-}(5,3.6)(5,4)
\psline{->}(5,4)(4, 4)
\psline{-}(4,4)(3,4)
\psline{->}(3,4)(3,3.4)
\psline{-}(3,3.4)(3,3)
\rput(2.8,2.5){$e$}
\rput(9.8,2.5){$e$}
\rput(10.8,2.5){$e$}
%\rput(5.5,2){\large $\longrightarrow$}
%%
%\psframe(1,1)(2,2)
%\rput(0,-.2){\small $x$}
%\rput(8.9,-.2){\small $y$}
%\psline{*-*}(0,0)(9,0)
%\pscurve{-}(0,0)(2,2)(4.5,-1.1)(7,.5)(9,0)
%\psline[linestyle=dashed]{->}(2,1.2)(2,2)
%\rput(2,1){\small $D(x,y)$}
%\psline[linestyle=dashed]{->}(2, .8)(2,0)
%\rput(4.5,-.7){\small $G(x,y)$}
\end{pspicture}
%\end{pdfpic}
\caption{Positive splitting.}
\label{posspfig}
\end{figure}

\begin{figure}[t]
%\begin{pdfpic}
\begin{pspicture}(1,1)(11,4.5)
\psset{xunit=1cm,yunit=1cm}
\psline{->}(1,1)(1.6,1)
\psline{-}(1.6,1)(2,1)
\psline{->}(2,1)(2,1.6)
\psline{-}(2,1.6)(2,2)
\psline{->}(2,2)(2.6,2)
\psline{-}(2.6,2)(3,2)
\psline{->}(3,2)(3.6,2)
\psline{-}(3.6,2)(4,2)
\psline{->}(4,2)(4,2.6)
\psline{-}(4,2.6)(4,3)
\psline{->}(4,3)(4.6,3)
\psline{-}(4.6,3)(5,3)
\psline{->}(5,3)(5,3.6)
\psline{-}(5,3.6)(5,4)
\psline{->}(5,4)(4,4)
\psline{-}(4,4)(3,4)
\psline{->}(3,4)(3,3)
\psline{-}(3,3)(3,2.05)
\psline{->}(3,2.05)(2.4, 2.05)
\psline{-}(2.4,2.05)(2,2.05)
\psline{->}(2,2.05)(2, 2.6)
\psline{-}(2,2.6)(2,3)
\psline{->}(2,3)(1,3)
\psline{-}(1,3)(0,3)
\psline{->}(0,3)(0,2.4)
\psline{-}(0,2.4)(0,2)
\psline{->}(0,2)(.6,2)
\psline{-}(.6,2)(1,2)
\psline{->}(1,2)(1,1.4)
\psline{-}(1,1.4)(1,1)
\psline[linestyle = dotted]{->}(5.5, 2.5)(7, 2.5)
\psline{->}(9,1)(9.6,1)
\psline{-}(9.6,1)(10,1)
\psline{->}(10,1)(10,1.6)
\psline{-}(10,1.6)(10,2)
\psline{->}(11,2)(11.6,2)
\psline{-}(11.6,2)(12,2)
\psline{->}(12,2)(12,2.6)
\psline{-}(12,2.6)(12,3)
\psline{->}(12,3)(12.6,3)
\psline{-}(12.6,3)(13,3)
\psline{->}(13,3)(13,3.6)
\psline{-}(13,3.6)(13,4)
\psline{->}(13,4)(12,4)
\psline{-}(12,4)(11,4)
\psline{->}(11,4)(11,3)
\psline{-}(11,3)(11,2)
\psline{->}(10,2)(10, 2.6)
\psline{-}(10,2.6)(10,3)
\psline{->}(10,3)(9,3)
\psline{-}(9,3)(8,3)
\psline{->}(8,3)(8,2.4)
\psline{-}(8,2.4)(8,2)
\psline{->}(8,2)(8.6,2)
\psline{-}(8.6,2)(9,2)
\psline{->}(9,2)(9,1.4)
\psline{-}(9,1.4)(9,1)
\rput(2.4, 1.8){$e$}
\rput(2.5, 2.3){$e^{-1}$}
%\rput(5.5,2){\large $\longrightarrow$}
%%
%\psframe(1,1)(2,2)
%\rput(0,-.2){\small $x$}
%\rput(8.9,-.2){\small $y$}
%\psline{*-*}(0,0)(9,0)
%\pscurve{-}(0,0)(2,2)(4.5,-1.1)(7,.5)(9,0)
%\psline[linestyle=dashed]{->}(2,1.2)(2,2)
%\rput(2,1){\small $D(x,y)$}
%\psline[linestyle=dashed]{->}(2, .8)(2,0)
%\rput(4.5,-.7){\small $G(x,y)$}
\end{pspicture}
%\end{pdfpic}
\caption{Negative splitting.}
\label{negspfig}
\end{figure}
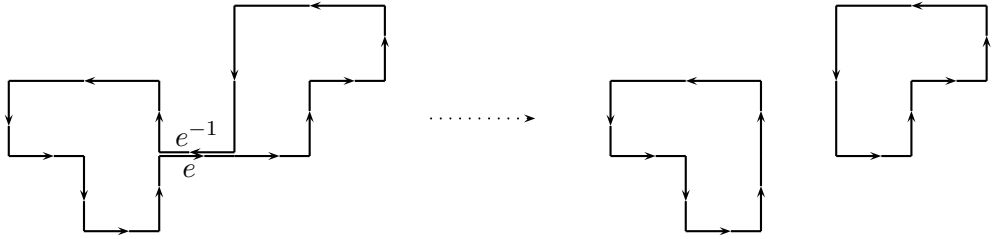

The final loop operation that we will define is `twisting'. Let $l$ be a non-null loop and $x,y$ be two locations in $l$. If $l$ contains an edge $e$ at both $x$ and $y$, write $l= aebec$ and define the negative twisting of $l$ between $x$ and $y$ as the loop
\[
\propto_{x,y} l := [ab^{-1} c]\, .
\]
If $l$ contains an edge $e$ at location $x$ and $e^{-1}$ at location $y$, write $l = aebe^{-1}c$, and define the positive twisting of $l$ between $x$ and $y$ as  the loop
\[
\propto_{x,y} l := [aeb^{-1}e^{-1} c]\, .
\]
It is easy to verify that these are indeed loops. %If $x>y$, let $\propto_{x,y} l :=\, \propto_{y,x} l$ in both of the above cases. 
Positive twisting is illustrated in Figure \ref{postwfig} and negative twisting is illustrated in Figure \ref{negtwfig}.

\begin{figure}[t]
%\begin{pdfpic}
\begin{pspicture}(1,1)(11,4.5)
\psset{xunit=1cm,yunit=1cm}
\psline{->}(1,1)(1.6,1)
\psline{-}(1.6,1)(2,1)
\psline{->}(2,1)(2,1.6)
\psline{-}(2,1.6)(2,2)
\psline{->}(2,2)(2.6,2)
\psline{-}(2.6,2)(3,2)
\psline{->}(3,2)(3.6,2)
\psline{-}(3.6,2)(4,2)
\psline{->}(4,2)(4,2.6)
\psline{-}(4,2.6)(4,3)
\psline{->}(4,3)(4.6,3)
\psline{-}(4.6,3)(5,3)
\psline{->}(5,3)(5,3.6)
\psline{-}(5,3.6)(5,4)
\psline{->}(5,4)(4,4)
\psline{-}(4,4)(3,4)
\psline{->}(3,4)(3,3)
\psline{-}(3,3)(3,2.05)
\psline{->}(3,2.05)(2.4, 2.05)
\psline{-}(2.4,2.05)(2,2.05)
\psline{->}(2,2.05)(2, 2.6)
\psline{-}(2,2.6)(2,3)
\psline{->}(2,3)(1,3)
\psline{-}(1,3)(0,3)
\psline{->}(0,3)(0,2.4)
\psline{-}(0,2.4)(0,2)
\psline{->}(0,2)(.6,2)
\psline{-}(.6,2)(1,2)
\psline{->}(1,2)(1,1.4)
\psline{-}(1,1.4)(1,1)
\psline[linestyle = dotted]{->}(5.5, 2.5)(7, 2.5)
\psline{->}(9,1)(9.6,1)
\psline{-}(9.6,1)(10,1)
\psline{->}(10,1)(10,1.6)
\psline{-}(10,1.6)(10,2)
\psline{->}(10,2)(10.6,2)
\psline{-}(10.6,2)(11,2)
\psline{->}(11,2)(11,3)
\psline{-}(11,3)(11,4)
\psline{->}(11,4)(12,4)
\psline{-}(12,4)(13,4)
\psline{->}(13,4)(13,3.4)
\psline{-}(13,3.4)(13,3)
\psline{->}(13,3)(12.4,3)
\psline{-}(12.4,3)(12,3)
\psline{->}(12,3)(12,2.4)
\psline{-}(12,2.4)(12,2.05)
\psline{->}(12,2.05)(11.4, 2.05)
\psline{-}(11.4,2.05)(11,2.05)
\psline{->}(11,2.05)(10.4, 2.05)
\psline{-}(10.4,2.05)(10,2.05)
\psline{->}(10,2.05)(10, 2.6)
\psline{-}(10,2.6)(10,3)
\psline{->}(10,3)(9,3)
\psline{-}(9,3)(8,3)
\psline{->}(8,3)(8,2.4)
\psline{-}(8,2.4)(8,2)
\psline{->}(8,2)(8.6,2)
\psline{-}(8.6,2)(9,2)
\psline{->}(9,2)(9,1.4)
\psline{-}(9,1.4)(9,1)
\rput(2.4, 1.8){$e$}
\rput(2.5, 2.3){$e^{-1}$}
\rput(10.4, 1.8){$e$}
\rput(10.5, 2.3){$e^{-1}$}
%\rput(5.5,2){\large $\longrightarrow$}
%%
%\psframe(1,1)(2,2)
%\rput(0,-.2){\small $x$}
%\rput(8.9,-.2){\small $y$}
%\psline{*-*}(0,0)(9,0)
%\pscurve{-}(0,0)(2,2)(4.5,-1.1)(7,.5)(9,0)
%\psline[linestyle=dashed]{->}(2,1.2)(2,2)
%\rput(2,1){\small $D(x,y)$}
%\psline[linestyle=dashed]{->}(2, .8)(2,0)
%\rput(4.5,-.7){\small $G(x,y)$}
\end{pspicture}
%\end{pdfpic}
\caption{Positive twisting.}
\label{postwfig}
\end{figure}
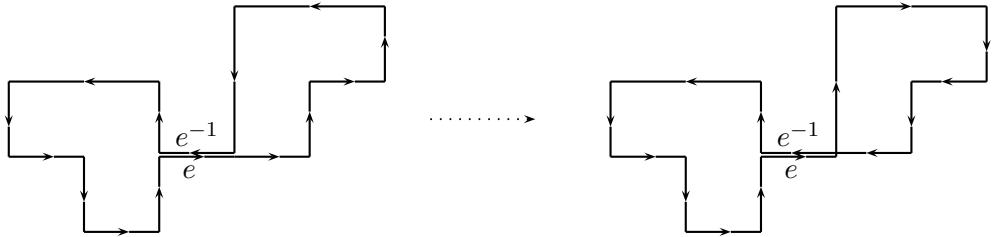

\begin{figure}[t]
%\begin{pdfpic}
\begin{pspicture}(1,1)(11,4.5)
\psset{xunit=1cm,yunit=1cm}
\psline{->}(1,4)(1.6, 4)
\psline{-}(1.6,4)(2,4)
\psline{->}(2,4)(2, 3.4)
\psline{-}(2,3.4)(2,3)
\psline{->}(2,3)(2.6, 3)
\psline{-}(2.6,3)(3.05,3)
\psline{->}(3.05,3)(3.05, 2.4)
\psline{-}(3.05, 2.4)(3.05,2)
\psline{->}(3,2)(2.4,2)
\psline{-}(2.4,2)(2,2)
\psline{->}(2,2)(2, 1.4)
\psline{-}(2,1.4)(2,1)
\psline{->}(2,1)(1.4,1)
\psline{-}(1.4,1)(1,1)
\psline{->}(1,1)(1,2.5)
\psline{-}(1,2.5)(1,4)
\psline{->}(3,3)(3,2.4)
\psline{-}(3,2.4)(3,2)
\psline{->}(3.05,2)(3.6,2)
\psline{-}(3.6,2)(4,2)
\psline{->}(4,2)(4,2.6)
\psline{-}(4,2.6)(4,3)
\psline{->}(4,3)(4.6,3)
\psline{-}(4.6,3)(5,3)
\psline{->}(5,3)(5, 3.6)
\psline{-}(5,3.6)(5,4)
\psline{->}(5,4)(4, 4)
\psline{-}(4,4)(3,4)
\psline{->}(3,4)(3,3.4)
\psline{-}(3,3.4)(3,3)
\psline[linestyle = dotted]{->}(5.5, 2.5)(7, 2.5)
\psline{->}(8,4)(8.6, 4)
\psline{-}(8.6,4)(9,4)
\psline{->}(9,4)(9, 3.4)
\psline{-}(9,3.4)(9,3)
\psline{->}(9,3)(9.6, 3)
\psline{-}(9.6,3)(10,3)
\psline{->}(9,2)(9, 1.4)
\psline{-}(9,1.4)(9,1)
\psline{->}(9,1)(8.4,1)
\psline{-}(8.4,1)(8,1)
\psline{->}(8,1)(8,2.5)
\psline{-}(8,2.5)(8,4)
\psline{->}(10,3)(10, 3.6)
\psline{-}(10,3.6)(10,4)
\psline{->}(10,4)(11,4)
\psline{-}(11,4)(12,4)
\psline{->}(12,4)(12, 3.4)
\psline{-}(12,3.4)(12,3)
\psline{->}(12,3)(11.4,3)
\psline{-}(11.4,3)(11,3)
\psline{->}(11,3)(11, 2.4)
\psline{-}(11,2.4)(11,2)
\psline{->}(11,2)(10,2)
\psline{-}(10,2)(9,2)
\rput(2.8,2.5){$e$}
%\rput(5.5,2){\large $\longrightarrow$}
%%
%\psframe(1,1)(2,2)
%\rput(0,-.2){\small $x$}
%\rput(8.9,-.2){\small $y$}
%\psline{*-*}(0,0)(9,0)
%\pscurve{-}(0,0)(2,2)(4.5,-1.1)(7,.5)(9,0)
%\psline[linestyle=dashed]{->}(2,1.2)(2,2)
%\rput(2,1){\small $D(x,y)$}
%\psline[linestyle=dashed]{->}(2, .8)(2,0)
%\rput(4.5,-.7){\small $G(x,y)$}
\end{pspicture}
%\end{pdfpic}
\caption{Negative twisting.}
\label{negtwfig}
\end{figure}
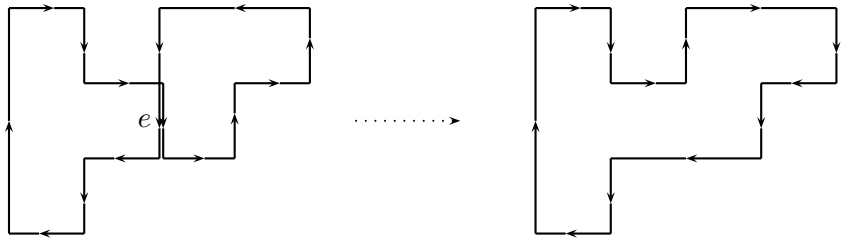

If $s$ and $s'$ are loop sequences, say that $s'$ is a splitting of $s$ if $s'$ is obtained by splitting one of the component loops in the minimal representation of $s$. Similarly, say that $s'$ is a deformation or twisting of $s$ if $s'$ is obtained by deforming or twisting one of the component loops of $s$, and say that $s'$ is a merger of $s$ if $s'$ is obtained by merging two of the component loops of $s$. Let
\begin{align*}
\fd^+(s) &:= \{s': \text{ $s'$ is a positive deformation of $s$}\}\, ,\\
\fd^-(s) &:= \{s': \text{ $s'$ is a negative deformation of $s$}\}\, ,\\
\fs^+(s) &:= \{s': \text{ $s'$ is a positive splitting of $s$}\}\, ,\\
\fs^-(s) &:= \{s': \text{ $s'$ is a negative splitting of $s$}\}\, ,\\
\fst^+(s) &:= \{s': \text{ $s'$ is a positive merger of $s$}\}\, ,\\
\fst^-(s) &:= \{s': \text{ $s'$ is a negative merger of $s$}\}\, ,\\
\ftw^+(s) &:= \{s': \text{ $s'$ is a positive twisting of $s$}\}\, ,\\
\ftw^-(s) &:= \{s': \text{ $s'$ is a negative twisting of $s$}\}\, .
\end{align*}
In the above definitions, it is important to clarify how the operations are counted. No clarification is necessary for counting deformations, but clarifications are needed for the other three kinds of operations. If a loop can be split positively at locations $x$ and $y$, then it can also be split positively at $y$ and $x$, producing the same pair of loops but in reverse order. Since the order of loops in a loop sequence is important, these two splittings are identified as distinct elements of $\fs^+(s)$. Similar remarks apply for negative splittings, twistings and mergers. For example, while counting mergers, one should be careful about the following. Let $s = (l_1,\ldots,l_n)$ be a loop sequence. Suppose that $l_1$ and $l_r$ can be positively merged at locations $x$ and $y$. Then $(l_1\oplus_{x,y} l_r,l_2,\ldots, l_{r-1}, l_{r+1},\ldots, l_n)$ is the sequence obtained by performing this merging operation. However, in this situation, $l_r$ and $l_1$ can be positively merged at locations $y$ and $x$, producing the loop sequence $(l_2,\ldots, l_{r-1}, l_r\oplus_{y,x} l_1,l_{r+1},\ldots, l_n)$. Although the loops $l_1\oplus_{x,y} l_r$ and $l_r\oplus_{y,x} l_1$ are the same, the two operations mentioned above are counted as distinct elements of $\fst^+(s)$, since the order of loops in a loop sequence is important.  A similar remark applies for negative mergers.

Define a `trajectory' to be a sequence $(s_0,s_1,\ldots)$ of loop sequences such that each $s_{i+1}$ is either a deformation or a splitting of $s_i$. Note that mergers and twistings are not allowed. A  trajectory may be finite or infinite. A `vanishing  trajectory' is a finite trajectory whose last element is the null loop sequence. Given a loop sequence $s$, define $\mx(s)$ to be the set of all vanishing  trajectories that begin at $s$. For an integer $k\ge 0$, let $\mx_k(s)$ be the set of vanishing trajectories that start at $s$ and have exactly $k$ deformations. 

A trajectory of loop sequences in $\zz^d$ naturally traces out a  surface in $\rr^{d+1}$. This is analogous to the Riemann surfaces traced out by strings in string theories. Figure \ref{trajfig} illustrates a vanishing trajectory of loop sequences in $\zz^2$ that starts from a single loop, then splits into two loops, which ultimately both vanish. Figure \ref{surffig} illustrates the surface traced out in $\rr^3$ by the trajectory from Figure \ref{trajfig}.

\begin{figure}[t]
\begin{pspicture}(3.5,-2.5)(13.5,4.5)
\psset{xunit=1cm,yunit=1cm}
\psline{->}(1,4)(2.6,4)
\psline{-}(2.6,4)(4,4)
\psline{->}(4,4)(4,3.4)
\psline{-}(4,3.4)(4,3)
\psline{->}(4,3)(3,3)
\psline{-}(3,3)(2,3)
\psline{->}(2,3)(2,2.4)
\psline{-}(2,2.4)(2,2)
\psline{->}(2,2)(1.4,2)
\psline{-}(1.4,2)(1,2)
\psline{->}(1,2)(1,3)
\psline{-}(1,3)(1,4)
\psline[linestyle = dotted]{->}(5,3)(6,3)
\psline{->}(7,4)(8.6,4)
\psline{-}(8.6,4)(10,4)
\psline{->}(10,4)(10,3.4)
\psline{-}(10,3.4)(10,3)
\psline{->}(10,3)(9.4,3)
\psline{-}(9.4,3)(9,3)
\psline{->}(9,3)(9,3.6)
\psline{-}(9,3.6)(9,3.95)
\psline{->}(9,3.95)(8.4,3.95)
\psline{-}(8.4, 3.95)(8,3.95)
\psline{->}(8, 3.95)(8,3)
\psline{-}(8,3)(8,2)
\psline{->}(8,2)(7.4,2)
\psline{-}(7.4,2)(7,2)
\psline{->}(7,2)(7,3)
\psline{-}(7,3)(7,4)
\psline[linestyle = dotted]{->}(11,3)(12,3)
\psline{->}(13,4)(13.6,4)
\psline{-}(13.6,4)(14,4)
\psline{->}(15,4)(15.6,4)
\psline{-}(15.6,4)(16,4)
\psline{->}(16,4)(16,3.4)
\psline{-}(16,3.4)(16,3)
\psline{->}(16,3)(15.4,3)
\psline{-}(15.4,3)(15,3)
\psline{->}(15,3)(15,3.6)
\psline{-}(15,3.6)(15,4)
\psline{->}(14, 4)(14,3)
\psline{-}(14,3)(14,2)
\psline{->}(14,2)(13.4,2)
\psline{-}(13.4,2)(13,2)
\psline{->}(13,2)(13,3)
\psline{-}(13,3)(13,4)
\psline[linestyle = dotted]{->}(14.5, 1.5)(14.5, .5)
\psline{->}(13, -1)(13.6, -1)
\psline{-}(13.6,-1)(14,-1)
\psline{->}(14,-1)(14, -1.6)
\psline{-}(14,-1.6)(14,-2)
\psline{->}(14,-2)(13.4,-2)
\psline{-}(13.4,-2)(13,-2)
\psline{->}(13,-2)(13, -1.4)
\psline{-}(13,-1.4)(13,-1)
\psline{->}(15,0)(15.6,0)
\psline{-}(15.6,0)(16,0)
\psline{->}(16,0)(16, -.6)
\psline{-}(16, -.6)(16, -1)
\psline{->}(16, -1)(15.4,-1)
\psline{-}(15.4,-1)(15,-1)
\psline{->}(15, -1)(15, -.4)
\psline{-}(15, -.4)(15, 0)
\psline[linestyle = dotted]{->}(12, -1)(11, -1)
\psline{->}(9,0)(9.6,0)
\psline{-}(9.6,0)(10,0)
\psline{->}(10,0)(10, -.6)
\psline{-}(10, -.6)(10, -1)
\psline{->}(10, -1)(9.4,-1)
\psline{-}(9.4,-1)(9,-1)
\psline{->}(9, -1)(9, -.4)
\psline{-}(9, -.4)(9, 0)
\psline[linestyle = dotted]{->}(6, -1)(5, -1)
\end{pspicture}
\caption{A vanishing trajectory of loop sequences. Note that although the loop sequences in this trajectory decrease in size at each step, this need not always be the case.}
\label{trajfig}
\end{figure}
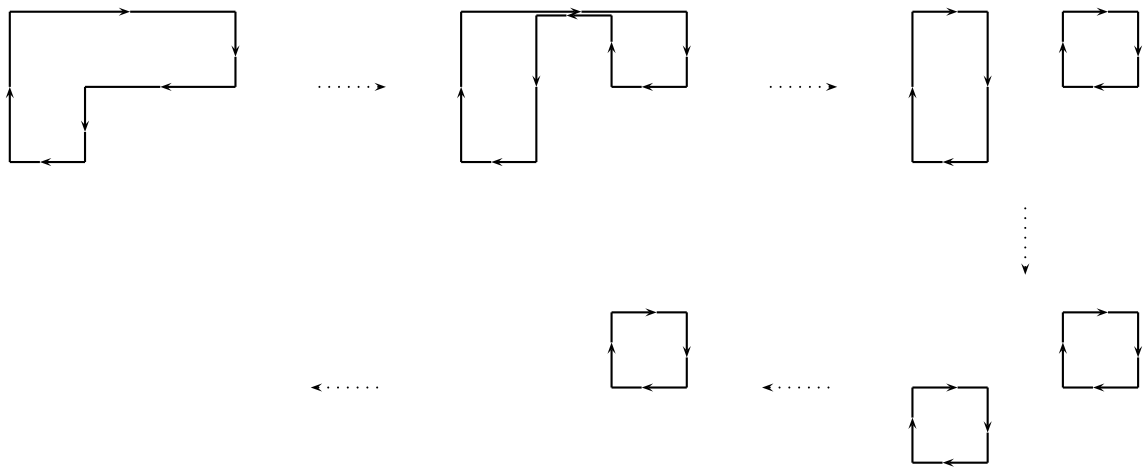

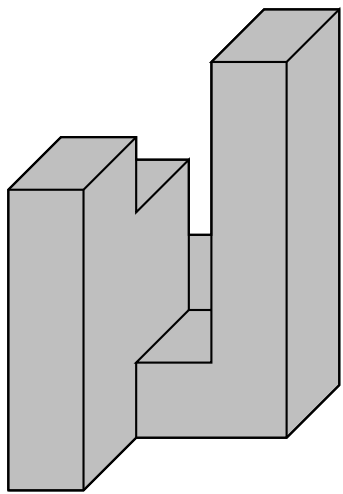
\begin{figure}[t]
%\begin{pdfpic}
\begin{pspicture}(-1,1)(9,7.5)
\psset{xunit=1cm,yunit=1cm}
\pspolygon[fillstyle=solid, fillcolor=lightgray](1,1)(1,5)(1.7,5.7)(2.7,5.7)(2.7,5.4)(3.4,5.4)(3.4,4.4)(3.7,4.4)(3.7,6.7)(4.4,7.4)(5.4,7.4)(5.4,2.4)(4.7,1.7)(2.7,1.7)(2,1)(1,1)
\psline{-}(1,1)(2,1)(2.7,1.7)(4.7,1.7)(5.4,2.4)(5.4,7.4)(4.7,6.7)(4.7,1.7)
\psline{-}(5.4,7.4)(4.4,7.4)(3.7,6.7)(4.7,6.7)
\psline{-}(3.7,6.7)(3.7,2.7)(2.7,2.7)(3.4,3.4)(3.7,3.4)
\psline{-}(1,1)(1,5)(1.7,5.7)(2.7,5.7)(2.7,4.7)(3.4,5.4)(3.4,3.4)
\psline{-}(3.4,4.4)(3.7,4.4)
\psline{-}(3.4,5.4)(2.7,5.4)
\psline{-}(2,1)(2,5)(2.7,5.7)
\psline{-}(2,5)(1,5)
\psline{-}(2.7,1.7)(2.7,2.7)
\end{pspicture}
\caption{Surface in $\rr^3$ traced out by the trajectory from Figure \ref{trajfig}. Time moves in the upward direction. Cross sections of this surface  give the loop sequences depicted in Figure \ref{trajfig}, with changes occurring at integer time points.}
\label{surffig}
\end{figure}

Let $\beta$ be a real number, that will occur as the `inverse coupling strength' of the lattice gauge theories that we will investigate later. If $s'$ is a deformation of a loop sequence $s$, define the `weight' of the transition from $s$ to $s'$ at inverse coupling strength $\beta$ as 
\[
w_\beta(s,s') := 
\begin{cases}
-\beta/|s| &\text{ if $s'$ is a positive deformation of $s$,}\\
\beta/|s| &\text{ if $s'$ is a negative deformation of $s$.}
\end{cases}
\]
If $s'$ is a splitting of $s$, define the weight of the transition from $s$ to $s'$ at inverse coupling strength $\beta$ as
\[
w_\beta(s,s') := 
\begin{cases}
-1/|s| &\text{ if $s'$ is a positive splitting of $s$,}\\
1/|s| &\text{ if $s'$ is a negative splitting of $s$.}
\end{cases}
\]
Note that in this case the weight does not actually depend on $\beta$.

Finally, if $X = (s_0,s_1,\ldots, s_n)$ is a vanishing   trajectory, define the weight of $X$ at inverse coupling strength $\beta$ as the product $w_\beta(X) := w_\beta(s_0, s_1)w_\beta(s_1, s_2)\cdots w_\beta(s_{n-1}, s_n)$. Note that the weight of a trajectory may be positive or negative. 

The trajectories defined above can be thought of as trajectories of strings in a discrete string theory, which are allowed to deform and split according to a given set of rules. The weight of a trajectory can be thought of as the exponential of the action in this string theory.

We will sometimes need a slight variation of the weights defined above, obtained by eliminating $\beta$ from the definitions. Define these `$\beta$-free weights' as 
\[
v(s,s') :=
\begin{cases}
\beta^{-1}w_\beta(s,s') &\text{ if $s'$ is a deformation of $s$,}\\
w_\beta(s,s') &\text{ if $s'$ is a splitting of $s$.} 
\end{cases}
\]
Note that in both cases $v(s,s')$ is independent of $\beta$. Define the $\beta$-free weight of a trajectory $X = (s_0,s_1,\ldots,s_n)$ as 
\[
v(X) := v(s_0,s_1)v(s_1,s_2)\cdots v(s_{n-1}, s_n)\, .  
\]
If $\delta(X)$ denotes the number of deformation steps in $X$, then $w_\beta(X)$ and $v(X)$ are related by the  relation $w_\beta(X) = v(X) \beta^{\delta(X)}$.

\section{Main result and corollaries}\label{results}
Let $SO(N)$ be the group of $N\times N$ orthogonal matrices with determinant $1$, and let $\sigma_N$ be the Haar measure on $SO(N)$. Let $\Lambda$ be a finite subset of $\zz^d$ and let $\beta$ be a real number. Let $\mu_{\Lambda, N, \beta}$ be a probability measure on the set of all collections $Q = (Q_e)_{e\in E_\Lambda^+}$ of $SO(N)$ matrices, defined as 
\begin{equation}\label{measure}
d\mu_{\Lambda, N, \beta}(Q) := Z_{\Lambda, N,\beta}^{-1}\exp\biggl(N\beta \sum_{p\in \cp^+_\Lambda} \tr(Q_p)\biggr) \prod_{e\in E^+_\Lambda} d\sigma_N(Q_e)\, ,
\end{equation}
where $Q_p := Q_{e_1}Q_{e_2}Q_{e_3}Q_{e_4}$ for a plaquette $p=e_1e_2e_3e_4$ (where $Q_{e^{-1}} := Q_e^{-1}$ for $e\in E_\Lambda^+$), and $Z_{\Lambda,  N, \beta}$ is the normalizing constant. This probability measure describes what is known as `lattice gauge theory' on $\Lambda$ for the gauge group $SO(N)$. The number $\beta$ is the inverse of the `coupling constant' of the theory. Lattice gauge theories can be defined similarly for other matrix groups such as $SU(N)$ and $U(N)$, where the trace is replaced by the real part of the trace if the group contains complex matrices.

If $f = f(Q)$ is a  function of the configuration $Q = (Q_e)_{e\in E^+_\Lambda}$, the expected value $\smallavg{f}_{\Lambda, N, \beta}$ of $f$ in the above lattice gauge theory is defined as the integral of $f$ with respect to the measure $\mu_{\Lambda, N, \beta}$, that is,
\begin{equation}\label{expdef}
\smallavg{f}_{\Lambda, N, \beta} := \int f(Q) d\mu_{\Lambda, N, \beta}(Q)\, ,
\end{equation}
where the integration is over the space of all configurations. When $\Lambda$, $N$ and $\beta$ are obvious from the context, we may omit the subscripts and simply write $\smallavg{f}$.

%Lattice gauge theories were introduced by Wilson \cite{wilson74} in an attempt to discretize quantum gauge theories, which are the building blocks of the standard model of quantum mechanics. Quantum gauge theories have not yet been made mathematically precise; indeed, this challenging task is the `existence' part of the Yang--Mill existence and mass gap problem that has been posed as one of the millennium problems by the Clay Institute \cite{jaffewitten}. Lattice gauge theories, on the other hand, are very clearly defined as mathematical objects as can be seen from the definition stated above. 

The main objects of interest in lattice gauge theories are the so-called Wilson loop variables and their expected values. Given a loop $l=e_1e_2\cdots e_n$, the Wilson loop variable $W_l$ is defined as
\[
W_l := \tr(Q_{e_1}Q_{e_2}\cdots Q_{e_n})\, ,
\]
and its expected value $\smallavg{W_l}_{\Lambda, N, \beta}$ is defined as in \eqref{expdef}, provided that the edges $e_1,\ldots, e_n$ all belong to $E_\Lambda$. (As mentioned before, if $e\in E_\Lambda^-$ then $Q_e$ is defined as $Q_{e^{-1}}^{-1}$.) By definition, $W_\emptyset = N$, where $\emptyset$ is the null loop. This is consistent with the convention that the empty product of $N\times N$ matrices is the identity matrix of order $N$.

The following theorem is the main result of this article. It gives a duality between expected values of Wilson loop variables in $SO(N)$ lattice gauge theory and certain sums over trajectories in the discrete string theory defined in Section \ref{not1}. 
\begin{thm}[Main result: Solution of $SO(N)$ lattice gauge theory and proof of gauge-string duality in the 't Hooft limit]\label{mainthm}
There exists a number $\beta_0(d) >0$, depending only on the dimension $d$,  such that the following is true. Let $\Lambda_1,\Lambda_2,\ldots$ be any sequence of finite subsets of $\zz^d$ such that $\Lambda_1\subseteq \Lambda_2\subseteq  \cdots$ and 
$\zz^d = \cup_{N=1}^\infty \Lambda_N$. If $|\beta| \le\beta_0(d)$, then for any loop sequence $s$ with minimal representation $(l_1,\ldots, l_n)$,
\[
\lim_{N\ra\infty}\frac{\smallavg{W_{l_1}W_{l_2}\cdots W_{l_n}}_{\Lambda_N, N, \beta}}{N^n} = \sum_{X\in \mx(s)} w_\beta(X)\, 
\]
where $\mx(s)$ is the set of all vanishing trajectories starting at $s$ and $w_\beta(X)$ is the weight of a  trajectory $X$ as defined in Section \ref{not1}. Moreover, the infinite sum is absolutely convergent. 
\end{thm}
%The representation in terms of trajectories of lattice strings is the main contribution of Theorem~\ref{mainthm}. One may say that this is more explicit than the connection with strings given by planar diagrams. 
An alternative (and perhaps more natural) way of formulating Theorem \ref{mainthm} would be to first take $\Lambda_i\uparrow \zz^d$, with $N$ fixed, to get an infinite volume limit of $SO(N)$ lattice gauge theory on $\zz^d$, and then send $N\to\infty$. A small problem with this approach is that we do not know whether the infinite volume limit is unique, even for small enough $\beta$, since the 't Hooft coupling places $N\beta$ instead of $\beta$ in front of the Hamiltonian in~\eqref{measure}. Let $\smallavg{\cdot}_{N,\beta}$ denote expectation with respect to some chosen infinite volume limit of $SO(N)$ lattice gauge theory on $\zz^d$. Since $\Lambda_N$ is allowed to vary arbitrarily with $N$ in Theorem~\ref{mainthm}, it is not difficult to deduce that indeed,
\[
\lim_{N\ra\infty}\frac{\smallavg{W_{l_1}W_{l_2}\cdots W_{l_n}}_{N, \beta}}{N^n} = \sum_{X\in \mx(s)} w_\beta(X)\, , 
\]
irrespective of how  the infinite volume limits are chosen.

Theorem \ref{mainthm} has a number of interesting corollaries. The first corollary proves the `factorization property' of Wilson loops in the large $N$ limit. This fact is widely believed to be true  in the theoretical physics community and has been the basis of many theoretical calculations (for example, \cite{bhanotetal82, eguchikawai82, gonzalez83, makeenkomigdal79}), but was lacking a rigorous proof until now. 
\begin{cor}[Factorization of Wilson loops]\label{factor}
Let all notation be as in Theorem \ref{mainthm}, and suppose that $|\beta| \le \beta_0(d)$. Then for any non-null loops $l_1,\ldots, l_n$,
\[
\lim_{N\ra\infty} \frac{\smallavg{W_{l_1}W_{l_2}\cdots W_{l_n}}_{\Lambda_N, N, \beta}}{N^n}  = \lim_{N\ra\infty} \prod_{i=1}^n \frac{\smallavg{W_{l_i}}_{\Lambda_N, N, \beta}}{N}\, .
\]
In particular, for any loop $l$, $\smallavg{(W_l/N)^2}_{\Lambda_N, N, \beta}$ has the same limit as $\smallavg{W_l/N}_{\Lambda_N, N, \beta}^2$ when $N\ra\infty$ and $|\beta|\le \beta_0(d)$, implying that the random variable $W_l/N$ converges in probability to the (deterministic) limit of $\smallavg{W_l}/N$ given in Theorem~\ref{mainthm}.
\end{cor}
A second corollary that follows from Theorem \ref{mainthm} is the so-called `Wilson area law upper bound' in the large $N$ limit. A lattice gauge theory is said to obey the area law upper bound if 
\begin{align}\label{arealaw}
\smallavg{W_l}\le C_1e^{-C_2 \textup{area}(l)}\, ,
\end{align}
where $C_1$ and $C_2$ are positive constants that depend on the  coupling strength and the gauge group, and $\textup{area}(l)$ is the area enclosed by the loop $l$. Similarly, the theory is said to satisfy the area law lower bound if the inequality is valid in the opposite direction (with different constants). The upper bound is more important because of its connection with quark confinement \cite{wilson74}.

If $l$ lies on a coordinate plane, there is usually no ambiguity about the meaning of `area enclosed by $l$'. However for general lattice loops, one has to define what this term means. A natural definition is that this is the area of the minimal lattice surface enclosed by $l$, where  `lattice surface' needs to be appropriately defined.  Wilson's motivation for studying the area law upper bound in \cite{wilson74} was to understand the phenomenon of quark confinement from a theoretical perspective. Wilson gave an argument supporting the claim that the area law upper bound holds at strong coupling under fairly general conditions. For groups with nontrivial center such as $SU(N)$ and $U(N)$, this was rigorously verified for planar loops by Osterwalder and Seiler~\cite{osterwalderseiler78}, and extended to general loops by Seiler~\cite{seiler82}. For groups with trivial center, such as $SO(3)$, it is believed that the area law does not hold, even at strong coupling~\cite{gl81}. However, it turns out that in the 't Hooft limit of $SO(N)$ lattice gauge theory at strong coupling, the area law upper bound holds. This is our next corollary, stated below. %It is a consequence of the string theoretic representation of limiting Wilson loop expectations given in Theorem \ref{mainthm}.

Before stating the corollary, we need to have a definition of `area enclosed by a loop'. The most natural way to do this is through the language of algebraic topology, or more precisely, through the language of cell complexes. We do not need to deal with the full definition of the standard cell complex in $\zz^d$, but only with $1$-chains and $2$-chains of the complex. In the notation used in this paper, $1$-chains are elements of the free $\zz$-module over $E^+$, and $2$-chains are elements of the free $\zz$-module over $\cp^+$. In algebraic topology, $2$-chains are viewed as the natural algebraic objects corresponding to surfaces. Following that convention, we define a `lattice surface' to be simply a $2$-chain in the standard cell complex of $\zz^d$. 

Any $p\in\cp^+$ can be written uniquely as $e_1e_2e_3^{-1}e_4^{-1}$ where $e_1,e_2,e_3,e_4\in E^+$. There is a standard homomorphism $\delta$ from the module of $2$-chains into the module of $1$-chains that takes $p$ to $e_1+e_2-e_3-e_4$. This is known as the `differential map'. If $x$ is a $2$-chain, then $\delta(x)$ is called its `boundary' in cell complex terminology.  

Given an oriented edge $e\in E$, define a $1$-chain $r(e)$ as 
\[
r(e) := 
\begin{cases}
e &\text{ if } e\in E^+,\\
-e^{-1} &\text{ if } e\in E^-.
\end{cases}
\]
If $\rho=e_1e_2\cdots e_n$ is a path, define $r(\rho) := r(e_1)+\cdots+r(e_n)$. Note that if  $\rho$ and $\rho'$ are cyclically equivalent closed paths, then $r(\rho)=r(\rho')$. Thus there is no ambiguity in defining $r(l)$ for a loop $l$. A loop will be called `non-canceling' if there is no edge $e$ in the loop such that $e^{-1}$ is also in the loop (which means that there are no cancelations when computing $r(l)$).

We will say that a loop $l$ is the boundary of a lattice surface $x$ if $\delta(x) = r(l)$. Note that we are referring to both $l$ and $r(l)$ as the boundary of $x$, in different contexts. Define the area of a lattice surface $x = \sum_{p\in \cp^+} n_p p$ as 
\[
\textup{area}(x) := \sum_{p\in \cp^+} |n_p|\, .
\]
Finally, define the area of the minimal lattice surface enclosed by $l$ to be the minimum of the areas of all lattice surfaces with boundary $l$. Denote this by $\textup{area}(l)$. It follows from standard facts about the cell complex of $\zz^d$ that for any loop $l$ there exists at least one $2$-chain $x$ such that $\delta(x)=r(l)$, and therefore $\textup{area}(l)$ is well-defined.
\begin{cor}[Area law upper bound in the 't Hooft limit]\label{area}
Let all notation be as in Theorem~\ref{mainthm}. Then for any non-null, non-canceling loop $l$, 
\[
\lim_{N\ra\infty} \frac{|\smallavg{W_l}|}{N} \le (C(d)|\beta|)^{\textup{area}(l)}\, ,
\]
where $C(d)$ is a positive constant that depends only on the dimension $d$, and $\textup{area}(l)$ is the area of the minimal lattice surface enclosed by $l$, as defined above. 
\end{cor}
Note that although Corollary \ref{area} looks slightly different than the description of the area law upper bound given in  \eqref{arealaw}, it is actually the same, as can be seen by choosing $C_2 = \max\{0, -\log (C(d)|\beta|)\}$. 

The key idea in the proof of Corollary \ref{area} is that a vanishing trajectory starting from a loop $l$ must have at least area$(l)$ deformations, implying that the weight of the trajectory must accumulate a product of at least area$(l)$ $\beta$'s. A natural question is whether a similar argument can give the area law lower bound as well. This seems to be more difficult to prove, because of the possibility that the trajectory weights may cancel each other out. For finite $N$, a general area law lower bound for rectangles was given by Seiler \cite{seiler78} using reflection positivity. It is not clear whether Seiler's result extends to the $N\ra\infty$ limit. Whether a general lower bound can be proved as a corollary of Theorem \ref{mainthm} is left as an open problem in this paper (see Section \ref{opensec}). 

The next corollary gives a formula for the limiting partition function.
\begin{cor}[Limiting partition function of $SO(N)$ lattice gauge theory]\label{part}
Let all notation be as in Theorem \ref{mainthm}. Let $M_N$ be a sequence of integers increasing to infinity, and suppose that $\Lambda_N = [-M_N, M_N]^d\cap \zz^d$. Let $p$ be any plaquette. Then
\[
\lim_{N\ra\infty} \frac{\log Z_{\Lambda_N, N, \beta}}{N^2|\Lambda_N|} = \frac{\beta d(d-1)}{2} \sum_{X\in \mx(p)} \frac{ w_\beta(X)}{\delta(X)+1}\, ,
\]
where $\delta(X)$ is the number of deformations in the trajectory $X$. 
\end{cor}
Another simple corollary of Theorem \ref{mainthm} shows that the limits of Wilson loop expectations have power series expansions in $\beta$ when $|\beta|$ is small.
\begin{cor}[Real analyticity at strong coupling]\label{series}
Let all notation be as in Theorem \ref{mainthm}, and assume that $|\beta|\le \beta_0(d)$. Recall the $\beta$-free weights $v(X)$ defined in Section \ref{not1}. For any non-null loop sequence $s$ with minimal representation $(l_1,\ldots, l_n)$,
\[
\lim_{N\ra\infty}\frac{\smallavg{W_{l_1}W_{l_2}\cdots W_{l_n}}_{\Lambda_N, N, \beta}}{N^n} = \sum_{k=0}^\infty a_k(s) \beta^k\,, 
\]
where
\[
a_k(s) = \sum_{X\in \mx_k(s)} v(X)\, 
\]
and the infinite series is absolutely convergent. The rescaled log-partition function has a similar representation:
\[
\lim_{N\ra\infty} \frac{\log Z_{\Lambda_N, N, \beta}}{N^2|\Lambda_N|} = \frac{ d(d-1)}{2} \sum_{k=0}^\infty \frac{a_k(s)\beta^{k+1}}{k+1}\, ,
\]
where, again, the series is absolutely convergent. 
\end{cor}
The main ingredient in the proof of Theorem \ref{mainthm} is a recursive equation for Wilson loop expectations in $SO(N)$ lattice gauge theories. Since this may be of independent interest, it is presented below as a theorem. Such equations are often called `master loop equations'. In the context of lattice gauge theories, they first appeared in the work of Makeenko and Migdal \cite{makeenkomigdal79}, and are therefore sometimes called `Makeenko--Migdal equations'. The main difference between the master loop equations appearing in the literature \cite{bhanotetal82, collinsetal09, eguchikawai82, makeenkomigdal79} and the following theorem is that our result is true for finite $N$, while the equations derived previously are valid only in the limit $N\ra\infty$. 
\begin{thm}[Finite $N$ master loop equation]\label{mastersymm}
Fix a nonempty finite set $\Lambda\subseteq \zz^d$, an integer $N\ge 2$ and a real number $\beta$. Let $\smallavg{\cdot}$ denote expectation with respect to the $SO(N)$ lattice gauge theory on $\Lambda$ at inverse coupling strength $\beta$. For any non-null loop sequence $s$ with minimal representation $(l_1,\ldots,l_n)$ such that each $l_i$ is contained in $\Lambda$, define
\[
\phi(s) := \frac{\smallavg{W_{l_1}W_{l_2}\cdots W_{l_n}}}{N^n}\, .
\]
Recall the sets $\ftw^+(s)$, $\ftw^-(s)$, $\fs^+(s)$, $\fs^-(s)$, $\fst^+(s)$, $\fst^-(s)$, $\fd^+(s)$ and $\fd^-(s)$ defined in Section~ \ref{not1}. Let $s$ be as above, and suppose that all vertices that are at distance $\le 1$ from any $l_i$ belong to $\Lambda$. Then $\phi$ satisfies the recursive equation
\begin{align*}
(N-1)|s|\phi(s) &= \sum_{s'\in \ftw^-(s)} \phi(s') - \sum_{s'\in \ftw^+(s)} \phi(s') + N\sum_{s'\in \fs^-(s)}\phi(s') - N\sum_{s'\in \fs^+(s)}\phi(s')\\
&\quad + \frac{1}{N} \sum_{s'\in \fst^-(s)}\phi(s') - \frac{1}{N} \sum_{s'\in \fst^+(s)}\phi(s') + N\beta \sum_{s'\in \fd^-(s)}\phi(s') - N\beta \sum_{s'\in \fd^+(s)}\phi(s')\, .
\end{align*}
\end{thm}
The rest of the paper is organized as follows. An algorithm for computing the coefficients of the power series expansion of Corollary \ref{series} is presented in Section~\ref{algosec}. Section~\ref{sketch} contains a sketch of the proof of Theorem~\ref{mainthm}. The proof itself is carried out in Sections~\ref{stein} through~\ref{mainthmsec}. Theorem~\ref{mastersymm} is proved in Section~\ref{mm1sec}.  Corollaries~\ref{factor}, \ref{area} and~\ref{part} are proved in Sections~\ref{factorsec}, \ref{areasec}, \ref{partsec} and~\ref{seriessec} respectively. Lemma~\ref{core} is proved in Section~\ref{coresec}. The paper ends with a list of open problems in Section~\ref{opensec}. 

\section{Algorithmic aspects}\label{algosec}
Recall the coefficients $a_k(s)$ of the power series expansion from Corollary \ref{series}. For practical purposes, it may be interesting to have an implementable algorithm for computing these coefficients. We present one such recursive algorithm below. The algorithm inputs a loop sequence $s$ with minimal representation $(l_1,\ldots, l_n)$ and a nonnegative integer $k$, and outputs $a_k(s)$. If $s$ is the null loop sequence, then it outputs $a_0(s)=1$ and $a_k(s)= 0$ for every $k>0$. If $s$ is non-null, then it outputs $a_0(s)=0$. In all other cases, it proceeds as follows. Let $e$ be an arbitrary edge in $l_1$. Let $A_1$ be the set of locations in $l_1$ where $e$ occurs, and let $B_1$ be the set of locations in $l_1$ where $e^{-1}$ occurs. Let $C_1 = A_1\cup B_1$ and $m$ be the size of $C_1$. Then $a_k(s)$ is computed using the following recursive formula:
\begin{align*}
a_k(s) &:= \frac{1}{m}\sum_{x\in A_1, \, y\in B_1} a_k(\times_{x,y}^1 l_1, \times_{x,y}^2 l_1,l_2,\ldots,l_n) + \frac{1}{m}\sum_{x\in B_1, \, y\in A_1} a_k(\times_{x,y}^1 l_1, \times_{x,y}^2 l_1,l_2,\ldots,l_n)\\
&\qquad - \frac{1}{m}\sum_{\substack{x,y\in A_1\\ x\ne y}} a_k(\times^1_{x,y} l_1,\times^2_{x,y} l_1,l_2,\ldots,l_n)- \frac{1}{m}\sum_{\substack{x,y\in B_1\\ x\ne y}} a_k(\times^1_{x,y} l_1,\times^2_{x,y} l_1,l_2,\ldots,l_n)\\
&\quad   + \frac{1}{m} \sum_{p\in \cp^+(e)}\sum_{x\in C_1}a_{k-1}(l_1 \ominus_{x}p,l_2,\ldots,l_n) - \frac{1}{m} \sum_{p\in \cp^+(e)}\sum_{x\in C_1}a_{k-1}(l_1 \oplus_{x}p,l_2,\ldots,l_n) \, .
\end{align*}
It is not obvious that the the recursion terminates and gives the same $a_k(s)$ as in Corollary \ref{series}. The following result states that it indeed does.
\begin{prop}\label{algoprop}
The recursion described above terminates for any $k$ and $s$, and gives the same $a_k(s)$ as in Corollary \ref{series}. 
\end{prop}
This proposition is proved at the end of Section \ref{power}. Although the algorithm is easy to implement using any standard programming language, the recursions can be time-consuming. It is possible that there exist cleverer recursions that terminate faster. 

As an application of the above algorithm, take $d=3$ and $s$ to be a single plaquette in $\zz^3$. The first six coefficients computed using the above algorithm (implemented on a standard laptop computer using a code written in the R programming language) turn out to be $a_0 = 0$, $a_1 = 1$, $a_2 = 0$, $a_3 = 0$, $a_4 = 0$ and $a_5 = -7$. In other words, if $p$ is a plaquette in $\zz^3$ and $\beta$ is small, then 
\[
\lim_{N\ra\infty}\frac{\smallavg{W_p}_{\Lambda_N, N,\beta}}{N} = \beta -7\beta^5 + O(\beta^6)\, ,
\]
where $\Lambda_N$ is a sequence of finite sets increasing to $\zz^3$. %On the other hand, in two dimension, the right-hand side turns out to be $\beta-\beta^5+82\beta^7/9 +\cdots$.

\section{Proof sketch}\label{sketch}
The purpose of this section is to outline the main steps in the proof of Theorem \ref{mainthm}. The details are worked out in Sections \ref{stein} through \ref{mainthmsec}.

The first step is to define a `Stein exchangeable pair'~\cite{stein72, stein86, stein95} of Haar-distributed random $SO(N)$ matrices, as follows. Let $Q$ be a Haar-distributed random $SO(N)$ matrix. Let $\ep\in (0,1)$ be a real number, and choose $(I,J)$ uniformly at random from the set $\{(i,j):1\le i\ne j\le N\}$. Let $\eta$ be a random variable that is $1$ with probability $1/2$, and $-1$ with probability $1/2$. Let $R_\ep$ be the $N\times N$ matrix whose $(i,j)^{\mathrm{th}}$ entry is 
\[
\begin{cases}
\sqrt{1-\ep^2} &\text{ if $i=j=I$ or $i=j=J$,}\\
\eta \ep &\text{ if $i=I$ and $j=J$,}\\
-\eta \ep &\text{ if $i=J$ and $j=I$,}\\
1 &\text{ if $i=j$ and $i\not \in \{I,J\}$,}\\
0 &\text{ in all other cases.}
\end{cases}
\]
It is easy to verify that $R_\ep$ is an element of $SO(N)$. Let $Q_\ep := R_\ep Q$. It turns out that $(Q,Q_\ep)$ is an exchangeable pair, which means that $(Q,Q_\ep)$ has the same joint distribution as $(Q_\ep, Q)$. A consequence of this exchangeability is the identity
\[
\ee((f(Q_\ep) - f(Q))g(Q)) = -\frac{1}{2} \ee((f(Q_\ep)-f(Q))(g(Q_\ep)-g(Q)))
\]
where $f$ and $g$ are arbitrary real-valued functions on $SO(N)$. 

Dividing both sides of the above identity by $\ep^2$ and sending $\ep$ to zero gives the following `Schwinger--Dyson equation for $SO(N)$': 
\begin{align*}
&\ee\biggl(\sum_{i,k} q_{ik} \fpar{f}{q_{ik}} g\biggr)\\
 &= \frac{1}{N-1}\ee\biggl(\sum_{i,k} \spar{f}{q_{ik}} g- \sum_{i,j,k,k'} q_{jk}q_{ik'} \mpar{f}{q_{ik}}{q_{jk'}} g + \sum_{i,k} \fpar{f}{q_{ik}}\fpar{g}{q_{ik}} - \sum_{i,j,k,k'} q_{jk}q_{ik'} \fpar{f}{q_{ik}} \fpar{g}{q_{jk'}}\biggr)\, ,
\end{align*}
where $q_{ij}$ denotes the $(i,j)^{\mathrm{th}}$ entry of $Q$, $\ee$ denotes expectation with respect to the Haar measure, and all indices run from $1$ to $N$. 

Let $\Lambda$ be a finite subset of $\zz^d$ and consider $SO(N)$ lattice gauge theory on $\Lambda$. Let $l_1,\ldots,l_n$ be non-null loops such that all vertices of $\zz^d$ that are at distance $\le 1$ from any of these loops are contained in $\Lambda$. In particular, the loops themselves are contained in $\Lambda$. Let $e$ be the first edge of~$l_1$. Let $q_{ij}^e$ denote the $(i,j)^{\mathrm{th}}$ entry of $Q_e$. Define two functions $f$ and $g$ as $f:= W_{l_1}$ and 
\[
g := Z_{\Lambda, N, \beta}^{-1}W_{l_2}W_{l_3}\cdots W_{l_n} \exp\biggl(N\beta\sum_{p\in \cp^+_\Lambda} W_p\biggr)\, .
\]
One property of Wilson loop variables that comes to our aid at this point is the identity
\[
\sum_{i,k} q_{ik}^e \fpar{W_{l_1}}{q_{ik}^e} = m W_{l_1}\, ,
\]
where $m$ is the number of locations in $l_1$ that contain either $e$ or $e^{-1}$. Using this identity, we get
\[
\ee\biggl(\sum_{i,k} q_{ik}^e \fpar{f}{q_{ik}^e}g\biggr) = \ee( mW_{l_1} g) =  m \smallavg{W_{l_1}W_{l_2}\cdots W_{l_n}}\, ,
\]
where, as before, $\ee$ is expectation with respect to Haar measure and $\smallavg{\cdot}$ is expectation with respect to the measure $\mu_{\Lambda, N,\beta}$ of lattice gauge theory. We are now in a setting where the previously derived Schwinger--Dyson equation for $SO(N)$ can be applied to the pair $(f,g)$. Luckily the right-hand side of the equation, after a lengthy sequence of computations,  emerges as a linear combination of expectations of products of Wilson loop variables. This gives rise to a `master loop equation' for $SO(N)$ lattice gauge theory (Theorem \ref{mastersymm}). 

Define, for a loop sequence $s = (l_1,\ldots, l_n)$, 
\[
\phi_{\Lambda_N, N, \beta}(s) := \frac{\smallavg{W_{l_1}W_{l_2}\cdots W_{l_n}}_{\Lambda_N, N, \beta}}{N^n}\, ,
\]
where $\Lambda_N$ is a sequence of finite sets increasing to $\zz^d$. Using the master loop equation and taking $N$ to infinity, one can then show that if $\phi_\beta$ is a limit point of $\phi_{\Lambda_N, N, \beta}$, then $\phi_\beta$ satisfies a `limiting master loop equation'. A main step in the proof, at this point, is to show that if $|\beta|$ is sufficiently small, then there is a unique $\phi_\beta$ that satisfies this master loop equation. %The proof of this uniqueness is a little too complicated and technical to outline in this section. 

The proof of uniqueness will be carried out as follows. Suppose that $\phi_\beta$ and $\psi_\beta$ are two functions that both satisfy the limiting master loop equation. Let $\Delta$ be the set of all finite sequences of integers. If $\delta,\delta'\in \Delta$, we will say that $\delta \le \delta'$ if the two sequences have the same length and $\delta'$ dominates $\delta$ in each component. If $s$ is a non-null loop sequence with minimal representation $(l_1,\ldots, l_n)$, let $\delta(s)\in \Delta$ be the sequence of length $n$ whose $i^{\mathrm{th}}$ component is $|l_i|$.  For $\delta\in \Delta$, let
\[
D(\delta):= \sup_{s:\delta(s)\le \delta} |\phi_\beta(s)-\psi_\beta(s)|\,,
\]
where the supremum is understood to be zero if there is no $s$ such that $\delta(s)\le \delta$.

For an element $\delta = (\delta_1,\ldots,\delta_n)\in \Delta$, let
\[
\iota(\delta) := \delta_1+\cdots +\delta_n - n\, .
\]
For each $\lambda \in (0,1)$, let
\[
F(\lambda) := \sum_{\delta\in\Delta} \lambda^{\iota(\delta)} D(\delta)\, .
\]
We will first prove that $F(\lambda)<\infty$ if $\lambda$ is sufficiently small. Next, we will spend a considerable amount of effort to prove the inequality
\begin{align*}
F(\lambda)&\le \biggl(4\lambda^3 + 4\lambda + \frac{4|\beta| d}{\lambda^4} + \frac{4|\beta| d}{1-\lambda}\biggr)F(\lambda)\, .
\end{align*}
This shows that if $|\beta|$ are small enough (depending on $\lambda$), then the coefficient of $F(\lambda)$ on the right is less than $1$. Due to the finiteness of $F(\lambda)$, this implies that $F(\lambda)=0$, and therefore $\phi_\beta=\psi_\beta$. Once we have proved the uniqueness  of $\phi_\beta$ when $|\beta|$ is small enough, a simple compactness argument shows that $\phi_{\Lambda_N, N, \beta}$ converges to $\phi_\beta$ as $N\ra\infty$.

The remainder of the proof is heavily inductive. The inductions will typically be on loop sequences. If we are dealing with single loops, it is natural to do induction on the length of the loop. However,  it is not immediately clear how to carry out induction on loop sequences. What will work for us is the following. If $s$ is a non-null loop sequence with minimal representation $(l_1,\ldots,l_n)$, define the `index' of $s$ as
\[
\iota(s) := |l_1|+\cdots +|l_n| - n\, .
\]
Since a non-null loop must have at least four edges, $\iota(s)$ is always a positive integer. This allows us to define functions of loop sequences and prove facts about loop sequences by induction on the index. The key result that helps us carry out the inductions is that if $s'$ is a loop sequence that is produced by splitting $s$, then $\iota(s')<\iota(s)$. 

The next step in the proof is to define a collection of coefficients $a_k(s)$, one  for each nonnegative integer $k$ and loop sequence $s$, using a certain inductive definition (by induction on $k$ and $\iota(s)$, as described above) that guarantees the following two properties: 
\begin{enumerate}[(a)]
\item For $|\beta|$ sufficiently small, the power series 
\begin{equation}\label{mainseries}
\psi_\beta(s) := \sum_{k=0}^\infty a_k(s) \beta^k
\end{equation}
converges absolutely for any $s$.
\item  The function $\psi_\beta$ satisfies the limiting master loop equation.
\end{enumerate}
These two properties and the uniqueness of the solution of the master loop equation imply that $\psi_\beta=\phi_\beta$. In other words, this identifies the power series expansion of $\phi_\beta$. 

For each $k\ge 0$ and loop sequence $s$, let $\mx_k(s)$ be the set of all vanishing trajectories that start at $s$ and have exactly $k$ deformations. We will show by induction that for any $k$ and $s$, 
\[
a_k(s)\beta^k = \sum_{X\in \mx_k(s)} w_\beta(X)\, .
\]
Once we have this, it follows that
\[
\phi_\beta(s) = \sum_{k=0}^\infty \sum_{X\in \mx_k(s)} w_\beta(X)\,.
\] 
The only thing that remains to be proved at this stage is that
\[
\sum_{X\in \mx(s)} |w_\beta(X)|<\infty\, .
\]
This will be shown as follows. We will inductively define a second set of coefficients $b_k(s)$ in such a way that for all $k$ and $s$, $b_k(s)\ge 0$  and
\[
b_k(s)|\beta|^k = \sum_{X\in \mx_k(s)} |w_\beta(X)|\, .
\]
We will then show, again by induction, that $b_k(s)$ grows at most exponentially in $k$, where the exponent does not depend on $s$. This will complete the proof.

Incidentally, the idea of computing asymptotic matrix integrals by first showing that the limit satisfies a Schwinger--Dyson equation and then solving this equation has been investigated previously, for example  by Guionnet and coauthors \cite{collinsetal09, guionnet04, guionnet06, guionnetetal12, guionnetmaida05, guionnetnovak14, guionnetzeitouni02}. Indeed, several steps in the proof of convergence of the infinite series \eqref{mainseries} are inspired by ideas contained in a paper of Collins, Guionnet and Maurel-Segala \cite{collinsetal09}. The main difference between the above papers and this one is that this paper deals with polynomials of a growing number of matrices, whereas the papers cited above deal with polynomials of a fixed number of matrices. The only exception is \cite{guionnetetal12}, where the finiteness of the number of matrices is replaced by an exponential decay of correlations and independence of matrix entries. The general technique of proving the limiting Schwinger--Dyson equations in the above papers relies on the finiteness of the number of matrices. It does not seem to generalize in any obvious way to the lattice gauge setting. The scheme of proving Schwinger--Dyson equations using Stein's method of exchangeable pairs  is a new technical contribution of this paper. Unlike previous techniques, this method does not actually require taking the matrix order $N$ to infinity, which raises the possibility that there may be something substantially different about this approach. It is this new kind of Schwinger--Dyson equation that leads to the explicit  representation in terms of string trajectories. % instead of the usual enumeration of planar diagrams. 

\section{Stein's exchangeable pair for $SO(N)$}\label{stein}
A pair of random variables $(U, U')$ is called an `exchangeable pair' if $(U,U')$ has the same probability law as $(U', U)$. Exchangeable pairs were introduced and effectively used by Charles Stein~\cite{stein72, stein86} to prove central limit theorems for sums of dependent random variables. Stein's paper gave birth to a flourishing subfield of probability theory, now called `Stein's method'.

In an unpublished work \cite{stein95}, Stein gave the following method for constructing an exchangeable pair of Haar-distributed random $SO(N)$ matrices. This was used in the paper \cite{chatterjeemeckes} and the thesis \cite{meckes} to prove a number of central limit theorems for $SO(N)$ and other matrix groups. 

Let $Q$ be a Haar-distributed random element of $SO(N)$. Take some $\ep>0$. Let $\eta$ be a random variable that is $1$ or $-1$ with equal probability. Pick a pair $(I,J)$ uniformly at random from the set $\{(i,j): 1\le i\ne j\le N\}$. Let $R_\ep = (r_{ij})_{1\le i,j\le N}$ be a random element of $SO(N)$, defined as follows: Let 
\[
r_{II} = r_{JJ} = \sqrt{1-\ep^2}\, , \ \ r_{IJ} = \eta \ep\, , \ \ r_{JI} = -\eta \ep\, ,
\]
and for all $k\ne I,J$ and $1\le k'\le N$, 
\[
r_{Ik} = r_{Jk} = 0\, , \ \text{ and } \ 
r_{kk'} = 
\begin{cases}
1 &\text{ if } k=k'\, ,\\
0 &\text{ if } k\ne k'\, .
\end{cases}
\]
It is easy to see that $R_\ep$ is indeed an $SO(N)$ matrix. 

Let $Q_\ep := R_\ep Q$. Since $Q$ is Haar-distributed, so is $Q_\ep$. Moreover, $(Q,Q_\ep)$ has the same distribution as $(Q_\ep, Q)$, by the following logic: Since $R_\ep^T$ is just $R_\ep$ with $\eta$ replaced by $-\eta$ (where $R_\ep^T$ is the transpose of the matrix $R_\ep$), therefore $R_\ep^T$ has the same law as $R_\ep$. Again, as observed above, $Q_\ep$ has the same law as $Q$. Not only that, the conditional distribution of $Q_\ep$ given $R_\ep$ is also the Haar distribution, irrespective of the value of $R_\ep$. Thus, $Q_\ep$ and $R_\ep$ are independent. Therefore, $(Q_\ep, R_\ep^T)$ has the same law as $(Q, R_\ep)$. Since $(Q,Q_\ep)= (Q, R_\ep Q)$ and $(Q_\ep, Q) = (Q_\ep, R_\ep^TQ_\ep)$, this shows that $(Q,Q_\ep)$ has the same law as $(Q_\ep, Q)$.

The exchangeability of $(Q,Q_\ep)$ will be used through the following lemma. Identities of this type are fairly common in papers on Stein's method; see  for example~\cite{chatterjee05, chatterjee07, stein72, stein86}. 
\begin{lmm}\label{exchlmm}
For any Borel measurable $f,g:SO(N)\ra \rr$,
\[
\ee((f(Q_\ep) - f(Q))g(Q)) = -\frac{1}{2} \ee((f(Q_\ep)-f(Q))(g(Q_\ep)-g(Q)))\, .
\]
\end{lmm}
\begin{proof}
Expanding the brackets on the right, we get
\begin{align*}
-\frac{1}{2}\ee(f(Q_\ep) g(Q_\ep) - f(Q) g(Q_\ep) - f(Q_\ep)g(Q)+f(Q)g(Q))\, .
\end{align*}
By exchangeability, 
\[
\ee (f(Q_\ep) g(Q_\ep)) = \ee (f(Q) g(Q))
\]
and 
\[
\ee(f(Q_\ep) g(Q)) = \ee(f(Q) g(Q_\ep))\, .
\]
The right-hand side therefore reduces to
\[
\ee(f(Q_\ep) g(Q) - f(Q)g(Q))\, ,
\]
which is exactly the same as the left-hand side. 
\end{proof}

\section{Generalized Schwinger--Dyson equation for $SO(N)$}\label{sd}
The main result of this section is the following theorem, which gives a kind of integration by parts formula for the $SO(N)$ group. Equations of this type are generally known as `Schwinger--Dyson equations' in the physics literature, and `Stein characterizing equations' in the probability and statistics literature. The difference between the usual forms of such equations and the version presented below is that our version involves two functions $f$ and $g$, whereas ordinarily Schwinger--Dyson equations involve one function only. Usually in Stein's method such equations characterize the underlying distribution (the Haar measure in this case), which should be the case here too. We will not pursue this point further. 
\begin{thm}\label{sdthm}
Let $\rr^{N\times N}$ be the space of $N\times N$ real matrices with the Euclidean topology. Let $f$ and $g$ be $C^2$ functions defined on an open subset of $\rr^{N\times N}$ that contains $SO(N)$. Let $Q= (q_{ij})_{1\le i,j\le N}$ be a Haar-distributed random element of $SO(N)$, and let $f$ and $g$ be shorthand notations for the random variables $f(Q)$ and $g(Q)$. 
Then
\begin{align*}
&\ee\biggl(\sum_{i,k} q_{ik} \fpar{f}{q_{ik}} g\biggr)\\
 &= \frac{1}{N-1}\ee\biggl(\sum_{i,k} \spar{f}{q_{ik}} g- \sum_{i,j,k,k'} q_{jk}q_{ik'} \mpar{f}{q_{ik}}{q_{jk'}} g + \sum_{i,k} \fpar{f}{q_{ik}}\fpar{g}{q_{ik}} - \sum_{i,j,k,k'} q_{jk}q_{ik'} \fpar{f}{q_{ik}} \fpar{g}{q_{jk'}}\biggr)\, ,
\end{align*}
where all indices run from $1$ to $N$. 
%where $\ee$ denotes integration with respect to the Haar measure.
\end{thm}
(Here I would like to emphasize that in the above statement, $SO(N)$ is viewed as a sub-manifold of the manifold of $N\times N$ matrices, and the derivatives are taken in this ambient manifold, and not in $SO(N)$. For example, if $f(Q)$ is the sum of squares of the elements of $Q$, then $\partial f/\partial q_{11} = 2q_{11}$.)
\begin{proof}
By the compactness of $SO(N)$, $f$, $g$ and their first- and second-order derivatives are uniformly bounded on $SO(N)$. This fact will be used several times without explicit mention in this proof.

The plan is to use Lemma \ref{exchlmm}. Let $Q_\ep$ and $\eta$ be as in Section \ref{stein}. 
Let $q_{ij}^\ep$ denote the $(i,j)^{\mathrm{th}}$ entry of $Q_\ep$. Then for  any $1\le k\le N$,
\begin{align*}
q^\ep_{Ik} &= \sqrt{1-\ep^2}\, q_{Ik} + \eta \ep q_{Jk}\, ,\\
q^\ep_{Jk} &= -\eta \ep q_{Ik} + \sqrt{1-\ep^2}\, q_{Jk}\, , \ \text{ and }\\
q^\ep_{k'k} &= q_{k'k} \ \text{ for all } k' \ne I,J\, .
\end{align*}
For each $i$ and $j$, let $\delta_{ij} := q^\ep_{ij} - q_{ij}$. Let $O(\ep^r)$ denote any quantity, random or nonrandom, whose absolute value is bounded by $C\ep^r$, where $C$ is a deterministic constant that does not depend on $\ep$ (but may depend on $N$, $f$ and $g$). Then note that for all $k$,
\begin{align*}
\delta_{Ik} &= -\frac{\ep^2}{2} q_{Ik} + \eta \ep q_{Jk} + O(\ep^3)\, ,\\
\delta_{Jk} &= -\eta \ep q_{Ik} -\frac{\ep^2}{2} q_{Jk} + O(\ep^3)\, , \ \text{ and }\\
\delta_{k'k} &= 0 \ \text{ for all } k' \ne I,J\, .
\end{align*}
Thus, by the boundedness of $f$ and derivatives on $SO(N)$, 
\begin{align}
f(Q_\ep) - f(Q) &= \sum_k \delta_{Ik} \fpar{f}{q_{Ik}} + \sum_k \delta_{Jk}\fpar{f}{q_{Jk}} +\frac{1}{2} \sum_{k,k'} \delta_{Ik}\delta_{Ik'} \mpar{f}{q_{Ik}}{q_{Ik'}}\nonumber \\
&\quad + \frac{1}{2} \sum_{k,k'} \delta_{Jk}\delta_{Jk'} \mpar{f}{q_{Jk}}{q_{Jk'}} +  \sum_{k,k'} \delta_{Ik}\delta_{Jk'} \mpar{f}{q_{Ik}}{q_{Jk'}} + O(\ep^3)\nonumber\\
&= \sum_k \biggl(\eta \ep q_{Jk} - \frac{\ep^2}{2} q_{Ik}\biggr) \fpar{f}{q_{Ik}} - \sum_k \biggl(\eta \ep q_{Ik} +\frac{\ep^2}{2} q_{Jk} \biggr) \fpar{f}{q_{Jk}} \nonumber\\ 
&\quad + \frac{\ep^2}{2}\sum_{k,k'} q_{Jk}q_{Jk'} \mpar{f}{q_{Ik}}{q_{Ik'}} + \frac{\ep^2}{2}\sum_{k,k'} q_{Ik}q_{Ik'} \mpar{f}{q_{Jk}}{q_{Jk'}}\nonumber\\
&\quad - \ep^2\sum_{k,k'} q_{Jk}q_{Ik'} \mpar{f}{q_{Ik}}{q_{Jk'}} + O(\ep^3)\, . \label{f1}
\end{align}
Now note that 
\[
\ee(\eta \mid Q, I, J) = 0
\]
and by the orthogonality of $Q$, 
\begin{equation}\label{eta1}
\ee(q_{Ik}q_{Ik'} \mid Q, J) = \frac{1}{N-1} \sum_{i\ne J} q_{ik}q_{ik'} =
\begin{cases}
-q_{Jk}q_{Jk'}/(N-1) &\text{ if } k\ne k'\, ,\\
\\
(1-q_{Jk}^2)/(N-1)&\text{ if } k= k'\, ,
\end{cases}
\end{equation}
and a similar expression holds with $I$ and $J$ interchanged. Using these identities in \eqref{f1}, we get
\begin{align*}
\ee((f(Q_\ep)-f(Q)) g(Q)) &= -\frac{\ep^2}{2} \sum_k \ee\biggl(q_{Ik}\fpar{f}{q_{Ik}} g\biggr)-\frac{\ep^2}{2} \sum_k \ee\biggl(q_{Jk}\fpar{f}{q_{Jk}} g\biggr)\\
&\quad + \frac{\ep^2}{2(N-1)} \sum_k \ee\biggl(\spar{f}{q_{Ik}}g\biggr) - \frac{\ep^2}{2(N-1)} \sum_{k,k'} \ee\biggl(q_{Ik}q_{Ik'} \mpar{f}{q_{Ik}}{q_{Ik'}}g\biggr)\\
&\quad + \frac{\ep^2}{2(N-1)} \sum_k \ee\biggl(\spar{f}{q_{Jk}}g\biggr) - \frac{\ep^2}{2(N-1)} \sum_{k,k'} \ee\biggl(q_{Jk}q_{Jk'} \mpar{f}{q_{Jk}}{q_{Jk'}}g\biggr)\\
&\quad - \ep^2\sum_{k,k'} \ee\biggl(q_{Jk}q_{Ik'} \mpar{f}{q_{Ik}}{q_{Jk'}}g\biggr) + O(\ep^3)\, .
\end{align*}
Since $(I,J)$ is uniformly distributed over all $(i,j)$ such that $i\ne j$, 
\begin{align*}
&\ee\biggl(q_{Jk}q_{Ik'} \mpar{f}{q_{Ik}}{q_{Jk'}}g\biggr) \\
&= \frac{1}{N(N-1)} \sum_{1\le i\ne j\le N} \ee\biggl(q_{jk}q_{ik'} \mpar{f}{q_{ik}}{q_{jk'}}g\biggr)\\
&= \frac{1}{N(N-1)} \sum_{i,j} \ee\biggl(q_{jk}q_{ik'} \mpar{f}{q_{ik}}{q_{jk'}}g\biggr) - \frac{1}{N(N-1)} \sum_i \ee\biggl(q_{ik}q_{ik'} \mpar{f}{q_{ik}}{q_{ik'}}g\biggr)\, .
\end{align*}
Tackling the other terms in a similar manner, we get
\begin{align*}
\ee((f(Q_\ep)-f(Q)) g(Q)) &= -\frac{\ep^2}{N}\sum_{i,k} \ee\biggl(q_{ik}\fpar{f}{q_{ik}} g\biggr)  + \frac{\ep^2}{N(N-1)} \sum_{i,k} \ee\biggl(\spar{f}{q_{ik}} g\biggr)\\
&\quad - \frac{\ep^2}{N(N-1)} \sum_{i,j,k,k'} \ee\biggl(q_{jk}q_{ik'}\mpar{f}{q_{ik}}{ q_{jk'}} g\biggr) + O(\ep^3)\, . 
\end{align*}
Next, note that by \eqref{f1}, the fact that $\eta^2 = 1$, the symmetry between $I$ and $J$, and equation \eqref{eta1}, 
\begin{align*}
&\frac{1}{2}\ee((f(Q_\ep)-f(Q))(g(Q_\ep)-g(Q))) \\
 &= \frac{\ep^2}{2}\ee\biggl(\biggl(\sum_k q_{Jk}\fpar{f}{q_{Ik}} - \sum_k q_{Ik}\fpar{f}{q_{Jk}}\biggr)\biggl(\sum_k q_{Jk}\fpar{g}{q_{Ik}} - \sum_k q_{Ik}\fpar{g}{q_{Jk}}\biggr)\biggr) + O(\ep^3)\\
 &= \ep^2 \ee\biggl(\sum_{k,k'} q_{Jk}q_{Jk'} \fpar{f}{q_{Ik}} \fpar{g}{q_{Ik'}}- \sum_{k,k'}q_{Jk}q_{Ik'} \fpar{f}{q_{Ik}} \fpar{g}{q_{Jk'}}\biggr) + O(\ep^3)\\
 &= \ep^2 \ee\biggl(\frac{1}{N-1}\sum_k \fpar{f}{q_{Ik}} \fpar{g}{q_{Ik}} - \frac{1}{N-1}\sum_{k,k'} q_{Ik}q_{Ik'} \fpar{f}{q_{Ik}} \fpar{g}{q_{Ik'}}- \sum_{k,k'}q_{Jk}q_{Ik'} \fpar{f}{q_{Ik}} \fpar{g}{q_{Jk'}}\biggr) + O(\ep^3)\\
 &= \ep^2 \ee\biggl(\frac{1}{N(N-1)}\sum_{i,k} \fpar{f}{q_{ik}} \fpar{g}{q_{ik}} - \frac{1}{N(N-1)}\sum_{i,j,k,k'}q_{jk}q_{ik'} \fpar{f}{q_{ik}} \fpar{g}{q_{jk'}}\biggr) + O(\ep^3)\, .
\end{align*}
The proof is now completed by applying Lemma \ref{exchlmm} and taking $\ep$ to zero. 
\end{proof}

\section{The master loop equation for finite $N$}\label{mm1sec}
The goal of this section is to prove Theorem \ref{mastersymm}. We will first prove an `unsymmetrized' version of the theorem. The proof is an application of Theorem \ref{sdthm}, with some heavy computations along the way. 
\begin{thm}\label{mastern}
Take any $N\ge 2$, $\beta\in \rr$ and a finite set $\Lambda\subseteq \zz^d$, and consider $SO(N)$ lattice gauge theory on $\Lambda$ at inverse coupling strength $\beta$. Take a loop sequence $s$ with minimal representation $(l_1,\ldots, l_n)$ such that all vertices of $\zz^d$ that are at distance $\le 1$ from any of the loops in $s$ are contained in $\Lambda$. Let $e$ be the first edge of $l_1$. For each $1\le r\le n$, let $A_r$ be the set of locations in $l_r$ where $e$ occurs, and let $B_r$ be the set of locations in $l_r$ where $e^{-1}$ occurs. Let $C_r = A_r \cup B_r$, and let $m$ be the size of $C_1$. Then 
\begin{align*}
(N-1)m \smallavg{W_{l_1}W_{l_2}\cdots W_{l_n}} &= \textup{twisting term} + \textup{splitting term}\\
&\quad  + \textup{merger term} + \textup{deformation term}\, , 
\end{align*}
where the twisting term is given by
\begin{align*}
&\sum_{\substack{x, y\in A_1\\x\ne y}} \smallavg{W_{\propto_{x,y} l_1}W_{l_2}\cdots W_{l_n}} + \sum_{\substack{x, y\in B_1\\x\ne y}} \smallavg{W_{\propto_{x,y} l_1}W_{l_2}\cdots W_{l_n}} \\
&\qquad - \sum_{x\in A_1, \, y\in B_1} \smallavg{W_{\propto_{x,y} l_1}W_{l_2}\cdots W_{l_n}}- \sum_{x\in B_1, \, y\in A_1} \smallavg{W_{\propto_{x,y} l_1}W_{l_2}\cdots W_{l_n}}\, , 
\end{align*}
the splitting term is given by
\begin{align*}
&\sum_{x\in A_1, \, y\in B_1} \smallavg{W_{\times_{x,y}^1 l_1} W_{\times_{x,y}^2 l_1}W_{l_2}\cdots W_{l_n}} + \sum_{x\in B_1, \, y\in A_1} \smallavg{W_{\times_{x,y}^1 l_1} W_{\times_{x,y}^2 l_1}W_{l_2}\cdots W_{l_n}}\\
&\quad   - \sum_{\substack{x,y\in A_1\\ x\ne y}} \smallavg{W_{\times^1_{x,y} l_1} W_{\times^2_{x,y} l_1}W_{l_2}\cdots W_{l_n}} - \sum_{\substack{x,y\in B_1\\ x\ne y}} \smallavg{W_{\times^1_{x,y} l_1} W_{\times^2_{x,y} l_1}W_{l_2}\cdots W_{l_n}}\, ,
\end{align*}
the merger term is given by
\begin{align*}
&\sum_{r=2}^n\sum_{x\in C_1, \, y\in C_r} \bigavg{W_{l_1 \ominus_{x,y} l_r}\prod_{\substack{2\le t\le n\\t\ne r}} W_{l_t}} - \sum_{r=2}^n\sum_{x\in C_1, \, y\in C_r}\bigavg{ W_{l_1 \oplus_{x,y} l_r}\prod_{\substack{2\le t\le n\\t\ne r}} W_{l_t}}\, ,
\end{align*}
and the deformation term equals
\begin{align*}
&N\beta \sum_{p\in \cp^+(e)}\sum_{x\in C_1}\smallavg{ W_{l_1 \ominus_{x}p}W_{l_2}\cdots W_{l_n}}  - N\beta \sum_{p\in \cp^+(e)}\sum_{x\in C_1}\smallavg{ W_{l_1 \oplus_{x}p}W_{l_2}\cdots W_{l_n}} \, .
\end{align*}
In all of the above, empty sums denote zero.
\end{thm}
The proof of Theorem \ref{mastern} is divided into a number of steps. Throughout, $e$ denotes the first edge of $l_1$, and $q^e_{ij}$ denotes the $(i,j)^{\mathrm{th}}$ entry of $Q_e$. The derivative of a matrix with respect to a scalar variable is understood to be entry-wise differentiation. Note that if $P$ and $Q$ are two matrices, both of which are functions of a scalar variable $x$, then 
\[
\fpar{}{x} (PQ) = \fpar{P}{x} Q + P \fpar{Q}{x}\,.
\]
This fact will be used repeatedly and without mention. Another simple fact that will be used without mention is that for any cycle $l$, $W_l = W_{[l]}$, where $[l]$ is the nonbacktracking core of $l$. This is easy to see because $W_l = W_{l'}$ whenever $l'$ is obtained from $l$ by a backtrack erasure. 
\begin{lmm}\label{s1}
With the above notation,
\[
\sum_{i,k} q_{ik}^e \fpar{Q_e}{q_{ik}^e} = Q_e\, , \ \text{ and } \ \sum_{i,k} q_{ik}^e \fpar{Q_e^T}{q_{ik}^e} = Q_e^T\, .
\]
\end{lmm}
\begin{proof}
Note that 
\[
\fpar{Q_e}{q^e_{ik}} = u_i u_k^T\, ,
\]
where $u_i\in \rr^N$ is the vector whose $i^{\mathrm{th}}$ coordinate is $1$ and the rest are zero. Thus, 
\[
\sum_{i,k} q^e_{ik} \fpar{Q_e}{q_{ik}^e} = \sum_{i,k} q_{ik}^e u_i u_k^T = Q_e\, ,
\]
and similarly
\[
\sum_{i,k} q^e_{ik} \fpar{Q_e^T}{q_{ik}^e} = \sum_{i,k} q_{ik}^e u_k u_i^T = Q_e^T\, .
\]
This completes the proof of the lemma. 
\end{proof}
\begin{lmm}\label{s2}
%In the notation of Theorem \ref{mastern},
\[
\sum_{i,k} q_{ik}^e \fpar{W_{l_1}}{q_{ik}^e} = m W_{l_1}\, .
\]
\end{lmm}
\begin{proof}
Write 
\[
W_{l_1} = \tr(P_1 Q_1P_2Q_2\cdots P_m Q_m P_{m+1})\, ,
\]
where each $Q_1,\ldots, Q_m$ is either $Q_e$ or $Q_e^T$, and $P_1,\ldots, P_{m+1}$ are products of $Q_{e'}$ where $e'$ is neither $e$ nor $e^{-1}$. Then 
\[
\fpar{W_{l_1}}{q_{ik}^e} = \sum_{r=1}^m \tr\biggl(P_1Q_1\cdots P_r \fpar{Q_r}{q_{ik}^e} P_{r+1}\cdots P_mQ_mP_{m+1}\biggr)\, .
\]
Consequently, by Lemma \ref{s1},
\begin{align*}
\sum_{i,k} q_{ik}^e \fpar{W_{l_1}}{q_{ik}^e} 
&= \sum_{i,k} q_{ik}^e\sum_{r=1}^m \tr\biggl(P_1Q_1\cdots P_r \fpar{Q_r}{q_{ik}^e} P_{r+1}\cdots P_mQ_mP_{m+1}\biggr)\\
&= \sum_{r=1}^m \tr\biggl(P_1Q_1\cdots P_r\biggl(\sum_{i,k}q_{ik}^e \fpar{Q_r}{q_{ik}^e} \biggr)P_{r+1}\cdots P_mQ_mP_{m+1}\biggr)\\
&= \sum_{r=1}^m \tr(P_1Q_1\cdots P_r Q_rP_{r+1}\cdots P_mQ_mP_{m+1}) = m W_{l_1}\, .
\end{align*}
This completes the proof of the lemma. 
\end{proof}
\begin{lmm}\label{d1}
\begin{align*}
\sum_{i,j,k,k'} q^e_{jk} q^e_{ik'}\mpar{W_{l_1}}{q^e_{ik}}{q^e_{jk'}} &=  \sum_{\substack{x,y\in A_1\\ x\ne y}}W_{\times^1_{x,y} l_1} W_{\times^2_{x,y} l_1} + \sum_{\substack{x,y\in B_1\\ x\ne y}}W_{\times^1_{x,y} l_1} W_{\times^2_{x,y} l_1}  \\
&\qquad + \sum_{x\in A_1,\, y\in B_1} W_{\propto_{x,y} l_1} + \sum_{x\in B_1,\, y\in A_1} W_{\propto_{x,y} l_1}\, .
\end{align*}
\end{lmm}
\begin{proof}
Let us continue using the notation introduced in the proof of Lemma \ref{s2}. Observe that 
\begin{align*}
\mpar{W_{l_1}}{q_{ik}^e}{q_{jk'}^e}
&= \sum_{1\le r< s\le m} \tr\biggl(P_1Q_1\cdots P_r \fpar{Q_r}{q_{ik}^e} P_{r+1}\cdots P_s \fpar{Q_s}{q_{jk'}^e}P_{s+1}\cdots P_mQ_mP_{m+1}\biggr)\\
&\quad + \sum_{1\le r<s\le m} \tr\biggl(P_1Q_1\cdots P_r \fpar{Q_r}{q_{jk'}^e} P_{r+1}\cdots P_s \fpar{Q_s}{q_{ik}^e}P_{s+1}\cdots P_mQ_mP_{m+1}\biggr)\, .
\end{align*}
By the symmetry between the pairs of indices $(i,k)$ and $(j,k')$, this shows that
\begin{align*}
&\sum_{i,j,k,k'} q^e_{jk} q^e_{ik'}\mpar{W_{l_1}}{q^e_{ik}}{q^e_{jk'}}\\
 &= 2\sum_{1\le r< s\le m}\sum_{i,j,k,k'}q_{jk}^e q_{ik'}^e \tr\biggl(P_1Q_1\cdots P_r \fpar{Q_r}{q_{ik}^e} P_{r+1}\cdots P_s \fpar{Q_s}{q_{jk'}^e}P_{s+1}\cdots P_mQ_mP_{m+1}\biggr)
\end{align*}
Now take any $1\le r<s\le m$. First, suppose that $Q_r = Q_s = Q_e$. Let $u_i$ be as in the proof of Lemma \ref{s1}. Then 
\begin{align}
&\tr\biggl(P_1Q_1\cdots P_r \fpar{Q_r}{q_{ik}^e} P_{r+1}\cdots P_s \fpar{Q_s}{q_{jk'}^e}P_{s+1}\cdots P_mQ_mP_{m+1}\biggr) \nonumber\\
&= \tr(P_1Q_1\cdots P_r u_i u_k^T P_{r+1}\cdots P_s u_j u_{k'}^TP_{s+1}\cdots P_mQ_mP_{m+1})\nonumber\\
&= (u_{k'}^TP_{s+1}\cdots P_mQ_mP_{m+1}P_1Q_1\cdots P_r u_i)( u_k^T P_{r+1}\cdots P_s u_j)\nonumber\\
&= (P_{s+1}\cdots P_mQ_mP_{m+1}P_1Q_1\cdots P_r)_{k'i}(P_{r+1}\cdots P_s)_{kj}\, ,\label{pqeq1}
\end{align}
where we are following the convention that $M_{ij}$ denotes the $(i,j)^{\mathrm{th}}$ entry of a matrix $M$. Therefore for this $r$ and $s$,
\begin{align*}
&\sum_{i,j,k, k'} q_{jk}^eq_{ik'}^e\tr\biggl(P_1Q_1\cdots P_r \fpar{Q_r}{q_{ik}^e} P_{r+1}\cdots P_s \fpar{Q_s}{q_{jk'}^e}P_{s+1}\cdots P_mQ_mP_{m+1}\biggr)\\
&= \sum_{i,j,k,k'} q_{jk}^e q_{ik'}^e (P_{s+1}\cdots P_mQ_mP_{m+1}P_1Q_1\cdots P_r)_{k'i}(P_{r+1}\cdots P_s)_{kj}\\
&= \tr(P_{s+1}\cdots P_mQ_mP_{m+1}P_1Q_1\cdots P_rQ_r) \tr(P_{r+1}\cdots P_sQ_s)\\
&= \tr(P_1Q_1\cdots P_rQ_rP_{s+1}\cdots P_mQ_mP_{m+1}) \tr(P_{r+1}\cdots P_sQ_s)\, .
\end{align*}
For $t=1,\ldots, m$, let $z_t$ be the location in $l_1$ of the $t^{\mathrm{th}}$ occurrence of $e$ or $e^{-1}$. Then, if $x= z_r$ and $y = z_s$, it is easy to see that the last line of the above display is exactly equal to 
\[
W_{\times^1_{x,y} l_1} W_{\times^2_{x,y} l_1}\, .
\]
This produces the first kind of terms in the statement of the lemma.

Next, suppose that $Q_r=Q_s=Q_e^T$. Then 
\begin{align}
&\tr\biggl(P_1Q_1\cdots P_r \fpar{Q_r}{q_{ik}^e} P_{r+1}\cdots P_s \fpar{Q_s}{q_{jk'}^e}P_{s+1}\cdots P_mQ_mP_{m+1}\biggr) \nonumber\\
&= \tr(P_1Q_1\cdots P_r u_k u_i^T P_{r+1}\cdots P_s u_{k'} u_{j}^TP_{s+1}\cdots P_mQ_mP_{m+1})\nonumber\\
&= (u_{j}^TP_{s+1}\cdots P_mQ_mP_{m+1}P_1Q_1\cdots P_r u_k)( u_i^T P_{r+1}\cdots P_s u_{k'})\nonumber\\
&= (P_{s+1}\cdots P_mQ_mP_{m+1}P_1Q_1\cdots P_r)_{jk}(P_{r+1}\cdots P_s)_{ik'}\, .\label{pqeq2}
\end{align}
Consequently,
\begin{align*}
&\sum_{i,j,k, k'} q_{jk}^eq_{ik'}^e\tr\biggl(P_1Q_1\cdots P_r \fpar{Q_r}{q_{ik}^e} P_{r+1}\cdots P_s \fpar{Q_s}{q_{jk'}^e}P_{s+1}\cdots P_mQ_mP_{m+1}\biggr)\\
&= \sum_{i,j,k,k'} q_{jk}^e q_{ik'}^e (P_{s+1}\cdots P_mQ_mP_{m+1}P_1Q_1\cdots P_r)_{jk}(P_{r+1}\cdots P_s)_{ik'}\\
&= \tr(P_{s+1}\cdots P_mQ_mP_{m+1}P_1Q_1\cdots P_rQ_r) \tr(P_{r+1}\cdots P_sQ_s)\\
&= \tr(P_1Q_1\cdots P_rQ_rP_{s+1}\cdots P_mQ_mP_{m+1}) \tr(P_{r+1}\cdots P_sQ_s)\, .
\end{align*}
If $x = z_r$ and $y=z_s$, then the last line of the above display is easily seen to be
\[
W_{\times^1_{x,y} l_1} W_{\times^2_{x,y} l_1}\, .
\]
This produces the second kind of terms in the statement of the lemma. 

Next, suppose that $Q_r = Q_e$ and $Q_s = Q_e^T$. Then 
\begin{align}
&\tr\biggl(P_1Q_1\cdots P_r \fpar{Q_r}{q_{ik}^e} P_{r+1}\cdots P_s \fpar{Q_s}{q_{jk'}^e}P_{s+1}\cdots P_mQ_mP_{m+1}\biggr) \nonumber\\
&= \tr(P_1Q_1\cdots P_r u_i u_k^T P_{r+1}\cdots P_s u_{k'} u_{j}^TP_{s+1}\cdots P_mQ_mP_{m+1})\nonumber\\
&= (u_{j}^TP_{s+1}\cdots P_mQ_mP_{m+1}P_1Q_1\cdots P_r u_i)( u_k^T P_{r+1}\cdots P_s u_{k'})\nonumber\\
&= (P_{s+1}\cdots P_mQ_mP_{m+1}P_1Q_1\cdots P_r)_{ji}(P_{r+1}\cdots P_s)_{kk'}\, .\label{pqeq3}
\end{align}
So in this case,
\begin{align*}
&\sum_{i,j,k, k'} q_{jk}^eq_{ik'}^e\tr\biggl(P_1Q_1\cdots P_r \fpar{Q_r}{q_{ik}^e} P_{r+1}\cdots P_s \fpar{Q_s}{q_{jk'}^e}P_{s+1}\cdots P_mQ_mP_{m+1}\biggr)\\
&= \sum_{i,j,k,k'} q_{jk}^e q_{ik'}^e (P_{s+1}\cdots P_mQ_mP_{m+1}P_1Q_1\cdots P_r)_{ji}(P_{r+1}\cdots P_s)_{kk'}\\
&= \tr(P_{s+1}\cdots P_mQ_mP_{m+1}P_1Q_1\cdots P_rQ_r(P_{r+1}\cdots P_s)^T Q_s)\\
&= \tr(P_1Q_1\cdots P_rQ_r(P_{r+1}\cdots P_s)^T Q_sP_{s+1}\cdots P_mQ_mP_{m+1})\, .
\end{align*}
If $x = z_r$ and $y=z_s$, then the last line of the above display equals 
$W_{\propto_{x,y} l_1}$. 
This produces the third kind of terms in the statement of the lemma, namely, those that have $x\in A_1$ and $y\in B_1$. For the fourth kind of terms, suppose that $Q_r = Q_e^T$ and $Q_s = Q_e$. Then 
\begin{align}
&\tr\biggl(P_1Q_1\cdots P_r \fpar{Q_r}{q_{ik}^e} P_{r+1}\cdots P_s \fpar{Q_s}{q_{jk'}^e}P_{s+1}\cdots P_mQ_mP_{m+1}\biggr) \nonumber\\
&= \tr(P_1Q_1\cdots P_r u_k u_i^T P_{r+1}\cdots P_s u_{j} u_{k'}^TP_{s+1}\cdots P_mQ_mP_{m+1})\nonumber\\
&= (u_{k'}^TP_{s+1}\cdots P_mQ_mP_{m+1}P_1Q_1\cdots P_r u_k)( u_i^T P_{r+1}\cdots P_s u_{j})\nonumber\\
&= (P_{s+1}\cdots P_mQ_mP_{m+1}P_1Q_1\cdots P_r)_{k'k}(P_{r+1}\cdots P_s)_{ij}\, .\label{pqeq4}
\end{align}
Consequently,
\begin{align*}
&\sum_{i,j,k, k'} q_{jk}^eq_{ik'}^e\tr\biggl(P_1Q_1\cdots P_r \fpar{Q_r}{q_{ik}^e} P_{r+1}\cdots P_s \fpar{Q_s}{q_{jk'}^e}P_{s+1}\cdots P_mQ_mP_{m+1}\biggr)\\
&= \sum_{i,j,k,k'} q_{jk}^e q_{ik'}^e (P_{s+1}\cdots P_mQ_mP_{m+1}P_1Q_1\cdots P_r)_{k'k}(P_{r+1}\cdots P_s)_{ij}\\
&= \tr(Q_sP_{s+1}\cdots P_mQ_mP_{m+1}P_1Q_1\cdots P_rQ_r(P_{r+1}\cdots P_s)^T )\\
&= \tr(P_1Q_1\cdots P_rQ_r(P_{r+1}\cdots P_s)^T Q_sP_{s+1}\cdots P_mQ_mP_{m+1})\, .
\end{align*}
As before, if $x = z_r$ and $y=z_s$, then the last line of the above display equals $W_{\propto_{x,y} l_1}$. 
\end{proof}
\begin{lmm}\label{d2}
\begin{align*}
\sum_{i,k}\spar{W_{l_1}}{{q^e_{ik}}} &=  \sum_{\substack{x,y\in A_1\\ x\ne y}} W_{\propto_{x,y}l_1}+ \sum_{\substack{x,y\in B_1\\ x\ne y}}W_{\propto_{x,y} l_1} \\
&\qquad + \sum_{x\in A_1,\, y\in B_1} W_{\times^1_{x,y} l_1} W_{\times^2_{x,y} l_1}+ \sum_{x\in B_1,\, y\in A_1} W_{\times^1_{x,y} l_1} W_{\times^2_{x,y} l_1}\, .
\end{align*}
\end{lmm}
\begin{proof}
We will continue to use the notations and calculations from the proof of Lemma \ref{d1}. First, note that
\begin{align*}
\sum_{i,k} \spar{W_{l_1}}{{q^e_{ik}}}
 &= 2\sum_{1\le r< s\le m}\sum_{i,k}\tr\biggl(P_1Q_1\cdots P_r \fpar{Q_r}{q_{ik}^e} P_{r+1}\cdots P_s \fpar{Q_s}{q_{ik}^e}P_{s+1}\cdots P_mQ_mP_{m+1}\biggr)\, .
\end{align*}
Take any $1\le r<s \le m$. First, suppose that $Q_r = Q_s = Q_e$. Then  by \eqref{pqeq1},
\begin{align*}
&\sum_{i,k} \tr\biggl(P_1Q_1\cdots P_r \fpar{Q_r}{q_{ik}^e} P_{r+1}\cdots P_s \fpar{Q_s}{q_{ik}^e}P_{s+1}\cdots P_mQ_mP_{m+1}\biggr)\\
&= \sum_{i,k}(P_{s+1}\cdots P_mQ_mP_{m+1}P_1Q_1\cdots P_r)_{ki}(P_{r+1}\cdots P_s)_{ki}\\
&= \tr (P_{s+1}\cdots P_mQ_mP_{m+1}P_1Q_1\cdots P_r(P_{r+1}\cdots P_s)^T)\\
&= \tr(P_1Q_1\cdots Q_{r-1} P_r (P_{r+1}\cdots P_s)^T P_{s+1}Q_{s+1}\cdots P_mQ_mP_{m+1})\, .
\end{align*}
If $x= z_r$ and $y=z_s$, then this is simply $W_{\propto_{x,y}l_1}$. This gives the first kind of terms in the statement of the lemma. 

Next, suppose that $Q_r = Q_s = Q_e^T$. Then by \eqref{pqeq2},
\begin{align*}
&\sum_{i,k} \tr\biggl(P_1Q_1\cdots P_r \fpar{Q_r}{q_{ik}^e} P_{r+1}\cdots P_s \fpar{Q_s}{q_{ik}^e}P_{s+1}\cdots P_mQ_mP_{m+1}\biggr)\\
&= \sum_{i,k}(P_{s+1}\cdots P_mQ_mP_{m+1}P_1Q_1\cdots P_r)_{ik}(P_{r+1}\cdots P_s)_{ik}\\
&= \tr (P_{s+1}\cdots P_mQ_mP_{m+1}P_1Q_1\cdots P_r(P_{r+1}Q_{r+1}\cdots Q_{s-1}P_s)^T)\\
&= \tr (P_1Q_1\cdots P_r(P_{r+1}Q_{r+1}\cdots Q_{s-1}P_s)^TP_{s+1}\cdots P_mQ_mP_{m+1})\, ,
\end{align*}
Again, if $x= z_r$ and $y=z_s$, this is equal to $W_{\propto_{x,y}l_1}$. This gives the second kind of terms in the statement of the lemma.

Next, suppose that $Q_r = Q_e$ and $Q_s = Q_e^T$. Then by \eqref{pqeq3},
\begin{align*}
&\sum_{i,k} \tr\biggl(P_1Q_1\cdots P_r \fpar{Q_r}{q_{ik}^e} P_{r+1}\cdots P_s \fpar{Q_s}{q_{ik}^e}P_{s+1}\cdots P_mQ_mP_{m+1}\biggr)\\
&= \sum_{i,k}(P_{s+1}\cdots P_mQ_mP_{m+1}P_1Q_1\cdots P_r)_{ii}(P_{r+1}\cdots P_s)_{kk}\\
&= \tr (P_{s+1}\cdots P_mQ_mP_{m+1}P_1Q_1\cdots P_r)\tr(P_{r+1}Q_{r+1}\cdots Q_{s-1}P_s)\\
&= \tr (P_1Q_1\cdots P_rP_{s+1}\cdots P_mQ_mP_{m+1})\tr(P_{r+1}Q_{r+1}\cdots Q_{s-1}P_s)\, ,
\end{align*}
If $x= z_r$ and $y=z_s$, this is equal to 
\[
W_{\times^1_{x,y} l_1} W_{\times^2_{x,y} l_1}\, .
\]
This gives the third kind of terms in the statement of the lemma, namely, those that have $x\in A_1$ and $y\in B_1$.

Finally, suppose that $Q_r = Q_e^T$ and $Q_s = Q_e$. Then by \eqref{pqeq4}, 
\begin{align*}
&\sum_{i,k} \tr\biggl(P_1Q_1\cdots P_r \fpar{Q_r}{q_{ik}^e} P_{r+1}\cdots P_s \fpar{Q_s}{q_{ik}^e}P_{s+1}\cdots P_mQ_mP_{m+1}\biggr)\\
&= \sum_{i,k}(P_{s+1}\cdots P_mQ_mP_{m+1}P_1Q_1\cdots P_r)_{kk}(P_{r+1}\cdots P_s)_{ii}\\
&= \tr (P_{s+1}\cdots P_mQ_mP_{m+1}P_1Q_1\cdots P_r)\tr(P_{r+1}Q_{r+1}\cdots Q_{s-1}P_s)\\
&= \tr (P_1Q_1\cdots P_rP_{s+1}\cdots P_mQ_mP_{m+1})\tr(P_{r+1}Q_{r+1}\cdots Q_{s-1}P_s)\, ,
\end{align*}
Again, if $x= z_r$ and $y=z_s$, this is equal to 
\[
W_{\times^1_{x,y} l_1} W_{\times^2_{x,y} l_1}\, .
\]
This gives the fourth kind of terms in the statement of the lemma, namely, those that have $x\in B_1$ and $y\in A_1$.
\end{proof}

\begin{lmm}\label{t1}
Let $l'$ be a non-null loop such that all points within distance $1$ of $l'$ belong to $\Lambda$. Let $A'$ be the set of locations in $l'$ where $e$ occurs, and let $B'$ be the set of locations in $l'$ where $e^{-1}$ occurs. Let $C' = A'\cup B'$ and assume that $C'$ is nonempty. Then 
\[
\sum_{i,j,k,k'} q_{jk}^e q_{ik'}^e \fpar{W_{l_1}}{q_{ik}^e} \fpar{W_{l'}}{q_{jk'}^e} = \sum_{x\in C_1, \, y\in C'} W_{l_1\oplus_{x,y} l'}\, .
\]
\end{lmm}
\begin{proof}
Let $m'$ be the size of $C'$. Write $W_{l'}$ as 
\[
W_{l'} = \tr(P_1'Q_1'P_2'Q_2'\cdots P_{m'}'Q_{m'}' P_{m'+1}')\,  
\]
where each $Q_1',\ldots, Q_{m'}'$ is either $Q_e$ or $Q_e^T$, and $P_1',\ldots, P_{m'+1}'$ are products of $Q_{e'}$ where $e'$ is neither $e$ nor $e^{-1}$. Then
\begin{align*}
&\fpar{W_{l_1}}{q_{ik}^e}\fpar{W_{l'}}{q_{jk'}^e}\\
&= \sum_{\substack{1\le r\le m\\1\le s\le m'}}\tr\biggl(P_1Q_1\cdots P_r \fpar{Q_r}{q_{ik}^e} P_{r+1}\cdots Q_m P_{m+1}\biggr)\tr\biggl( P_1'Q_1'\cdots P_s'\fpar{Q_s'}{q_{jk'}^e}P_{s+1}'\cdots Q'_{m'}P'_{m'+1}\biggr)\, .
\end{align*}
Take any $r,s$. If $Q_r = Q_s' = Q_e$, then 
\begin{align}
&\tr\biggl(P_1Q_1\cdots P_r \fpar{Q_r}{q_{ik}^e} P_{r+1}\cdots Q_m P_{m+1}\biggr)\tr\biggl( P_1'Q_1'\cdots P_s'\fpar{Q_s'}{q_{jk'}^e}P_{s+1}'\cdots Q'_{m'}P'_{m'+1}\biggr) \nonumber\\
&= \tr(P_1Q_1\cdots P_r u_i u_k^T P_{r+1}\cdots Q_m P_{m+1})\tr( P_1'Q_1'\cdots P_s'u_j u_{k'}^TP_{s+1}'\cdots Q'_{m'}P'_{m'+1}) \nonumber\\
&= (u_k^T P_{r+1}\cdots Q_m P_{m+1}P_1Q_1\cdots P_r u_i )( u_{k'}^TP_{s+1}'\cdots Q'_{m'}P'_{m'+1}P_1'Q_1'\cdots P_s'u_j )\nonumber \\
&= (P_{r+1}\cdots Q_m P_{m+1}P_1Q_1\cdots P_r )_{ki}( P_{s+1}'\cdots Q'_{m'}P'_{m'+1}P_1'Q_1'\cdots P_s')_{k'j}\, .\label{spqeq1}
\end{align}
Therefore, 
\begin{align*}
&\sum_{i,j,k,k'} q_{jk}^eq_{ik'}^e\tr\biggl(P_1Q_1\cdots P_r \fpar{Q_r}{q_{ik}^e} P_{r+1}\cdots Q_m P_{m+1}\biggr)\tr\biggl( P_1'Q_1'\cdots P_s'\fpar{Q_s'}{q_{jk'}^e}P_{s+1}'\cdots Q'_{m'}P'_{m'+1}\biggr)\\
&= \sum_{i,j,k,k'} q_{jk}^eq_{ik'}^e(P_{r+1}\cdots Q_m P_{m+1}P_1Q_1\cdots P_r )_{ki}( P_{s+1}'\cdots Q'_{m'}P'_{m'+1}P_1'Q_1'\cdots P_s')_{k'j}\\
&= \tr(P_{r+1}\cdots Q_m P_{m+1}P_1Q_1\cdots P_rQ_r P_{s+1}'\cdots Q'_{m'}P'_{m'+1}P_1'Q_1'\cdots P_s'Q_s')\\
&= \tr(P_1Q_1\cdots P_rQ_r P_{s+1}'\cdots Q'_{m'}P'_{m'+1}P_1'Q_1'\cdots P_s'Q_s'P_{r+1}\cdots Q_m P_{m+1})\, .
\end{align*}
Let $z_t$ be the location in $l_1$ of the $t^{\mathrm{th}}$ occurrence of $e$ or $e^{-1}$ , as before. Let $z'_t$ be the location in $l'$ of the $t^{\mathrm{th}}$ occurrence of $e$ or $e^{-1}$. If $x= z_r$ and $y = z'_s$, then it is easy to see that the last term in the above display is exactly $W_{l_1\oplus_{x,y} l'}$. This gives the terms corresponding to $x\in A_1$ and $y\in A'$ in the statement of the lemma. 

Next, suppose that $Q_r = Q_e$ and $Q_s' = Q_e^T$. Then 
\begin{align}
&\tr\biggl(P_1Q_1\cdots P_r \fpar{Q_r}{q_{ik}^e} P_{r+1}\cdots Q_m P_{m+1}\biggr)\tr\biggl( P_1'Q_1'\cdots P_s'\fpar{Q_s'}{q_{jk'}^e}P_{s+1}'\cdots Q'_{m'}P'_{m'+1}\biggr)\nonumber \\
&= \tr(P_1Q_1\cdots P_r u_i u_k^T P_{r+1}\cdots Q_m P_{m+1})\tr( P_1'Q_1'\cdots P_s'u_{k'} u_{j}^TP_{s+1}'\cdots Q'_{m'}P'_{m'+1})\nonumber \\
&= (u_k^T P_{r+1}\cdots Q_m P_{m+1}P_1Q_1\cdots P_r u_i )( u_{j}^TP_{s+1}'\cdots Q'_{m'}P'_{m'+1}P_1'Q_1'\cdots P_s'u_{k'} ) \nonumber\\
&= (P_{r+1}\cdots Q_m P_{m+1}P_1Q_1\cdots P_r )_{ki}( P_{s+1}'\cdots Q'_{m'}P'_{m'+1}P_1'Q_1'\cdots P_s')_{jk'}\, .\label{spqeq2}
\end{align}
Thus, 
\begin{align*}
&\sum_{i,j,k,k'} q_{jk}^eq_{ik'}^e\tr\biggl(P_1Q_1\cdots P_r \fpar{Q_r}{q_{ik}^e} P_{r+1}\cdots Q_m P_{m+1}\biggr)\tr\biggl( P_1'Q_1'\cdots P_s'\fpar{Q_s'}{q_{jk'}^e}P_{s+1}'\cdots Q'_{m'}P'_{m'+1}\biggr)\\
&= \sum_{i,j,k,k'} q_{jk}^eq_{ik'}^e(P_{r+1}\cdots Q_m P_{m+1}P_1Q_1\cdots P_r )_{ki}( P_{s+1}'\cdots Q'_{m'}P'_{m'+1}P_1'Q_1'\cdots P_s')_{jk'}\\
&= \tr(P_{r+1}\cdots Q_m P_{m+1}P_1Q_1\cdots P_rQ_r (P_{s+1}'\cdots Q'_{m'}P'_{m'+1}P_1'Q_1'\cdots P_s')^T{Q_s'}^T)\\
&= \tr(P_1Q_1\cdots P_rQ_r (P_{s+1}'\cdots Q'_{m'}P'_{m'+1}P_1'Q_1'\cdots P_s')^T{Q_s'}^TP_{r+1}\cdots Q_m P_{m+1})\, .
\end{align*}
If $x=z_r$ and $y=z_s$, this is equal to $W_{l_1\oplus_{x,y}l'}$. This gives the terms corresponding to $x\in A_1$, $y\in B'$. 

Next, suppose that  $Q_r = Q_e^T$ and $Q_s' = Q_e$. Then 
\begin{align}
&\tr\biggl(P_1Q_1\cdots P_r \fpar{Q_r}{q_{ik}^e} P_{r+1}\cdots Q_m P_{m+1}\biggr)\tr\biggl( P_1'Q_1'\cdots P_s'\fpar{Q_s'}{q_{jk'}^e}P_{s+1}'\cdots Q'_{m'}P'_{m'+1}\biggr) \nonumber\\
&= \tr(P_1Q_1\cdots P_r u_k u_i^T P_{r+1}\cdots Q_m P_{m+1})\tr( P_1'Q_1'\cdots P_s'u_{j} u_{k'}^TP_{s+1}'\cdots Q'_{m'}P'_{m'+1}) \nonumber\\
&= (u_i^T P_{r+1}\cdots Q_m P_{m+1}P_1Q_1\cdots P_r u_k)( u_{k'}^TP_{s+1}'\cdots Q'_{m'}P'_{m'+1}P_1'Q_1'\cdots P_s'u_{j} ) \nonumber\\
&= (P_{r+1}\cdots Q_m P_{m+1}P_1Q_1\cdots P_r )_{ik}( P_{s+1}'\cdots Q'_{m'}P'_{m'+1}P_1'Q_1'\cdots P_s')_{k'j}\, .\label{spqeq3}
\end{align}
Consequently,
\begin{align*}
&\sum_{i,j,k,k'} q_{jk}^eq_{ik'}^e\tr\biggl(P_1Q_1\cdots P_r \fpar{Q_r}{q_{ik}^e} P_{r+1}\cdots Q_m P_{m+1}\biggr)\tr\biggl( P_1'Q_1'\cdots P_s'\fpar{Q_s'}{q_{jk'}^e}P_{s+1}'\cdots Q'_{m'}P'_{m'+1}\biggr)\\
&= \sum_{i,j,k,k'} q_{jk}^eq_{ik'}^e(P_{r+1}\cdots Q_m P_{m+1}P_1Q_1\cdots P_r )_{ik}( P_{s+1}'\cdots Q'_{m'}P'_{m'+1}P_1'Q_1'\cdots P_s')_{k'j}\\
&= \tr(P_{r+1}\cdots Q_m P_{m+1}P_1Q_1\cdots P_rQ_r (P_{s+1}'\cdots Q'_{m'}P'_{m'+1}P_1'Q_1'\cdots P_s')^T{Q_s'}^T)\\
&= \tr(P_1Q_1\cdots P_rQ_r (P_{s+1}'\cdots Q'_{m'}P'_{m'+1}P_1'Q_1'\cdots P_s')^T{Q_s'}^TP_{r+1}\cdots Q_m P_{m+1})\, .
\end{align*}
Again, if we take $x = z_r$ and $ y = z'_s$ then the above expression is nothing but $W_{l_1\oplus_{x,y} l'}$. This takes care of the terms corresponding to $x\in B_1$ and $y\in A'$ in the statement of the lemma. 

Finally, suppose that $Q_r = Q_s' = Q_e^T$. Then 
\begin{align}
&\tr\biggl(P_1Q_1\cdots P_r \fpar{Q_r}{q_{ik}^e} P_{r+1}\cdots Q_m P_{m+1}\biggr)\tr\biggl( P_1'Q_1'\cdots P_s'\fpar{Q_s'}{q_{jk'}^e}P_{s+1}'\cdots Q'_{m'}P'_{m'+1}\biggr) \nonumber\\
&= \tr(P_1Q_1\cdots P_r u_k u_i^T P_{r+1}\cdots Q_m P_{m+1})\tr( P_1'Q_1'\cdots P_s'u_{k'} u_{j}^TP_{s+1}'\cdots Q'_{m'}P'_{m'+1}) \nonumber\\
&= (u_i^T P_{r+1}\cdots Q_m P_{m+1}P_1Q_1\cdots P_r u_k)( u_{j}^TP_{s+1}'\cdots Q'_{m'}P'_{m'+1}P_1'Q_1'\cdots P_s'u_{k'} ) \nonumber\\
&= (P_{r+1}\cdots Q_m P_{m+1}P_1Q_1\cdots P_r )_{ik}( P_{s+1}'\cdots Q'_{m'}P'_{m'+1}P_1'Q_1'\cdots P_s')_{jk'}\, .\label{spqeq4}
\end{align}
Therefore, 
\begin{align*}
&\sum_{i,j,k,k'} q_{jk}^eq_{ik'}^e\tr\biggl(P_1Q_1\cdots P_r \fpar{Q_r}{q_{ik}^e} P_{r+1}\cdots Q_m P_{m+1}\biggr)\tr\biggl( P_1'Q_1'\cdots P_s'\fpar{Q_s'}{q_{jk'}^e}P_{s+1}'\cdots Q'_{m'}P'_{m'+1}\biggr)\\
&= \sum_{i,j,k,k'} q_{jk}^eq_{ik'}^e(P_{r+1}\cdots Q_m P_{m+1}P_1Q_1\cdots P_r )_{ik}( P_{s+1}'\cdots Q'_{m'}P'_{m'+1}P_1'Q_1'\cdots P_s')_{jk'}\\
&= \tr(P_{r+1}\cdots Q_m P_{m+1}P_1Q_1\cdots P_rQ_r P_{s+1}'\cdots Q'_{m'}P'_{m'+1}P_1'Q_1'\cdots P_s'Q_s')\\
&= \tr(P_1Q_1\cdots P_rQ_r P_{s+1}'\cdots Q'_{m'}P'_{m'+1}P_1'Q_1'\cdots P_s'Q_s'P_{r+1}\cdots Q_m P_{m+1})\, .
\end{align*}
As before, if $x = z_r$ and $ y = z'_s$ then the above expression equals $W_{l_1\oplus_{x,y} l'}$. This takes care of the terms corresponding to $x\in B_1$ and $y\in B'$ in the statement of the lemma. 
\end{proof}

\begin{lmm}\label{t2}
Let $l'$ and $C'$ be as in the previous lemma. Then 
\[
\sum_{i,k} \fpar{W_{l_1}}{q^e_{ik}} \fpar{W_{l'}}{q^e_{ik}} = \sum_{x\in C_1,\, y\in C'} W_{l_1\ominus_{x,y} l'}\, .
\]
\end{lmm}
\begin{proof}
We will continue using the notations introduced in the proof of Lemma \ref{t1}. First, note that
\begin{align*}
&\sum_{i,k}\fpar{W_{l_1}}{q_{ik}^e}\fpar{W_{l'}}{q_{jk'}^e}\\
&= \sum_{\substack{1\le r\le m\\1\le s\le m'}}\sum_{i,k}\tr\biggl(P_1Q_1\cdots P_r \fpar{Q_r}{q_{ik}^e} P_{r+1}\cdots Q_m P_{m+1}\biggr)\tr\biggl( P_1'Q_1'\cdots P_s'\fpar{Q_s'}{q_{ik}^e}P_{s+1}'\cdots Q'_{m'}P'_{m'+1}\biggr)\, .
\end{align*}
Take any $r$, $s$. First, suppose that $Q_r = Q_s' = Q_e$. By \eqref{spqeq1}, 
\begin{align*}
&\sum_{i,k} \tr\biggl(P_1Q_1\cdots P_r \fpar{Q_r}{q_{ik}^e} P_{r+1}\cdots Q_m P_{m+1}\biggr)\tr\biggl( P_1'Q_1'\cdots P_s'\fpar{Q_s'}{q_{ik}^e}P_{s+1}'\cdots Q'_{m'}P'_{m'+1}\biggr)\\
&=\sum_{i,k} (P_{r+1}\cdots Q_m P_{m+1}P_1Q_1\cdots P_r )_{ki}( P_{s+1}'\cdots Q'_{m'}P'_{m'+1}P_1'Q_1'\cdots P_s')_{ki}\\
&= \tr(P_{r+1}\cdots Q_m P_{m+1}P_1Q_1\cdots P_r( P_{s+1}'\cdots Q'_{m'}P'_{m'+1}P_1'Q_1'\cdots P_s')^T)\\
&= \tr(P_1Q_1\cdots P_r( P_{s+1}'\cdots Q'_{m'}P'_{m'+1}P_1'Q_1'\cdots P_s')^TP_{r+1}\cdots Q_m P_{m+1})\, .
\end{align*}
If $x = z_r$ and $y = z_s'$, the above expression equals $W_{l_1\ominus_{x,y} l'}$. This takes care of the terms in the statement of the lemma that correspond to $x\in A_1$ and $y\in A'$. 

Next, suppose that $Q_r = Q_e$ and $Q'_s = Q_e^T$. Then by \eqref{spqeq2}, 
\begin{align*}
&\sum_{i,k} \tr\biggl(P_1Q_1\cdots P_r \fpar{Q_r}{q_{ik}^e} P_{r+1}\cdots Q_m P_{m+1}\biggr)\tr\biggl( P_1'Q_1'\cdots P_s'\fpar{Q_s'}{q_{ik}^e}P_{s+1}'\cdots Q'_{m'}P'_{m'+1}\biggr)\\
&=\sum_{i,k} (P_{r+1}\cdots Q_m P_{m+1}P_1Q_1\cdots P_r )_{ki}( P_{s+1}'\cdots Q'_{m'}P'_{m'+1}P_1'Q_1'\cdots P_s')_{ik}\\
&= \tr(P_{r+1}\cdots Q_m P_{m+1}P_1Q_1\cdots P_rP_{s+1}'\cdots Q'_{m'}P'_{m'+1}P_1'Q_1'\cdots P_s')\\
&= \tr(P_1Q_1\cdots P_rP_{s+1}'\cdots Q'_{m'}P'_{m'+1}P_1'Q_1'\cdots P_s'P_{r+1}\cdots Q_m P_{m+1})\, . 
\end{align*}
As before, if $x = z_r$ and $y = z_s'$, the above expression equals $W_{l_1\ominus_{x,y} l'}$. This takes care of the terms in the statement of the lemma that correspond to $x\in A_1$ and $y\in B'$. 

Next, suppose that $Q_r = Q_e^T$ and $Q_s' = Q_e$. Then by \eqref{spqeq3},
\begin{align*}
&\sum_{i,k} \tr\biggl(P_1Q_1\cdots P_r \fpar{Q_r}{q_{ik}^e} P_{r+1}\cdots Q_m P_{m+1}\biggr)\tr\biggl( P_1'Q_1'\cdots P_s'\fpar{Q_s'}{q_{ik}^e}P_{s+1}'\cdots Q'_{m'}P'_{m'+1}\biggr)\\
&=\sum_{i,k} (P_{r+1}\cdots Q_m P_{m+1}P_1Q_1\cdots P_r )_{ik}( P_{s+1}'\cdots Q'_{m'}P'_{m'+1}P_1'Q_1'\cdots P_s')_{ki}\\
&= \tr(P_{r+1}\cdots Q_m P_{m+1}P_1Q_1\cdots P_rP_{s+1}'\cdots Q'_{m'}P'_{m'+1}P_1'Q_1'\cdots P_s')\\
&= \tr(P_1Q_1\cdots P_rP_{s+1}'\cdots Q'_{m'}P'_{m'+1}P_1'Q_1'\cdots P_s'P_{r+1}\cdots Q_m P_{m+1})\, .
\end{align*}
Again, if $x = z_r$ and $y = z_s'$, this is equal to $W_{l_1\ominus_{x,y} l'}$. This accounts for the terms with $x\in B_1$ and $y\in A'$. 

Finally, suppose that $Q_r=Q_s'=Q_e^T$. Then by \eqref{spqeq4}, 
\begin{align*}
&\sum_{i,k}\tr\biggl(P_1Q_1\cdots P_r \fpar{Q_r}{q_{ik}^e} P_{r+1}\cdots Q_m P_{m+1}\biggr)\tr\biggl( P_1'Q_1'\cdots P_s'\fpar{Q_s'}{q_{ik}^e}P_{s+1}'\cdots Q'_{m'}P'_{m'+1}\biggr) \nonumber\\
&= \sum_{i,k}(P_{r+1}\cdots Q_m P_{m+1}P_1Q_1\cdots P_r )_{ik}( P_{s+1}'\cdots Q'_{m'}P'_{m'+1}P_1'Q_1'\cdots P_s')_{ik}\\
&= \tr(P_{r+1}\cdots Q_m P_{m+1}P_1Q_1\cdots P_r ( P_{s+1}'\cdots Q'_{m'}P'_{m'+1}P_1'Q_1'\cdots P_s')^T)\\
&= \tr(P_1Q_1\cdots P_r ( P_{s+1}'\cdots Q'_{m'}P'_{m'+1}P_1'Q_1'\cdots P_s')^TP_{r+1}\cdots Q_m P_{m+1})\, .
\end{align*}
As always, if $x = z_r$ and $y = z_s'$, this is equal to $W_{l_1\ominus_{x,y} l'}$. This accounts for the terms with $x\in B_1$ and $y\in B'$. 
\end{proof}

Having verified all the tedious calculations, we are now ready to prove Theorem \ref{mastern}. 
\begin{proof}[Proof of Theorem \ref{mastern}]
Let $Q = (Q_{e'})_{e'\in E^+_\Lambda}$ denote a collection of independent Haar-distributed random $SO(N)$ matrices. Let $f$ and $g$ be two functions of $Q$, defined as 
\[
f := W_{l_1}\, , 
\]
and
\[
g := Z_{\Lambda, N, \beta}^{-1} W_{l_2}W_{l_3}\cdots W_{l_n} \exp\biggl(N\beta\sum_{p\in \cp^+_\Lambda} W_p\biggr)\, .
\]
By Lemma \ref{s2}, 
\begin{align*}
\sum_{i,k} q_{ik}^e \fpar{f}{q_{ik}^e} = m W_{l_1}\, .
\end{align*}
Therefore,
\begin{align*}
\ee\biggl(\sum_{i,k} q_{ik}^e \fpar{f}{q_{ik}^e} g \biggr) &= m\smallavg{W_{l_1}W_{l_2}\cdots W_{l_n}}\, .
\end{align*}
By Lemma \ref{d2}, 
\begin{align*}
&\ee\biggl(\sum_{i,k} \spar{f}{{q_{ik}^e}} g \biggr) = \sum_{\substack{x,y\in A_1\\ x\ne y}} \smallavg{W_{\propto_{x,y}l_1}W_{l_2}\cdots W_{l_n}} + \sum_{\substack{x,y\in B_1\\ x\ne y}}\smallavg{W_{\propto_{x,y} l_1}W_{l_2}\cdots W_{l_n}} \\
&\quad  + \sum_{x\in A_1,\, y\in B_1}\smallavg{ W_{\times^1_{x,y} l_1} W_{\times^2_{x,y} l_1}W_{l_2}\cdots W_{l_n}}+\sum_{x\in B_1,\, y\in A_1}\smallavg{ W_{\times^1_{x,y} l_1} W_{\times^2_{x,y} l_1}W_{l_2}\cdots W_{l_n}}\, .
\end{align*}
By Lemma \ref{d1}, 
\begin{align*}
&\ee\biggl(\sum_{i,j,k,k'} q_{jk}^e q_{ik'}^e \mpar{f}{q_{ik}^e} {q_{jk'}^e}g \biggr) \\
&= \sum_{\substack{x,y\in A_1\\ x\ne y}}\smallavg{W_{\times^1_{x,y} l_1} W_{\times^2_{x,y} l_1} W_{l_2}\cdots W_{l_n}}+ \sum_{\substack{x,y\in B_1\\ x\ne y}}\smallavg{W_{\times^1_{x,y} l_1} W_{\times^2_{x,y} l_1}W_{l_2}\cdots W_{l_n}} \\
&\quad + \sum_{x\in A_1,\, y\in B_1}\smallavg{ W_{\propto_{x,y} l_1}W_{l_2}\cdots W_{l_n}}+ \sum_{x\in B_1,\, y\in A_1}\smallavg{ W_{\propto_{x,y} l_1}W_{l_2}\cdots W_{l_n}}\, .
\end{align*}
Next, note that since all vertices that are within distance $1$ of $l_1$ belong to $\Lambda$, therefore any element of $\cp^+(e)$ must necessarily belong to $\cp^+_\Lambda$. Therefore, for any $i$ and $k$,
\begin{align}
\fpar{g}{q_{ik}^e} &= \sum_{r=2}^n Z_{\Lambda, N,\beta}^{-1} W_{l_2}\cdots W_{l_{r-1}} \fpar{W_{l_r}}{q_{ik}^e} W_{l_{r+1}}\cdots W_{l_n}\exp\biggl(N\beta\sum_{p\in \cp^+_\Lambda} W_p\biggr)\nonumber\\
&\quad + \sum_{p\in \cp^+(e)} Z_{\Lambda, N,\beta}^{-1} W_{l_2}\cdots W_{l_n}N\beta \fpar{W_p}{q_{ik}^e}\exp\biggl(N\beta\sum_{p'\in \cp^+_\Lambda} W_{p'}\biggr)\, .\label{gq}
\end{align}
The above identity and Lemma \ref{t2} show that
\begin{align*}
\ee\biggl(\sum_{i,k} \fpar{f}{q_{ik}^e} \fpar{g}{q_{ik}^e}\biggr) &= \sum_{r=2}^n \sum_{x\in C_1,\, y\in C_r} \bigavg{W_{l_1\ominus_{x,y} l_r} \prod_{\substack{2\le t\le n\\t\ne r}} W_{l_t}} + N\beta \sum_{p\in \cp^+(e)} \sum_{x\in C_1} \smallavg{W_{l_1\ominus_x p}W_{l_2}\cdots W_{l_n}}\, .
\end{align*}
Similarly by Lemma \ref{t1} and equation \eqref{gq},
\begin{align*}
\ee\biggl(\sum_{i,j,k,k'} q_{jk}^e q_{ik'}^e \fpar{f}{q_{ik}^e} \fpar{g}{q_{jk'}^e}\biggr) &= \sum_{r=2}^n \sum_{x\in C_1, \, y\in C_r} \bigavg{W_{l_1\oplus_{x,y} l_r} \prod_{\substack{2\le t\le n\\t\ne r}} W_{l_t}} \\
&\quad + N\beta \sum_{p\in \cp^+(e)} \sum_{x\in C_1} \smallavg{W_{l_1\oplus_x p}W_{l_2}\cdots W_{l_n}}\, .
\end{align*}
Combining all of the above calculations and applying Theorem \ref{sdthm} (by first conditioning on $(Q_{e'})_{e'\ne e}$ and then taking  unconditional expectation on both sides), we get the identity claimed in the statement of Theorem \ref{mastern}. 
\end{proof}

We are now ready to prove Theorem \ref{mastersymm}.

\begin{proof}[Proof of Theorem \ref{mastersymm}]
Let $\fd^+_k(s)$ denote the set of loop sequences obtained from $s$ by positively deforming $l_k$. Define $\fd^-_k(s)$, $\fs^+_k(s)$, $\fs^-_k(s)$, $\ftw^+_k(s)$ and $\ftw^-_k(s)$ similarly. Let $\fst^+_k(s)$ be the set of loop sequences obtained by positively merging $l_k$ with some other $l_r$. Define $\fst^-_k(s)$ similarly. 

A simple but important observation that will be used in this proof is that if $(l_1, \ldots, l_n)$ is any representation of a loop sequence $s$ (and not necessarily the minimal representation), then
\[
\phi(s) = \frac{\smallavg{W_{l_1}W_{l_2}\cdots W_{l_n}}}{N^n}\, .
\]
This is because of our convention that $W_\emptyset = N$.

Now let $s$ be a non-null loop sequence with minimal representation $(l_1,\ldots, l_n)$. Clearly, 
\[
\phi(l_{\pi(1)},\ldots, l_{\pi(n)}) = \phi(l_1,\ldots, l_n)
\] 
for any permutation $\pi$ of $\{1,\ldots, n\}$. Moreover, by changing our arbitrary rule for defining the first edge of a cycle, it is evident that the master loop equation in Theorem \ref{mastern} should hold true, with appropriate modifications on the right-hand side, if $l_1$ is replaced by any other $l_k$ and $e$ is replaced by any edge in $l_k$. That is, take any $1\le k\le n$ and any edge $e$ in $l_k$. For each $1\le r\le n$, let $A_r(e)$ be the locations in $l_r$ where $e$ occurs and let $B_r(e)$ be the locations in $l_r$ where $e^{-1}$ occurs. Let $C_r(e):=A_r(e)\cup B_r(e)$ and let $m_k(e)$ be the size of $C_k(e)$.  Then the appropriate modification of Theorem \ref{mastern} gives:
\begin{align*}
(N-1) m_k(e) \phi_\beta(s)&= \textup{twisting term} + \textup{splitting term}\\
&\quad  + \textup{merger term} + \textup{deformation term}\, , 
\end{align*}
where the twisting term is given by
\begin{align*}
&\sum_{\substack{x, y\in A_k(e)\\x\ne y}} \phi(l_1,\ldots,l_{k-1},\propto_{x,y} l_k,l_{k+1},\ldots, l_n) + \sum_{\substack{x, y\in B_k(e)\\x\ne y}} \phi(l_1,\ldots,l_{k-1},\propto_{x,y} l_k,l_{k+1},\ldots, l_n)  \\
&\quad - \sum_{x\in A_k(e), \, y\in B_k(e)} \phi(l_1,\ldots, l_{k-1},\propto_{x,y} l_k, l_{k+1},\ldots,l_n)\\
&\qquad - \sum_{x\in B_k(e), \, y\in A_k(e)} \phi(l_1,\ldots, l_{k-1},\propto_{x,y} l_k, l_{k+1},\ldots,l_n)\, , 
\end{align*}
the splitting term is given by
\begin{align*}
&N\sum_{x\in A_k(e), \, y\in B_k(e)} \phi(l_1,\ldots,l_{k-1},\times_{x,y}^1 l_k,\times_{x,y}^2 l_k,l_{k+1},\ldots,l_n) \\
&\qquad + N\sum_{x\in B_k(e), \, y\in A_k(e)} \phi(l_1,\ldots,l_{k-1},\times_{x,y}^1 l_k,\times_{x,y}^2 l_k,l_{k+1},\ldots,l_n) \\
&\quad - N\sum_{\substack{x,y\in A_k(e)\\ x\ne y}} \phi(l_1,\ldots, l_{k-1},\times^1_{x,y} l_k,\times^2_{x,y} l_k,l_{k+1},\ldots,l_n)\\
&\quad  - N\sum_{\substack{x,y\in B_k(e)\\ x\ne y}} \phi(l_1,\ldots, l_{k-1},\times^1_{x,y} l_k,\times^2_{x,y} l_k,l_{k+1},\ldots,l_n)\, ,
\end{align*}
the merger term is given by
\begin{align*}
&\frac{1}{N}\sum_{k<r\le n}\sum_{x\in C_k(e), \, y\in C_r(e)} \phi(l_1,\ldots,l_{k-1}, l_k \ominus_{x,y} l_r,l_{k+1},\ldots, l_{r-1}, l_{r+1}, \ldots, l_n) \\
&\quad + \frac{1}{N}\sum_{1\le r< k}\sum_{x\in C_k(e), \, y\in C_r(e)} \phi(l_1,\ldots, l_{r-1}, l_{r+1},\ldots,l_{k-1}, l_k \ominus_{x,y} l_r,l_{k+1}, \ldots, l_n) \\
&\quad - \frac{1}{N}\sum_{k<r\le n}\sum_{x\in C_k(e), \, y\in C_r(e)}\phi(l_1,\ldots,l_{k-1}, l_k \oplus_{x,y} l_r,l_{k+1},\ldots, l_{r-1}, l_{r+1}, \ldots, l_n)\\
&\quad - \frac{1}{N}\sum_{1\le r< k}\sum_{x\in C_k(e), \, y\in C_r(e)} \phi(l_1,\ldots, l_{r-1}, l_{r+1},\ldots,l_{k-1}, l_k \oplus_{x,y} l_r,l_{k+1}, \ldots, l_n) \, .
\end{align*}
and the deformation term equals
\begin{align*}
&N\beta \sum_{p\in \cp^+(e)}\sum_{x\in C_k(e)}\phi(l_1,\ldots, l_{k-1}, l_k \ominus_{x}p,l_{k+1},\ldots,l_n)  \\
&\quad - N\beta \sum_{p\in \cp^+(e)}\sum_{x\in C_k(e)}\phi(l_1,\ldots, l_{k-1},l_k \oplus_{x}p,l_{k+1},\ldots,l_n) \, .
\end{align*}
Within the loop $l_k$, declare two edges to be equivalent if they are either equal or inverses of each other. If $e$ is an edge, then $C_k(e)$ is the equivalence class containing $e$. Construct a set $D_k$ by taking  one member from each equivalence class and sum both sides of the above equation over this set of edges. Since 
\[
\sum_{e\in D_k} m_k(e) = |l_k|\, ,
\]
this gives the equation
\begin{align*}
(N-1)|l_k|\phi(s) &= \sum_{s'\in \ftw^-_k(s)} \phi(s') - \sum_{s'\in \ftw^+_k(s)} \phi(s') + N\sum_{s'\in \fs^-_k(s)}\phi(s') - N\sum_{s'\in \fs^+_k}\phi(s')\\
&\quad + \frac{1}{N} \sum_{s'\in \fst^-_k(s)}\phi(s') - \frac{1}{N} \sum_{s'\in \fst^+_k(s)}\phi(s') + N\beta \sum_{s'\in \fd^-_k(s)}\phi(s') - N\beta \sum_{s'\in \fd^+_k(s)}\phi(s')\, .
\end{align*}
Summing both sides of the above equation over $k$, we get the equation claimed in the statement of the theorem. 
\end{proof}

\section{The master loop equation in the 't Hooft limit}\label{mm2sec}
For any $\Lambda$, $N$ and $\beta$, and any collection of loops $l_1,\ldots, l_n$ that are contained in $\Lambda$, define
\[
\phi_{\Lambda, N,\beta}(l_1,\ldots, l_n) := \frac{\smallavg{W_{l_1}\cdots W_{l_n}}}{N^n}\, .
\]
Since we defined $W_\emptyset = N$, the value of the above expression does not change if we insert some null loops into the collection $l_1,\ldots, l_n$ or delete some null loops from it. In other words, $\phi_{\Lambda, N, \beta}$ is well-defined as a function on loop sequences whose component loops are contained in $\Lambda$. 
%Given any $l_1,\ldots, l_n$, this is well-defined if $\Lambda$ is large enough.

Now note that since $|W_l|\le N$ for any loop $l$, therefore $|\phi_{\Lambda, N,\beta}(s)|\le 1$ for any loop sequence $s$ contained in $\Lambda$. Moreover, if $\Lambda_N$ is a sequence increasing to $\zz^d$, any loop sequence $s$ is eventually contained in $\Lambda_N$ as $N\ra\infty$. By a standard diagonal argument, there exists a subsequence along which the limit of $\phi_{\Lambda_N, N,\beta}(s)$ exists for every loop sequence $s$. The following theorem gives a  recursive relation for any such limit point. We may call this the master loop equation in the 't Hooft limit. 
\begin{thm}\label{master}
Let $\Lambda_N$ be a sequence of sets increasing to $\zz^d$, and take a subsequence of $N$'s such that along this subsequence, the limit of $\phi_{\Lambda_N, N,\beta}(s)$ exists for every loop sequence $s$. Call this limit $\phi_\beta(s)$. Take any loop sequence $s$ with minimal representation $(l_1,\ldots, l_n)$. Let $e$, $m$, $A_1$, $B_1$ and $C_1$ be as in Theorem \ref{mastern}. Then 
\begin{align*}
m \phi_\beta(s) &= \textup{splitting term} + \textup{deformation term}\, ,
\end{align*}
where the splitting term is given by
\begin{align*}
&\sum_{x\in A_1, \, y\in B_1} \phi_\beta(\times_{x,y}^1 l_1, \times_{x,y}^2 l_1,l_2,\ldots,l_n)+\sum_{x\in B_1, \, y\in A_1} \phi_\beta(\times_{x,y}^1 l_1, \times_{x,y}^2 l_1,l_2,\ldots,l_n) \\
&\qquad- \sum_{\substack{x,y\in A_1\\ x\ne y}} \phi_\beta(\times^1_{x,y} l_1,\times^2_{x,y} l_1,l_2,\ldots,l_n)  - \sum_{\substack{x,y\in B_1\\ x\ne y}} \phi_\beta(\times^1_{x,y} l_1,\times^2_{x,y} l_1,l_2,\ldots,l_n)\, ,
\end{align*}
and the deformation term equals
\begin{align*}
&\beta \sum_{p\in \cp^+(e)}\sum_{x\in C_1}\phi_\beta(l_1 \ominus_{x}p,l_2,\ldots,l_n)  - \beta \sum_{p\in \cp^+(e)}\sum_{x\in C_1}\phi_\beta(l_1 \oplus_{x}p,l_2,\ldots,l_n) \, .
\end{align*}
\end{thm}
\begin{proof}
The proof is a direct application of Theorem \ref{mastern} and the fact that $|W_l|\le N$ for any loop~$l$. Simply divide both sides of the finite $N$ master loop equation (as given in the statement of Theorem~\ref{mastern}) by $N^{n+1}$. It is an easy consequence of the bound $|W_l|\le N$ that the merger and twisting terms vanish as $N$ goes to infinity. The left-hand side of the finite $N$ master loop equation, upon dividing by $N^{n+1}$, tends to $m\phi_\beta(s)$ as $N$ tends to infinity along the subsequence. The splitting and deformation terms tend to the respective terms displayed above. 
\end{proof}

%\section{Existence and uniqueness of solution}\label{unique}

The notable thing about Theorem \ref{master} is that it is true irrespective of the value of $\beta$. However, it is a theorem about subsequential limits. To prove that $\phi_{\Lambda_N, N,\beta}(s)$ converges to a limit as $N\ra\infty$, it suffices to show that there is a unique set of solutions for the master loop equation. This turns out to be true if $|\beta|$ is small enough. This is the content of the next theorem.
\begin{thm}\label{master2}
Given any $L\ge 1$, there exists $\beta_0(L,d)>0$ such that if $|\beta| \le\beta_0(L,d)$, then there is a unique function $\phi_\beta:\cs \ra \rr$ such that (a) $\phi_\beta(\emptyset)=1$, (b) $|\phi_\beta(s)|\le L^{|s|}$ for all $s$,  and (c) $\phi_\beta$ satisfies the master loop equation of Theorem~\ref{master}. Consequently, there exists $\beta_0(d)>0$ such that for $|\beta|\le \beta_0(d)$, $\phi_{\Lambda_N, N,\beta}(s)$ converges to a limit $\phi_\beta(s)$ as $N\ra\infty$ for every loop sequence $s$.
\end{thm} 
Actually, to prove the convergence of $\phi_{\Lambda_N, N,\beta}$ we need only the case $L=1$. The general case will be needed in a later section. 

We need a few lemmas about properties of loop operations before proving Theorem \ref{master2}.
\begin{lmm}\label{split1}
Let $l$ be a non-null loop and suppose that $x$ and $y$ are two distinct locations in $l$ such that $l$ admits a positive splitting at $x$ and $y$. Let $l_1 := \times^1_{x,y} l$ and $l_2 := \times^2_{x,y} l$. Then $l_1$ and $l_2$ are non-null loops, $|l_1|\le |l| - |y-x|$, and $|l_2|\le |y-x|$. 
\end{lmm}
\begin{proof}
Without loss of generality suppose that $x< y$. Write $l = aebec$, where the two $e$'s occur at locations $x$ and $y$. Then $l_1 = [aec]$ and $l_2 = [be]$. From this it is clear that 
\[
|l_1|\le |l| - |be| = |l|-(y-x)
\]
and 
\[
|l_2|\le |be| = y-x\, .
\]
To prove that $l_2$ is non-null, observe that since $l$ is a loop, the path $be$ has no interior backtracks. Moreover, $eb$ does not have any interior backtracks either. Therefore the cycle $be$ has no backtracks, which proves that $l_2 = be$, which has length strictly bigger than zero.

Similarly, since $l$ is a loop, the paths $ae$ and $ec$ have no interior backtracks, and the first edge of the path $ae$ cannot be the inverse of the last edge of the path $ec$. Thus the cycle $aec$ has no backtracks and so $l_1 = aec$, which has length strictly bigger than zero.
\end{proof}

\begin{lmm}\label{one}
If $l$ is a cycle of length at least four whose nonbacktracking core is null, then $l$ has at least two backtracks.
\end{lmm}
\begin{proof}
The proof is by induction on $|l|$. If $|l|=4$, then it is easy to verify using the nature of $\zz^d$ that there are only two possibilities. One is that $l$ is a plaquette, which has no backtracks. The other is that $l$ is the cycle corresponding to a closed path of the form $e_1 e_2 e_2^{-1}e_1^{-1}$, which has two backtracks. This proves the claim when $|l|=4$.

Now suppose that the claim has been proved for all cycles of length less than $n$ and let $l$ be a cycle of length $n$, where $n>4$. Since $\zz^d$ is a bipartite graph, cycles can have only even lengths; so we may take $n\ge 6$. 

If $[l]$ is null, then $l$ must have at least one backtrack. Choose a representative $e_1e_2\cdots e_n$ of $l$ that has a backtrack at location $n$, so that $e_n = e_1^{-1}$. Remove this backtrack to get $l' = e_2\cdots e_{n-1}$. By Lemma \ref{core} $[l']$ is also null, so the induction hypothesis implies that $l'$ must have at least two backtracks. If  $e_{n-1}=e_2^{-1}$, this would account for one of the backtracks. But the other backtrack, wherever it is, must have also been a backtrack in $l$. This shows that $l$ must be having at least two backtracks, completing the proof.
\end{proof}

\begin{lmm}\label{split2}
Let $l$ be a non-null loop and suppose that $x$ and $y$ are two distinct locations in $l$ such that $l$ admits a negative splitting at $x$ and $y$. Let $l_1 := \times^1_{x,y} l$ and $l_2 := \times^2_{x,y} l$. Then $l_1$ and $l_2$ are non-null loops, $|l_1|\le |l| - |y-x|-1$, and $|l_2|\le |y-x|-1$. 
\end{lmm}
\begin{proof}
Without loss of generality suppose that $x< y$. Write $l = aebe^{-1}c$, where $e$ and $e^{-1}$ occur at locations $x$ and $y$ respectively. Then $l_1 = [ac]$ and $l_2 = [b]$. From this it is clear that 
\[
|l_1|\le |l| - |ebe^{-1}| = |l|-(y-x)-1
\]
and 
\[
|l_2|\le |b| = y-x-1\, .
\]
Since $b$ is sandwiched between $e$ and $e^{-1}$ in $l$, it must be a closed path. Since $l$ is a loop and $b$ is a closed path, therefore $|b|\ge 4$; because otherwise, $l$ would have a backtrack.  Again since $l$ is a loop, the closed path $b$ cannot have any interior backtracks (but may have a terminal backtrack). In particular, $b$ can have at most one backtrack. Thus by Lemma \ref{one}, $l_2 = [b]$ cannot be null. 

Similarly, note that $a$ and $c$ cannot both  be null, because otherwise $l$ has a backtrack. For the same reason, the paths $a$ and $c$ cannot have interior backtracks, nor can the last edge of $c$ be the inverse of the first edge of $a$. Therefore the only possible backtrack in $ac$ may be caused by the last edge of $a$ being the inverse of the first edge of $c$. It is easy to see that $ac$ must be a closed path, and therefore it must have length $\ge 4$, because otherwise $l$ has a backtrack. Thus by Lemma \ref{one}, $l_1 = [ac]$ cannot be null. 
\end{proof}

\begin{lmm}\label{deform}
If $l$ and $l'$ are two loops that can be merged together at locations $x$ and $y$, then $|l\oplus_{x,y} l'|$ and $|l\ominus_{x,y} l'|$ are both bounded above by $|l|+|l'|$. 
\end{lmm}
\begin{proof}
This is a straightforward verification using the definitions of positive and negative mergers and the fact that the length of a nonbacktracking core is always less than or equal to the length of the original closed path. 
\end{proof}

We now make two important definitions that will be useful throughout the remainder of this manuscript. We have already defined the `length' of a loop sequence. Define the `size' of a non-null sequence $s$ with minimal representation $(l_1,\ldots, l_n)$ as 
\[
\#s := n\, ,
\]
and the `index' of $s$ as 
\[
\iota(s) := |s| - \#s\,.
\]
The size and the index of the null loop sequence are define to be zero. The following two lemmas give two useful properties of the index.
\begin{lmm}\label{iota1}
For any $s$, $\iota(s)\ge 3\#s$. In particular, the index of any non-null loop sequence is strictly  positive. 
\end{lmm}
\begin{proof}
The smallest non-null loop in $\zz^d$ is a plaquette, which has four edges. This proves that $|s|\ge 4\#s$. Consequently, $\iota(s)\ge 4\#s - \#s = 3\#s$. 
\end{proof}
\begin{lmm}\label{iota2}
If $s'$ is obtained from $s$ by a splitting operation, then $\iota(s')<\iota(s)$.
\end{lmm}
\begin{proof}
Lemmas \ref{split1} and \ref{split2} show that $\#s' = \#s + 1$.  Moreover, these lemmas also show that $|s'|\le |s|$. Thus, $\iota(s')<\iota(s)$. 
\end{proof}
 After this initial preparation, we are now ready to prove Theorem \ref{master2}. 

\begin{proof}[Proof of Theorem \ref{master2}]
Fix $\beta\in \rr$ and suppose that $\phi_\beta$ and $\psi_\beta$ are two functions that satisfy the conditions (a), (b) and (c) of the theorem statement. For each $s\in \cs$, let
\[
T(s) := |\phi_\beta(s)-\psi_\beta(s)|\, .
\]
If $(l_1,\ldots, l_n)$ is the minimal representation of $s$ and $\delta_i := |l_i|$, the vector $\delta(s):= (\delta_1,\ldots, \delta_n)$ will be called the `degree vector' of $s$. The degree vector of the null loop sequence is the `null sequence' $\emptyset$ that does not contain any element.

Let $\Delta$ be the set of all finite sequences of integers, including the null sequence. Clearly, any degree vector is an element of $\Delta$, but not all elements of $\Delta$ are degree vectors. In particular, if any component of $\delta$ is nonpositive, then $\delta$ is not a degree vector.  Given a vector $\delta = (\delta_1,\ldots, \delta_n)\in \Delta$, define
\[
|\delta| := \sum_{i=1}^n \delta_i\, , \ \ \#\delta := n\, , \ \ \text{and} \ \ \iota(\delta) := |\delta|-\#\delta\, .
\]
All of the above quantities are defined to be zero for the empty sequence. Note that $\iota(s)=\iota(\delta(s))$. 

Given two non-null elements $\delta = (\delta_1,\ldots, \delta_n)\in \Delta$ and $\delta' = (\delta_1',\ldots, \delta_m')\in \Delta$, we will say that $\delta \le \delta'$ if $m=n$ and $\delta_i\le\delta_i'$ for each $i$.

Let $\Delta^+$ be the subset of $\Delta$ consisting of all $\delta$ whose components are all $\ge 4$. In particular if $s$ is a non-null loop sequence, then $\delta(s)\in\Delta^+$. 

For each $\delta\in \Delta^+$, define
\[
D(\delta) = \sup_{s\in \cs\, : \, \delta(s)\le\delta} T(s)\, .
\]
(Note that for any $\delta\in \Delta^+$, there is at least one non-null $s$ such that $\delta(s)\le \delta$.) If $\delta\in \Delta\backslash \Delta^+$, let $D(\delta)=0$. For each $\lambda > 0$, define
\[
F(\lambda) := \sum_{\delta\in \Delta^+} \lambda^{\iota(\delta)} D(\delta) =  \sum_{\delta\in \Delta} \lambda^{\iota(\delta)} D(\delta)\, .
\]
We claim that  if $\lambda< (2L)^{-4/3}$, then $F(\lambda)<\infty$. To see this, first note that
\begin{align*}
F(\lambda) &= \sum_{r=1}^\infty \sum_{n=1}^r\sum_{\substack{\delta\in \Delta^+\, : \, |\delta|= r,\\
\#\delta = n}} \lambda^{r-n} D(\delta)\, .
\end{align*}
Given $r$ and $n$, the number of $\delta\in \Delta^+$ such that $|\delta|=r$ and $\#\delta=n$ is bounded above by the number of ways of choosing positive integers $\delta_1,\ldots,\delta_n$ such that $\sum\delta_i = r$. This is equal to the number of ways of choosing a strictly increasing sequence of  $n+1$ numbers from the set $\{0,1,\ldots, r\}$ with the restriction that the first number is $0$ and the last number is $r$, since $\delta_1,\ldots,\delta_n$ may be obtained as the successive differences between these numbers. This shows that% the number of ways of choosing degree vectors $\delta$ such that $|\delta| = r$ and $\#\delta = n$ is bounded above by
\[
|\{\delta\in \Delta^+ : |\delta| = r, \, \#\delta = n\}| \le {r\choose n-1}\, .
\]
Next, note that if $\delta\in \Delta^+$, then $r\ge 4n$, and therefore $r-n \ge 3r/4$. Finally, note that by the condition (b) in the statement of the theorem, $D(\delta)\le 2L^{|\delta|}$ for all $\delta$. Combining all of the above, we get that for any $\lambda <1$, 
\begin{align*}
\sum_{r=1}^\infty \sum_{n=1}^r\sum_{\substack{\delta\in \Delta^+\, : \, |\delta|= r,\\
\#\delta = n}} \lambda^{r-n} D(\delta) &\le \sum_{r=1}^\infty\sum_{n=1}^r 2L^r\lambda^{3r/4} {r\choose n-1}\\
&\le \sum_{r=1}^\infty 2L^r\lambda^{3r/4} 2^r\, .
\end{align*}
This proves the claim that if $\lambda < (2L)^{-4/3}$ then $F(\lambda)<\infty$.

Now take any $\delta = (\delta_1,\ldots,\delta_n)\in \Delta^+$ such that $|\delta|=r$. Let $s$ be a non-null loop sequence with minimal representation $(l_1,\ldots, l_n)$, such that $\delta(s)\le \delta$. Let $e$, $m$, $A_1$, $B_1$ and $C_1$ be as in the statement of Theorem \ref{mastern}.  Then by condition (c), 
\begin{align}
T(s) &\le \frac{1}{m}\sum_{x\in A_1,\, y\in B_1} T(\times_{x,y}^1l_1, \times_{x,y}^2l_1,l_2\ldots,l_n)+ \frac{1}{m}\sum_{x\in B_1,\, y\in A_1} T(\times_{x,y}^1l_1, \times_{x,y}^2l_1,l_2\ldots,l_n) \nonumber\\
&\quad + \frac{1}{m}\sum_{\substack{x,y\in A_1\\ x\ne y}} T(\times^1_{x,y} l_1,\times^2_{x,y} l_1,l_2,\ldots,l_n) + \frac{1}{m}\sum_{\substack{x,y\in B_1\\ x\ne y}} T(\times^1_{x,y} l_1,\times^2_{x,y} l_1,l_2,\ldots,l_n)  \nonumber \\
&\quad + \frac{|\beta|}{m} \sum_{p\in \cp^+(e)}\sum_{x\in C_1}T(l_1 \ominus_{x}p,l_2,\ldots,l_n)+\frac{|\beta|}{m} \sum_{p\in \cp^+(e)}\sum_{x\in C_1}T(l_1 \oplus_{x}p,l_2,\ldots,l_n) \, .\label{tt1}
\end{align}
We will bound the terms on the right-hand side one by one, freely using the fact that $D(\delta)=0$ if $\delta\in \Delta\backslash \Delta^+$. By Lemma \ref{split1},
\begin{align}
&\frac{1}{m}\sum_{\substack{x,y\in A_1\\x\ne y}} T(\times_{x,y}^1 l_1, \times^2_{x,y} l_1, l_2,\ldots, l_n)\nonumber\\
&\le \frac{1}{m}\sum_{x\in A_1} \sum_{y\in A_1\backslash\{x\}} D(\delta_1-|y-x|, |y-x|, \delta_2,\ldots,\delta_n)\nonumber\\
&\le \frac{2}{m}\sum_{x\in A_1} \sum_{k=1}^{\infty} D(\delta_1-k, k, \delta_2,\ldots,\delta_n)\le 2\sum_{k=1}^{\infty} D(\delta_1-k, k, \delta_2,\ldots,\delta_n)\, .\label{tt2}
\end{align}
The same bound holds if $A_1$ is replaced by $B_1$. Next, by Lemma \ref{split2},
\begin{align}
&\frac{1}{m}\sum_{\substack{x\in A_1, \, y\in B_1}} T(\times_{x,y}^1 l_1, \times^2_{x,y} l_1, l_2,\ldots, l_n)\nonumber\\
&\le \frac{1}{m}\sum_{x\in A_1} \sum_{y\in B_1} D(\delta_1-|x-y|-1, |x-y|-1, \delta_2,\ldots,\delta_n)\nonumber\\
&\le \frac{2}{m}\sum_{x\in A_1} \sum_{k=1}^{\infty} D(\delta_1-k-1, k-1, \delta_2,\ldots,\delta_n)\le 2\sum_{k=1}^{\infty} D(\delta_1-k-1, k-1, \delta_2,\ldots,\delta_n)\, ,\label{tt3}
\end{align}
and the same bound holds if $A_1$ and $B_1$ are swapped. 
If $p\in \cp^+(e)$ and $x\in C_1$, then by Lemma~\ref{deform}, $|l_1\ominus_x p|\le \delta_1+4$. Thus, if $l_1\ominus_xp$ is non-null, then 
\begin{align*}
T(l_1\ominus_xp, l_2,\ldots, l_n) \le D(\delta_1+4,\delta_2,\ldots,\delta_n)\, .
\end{align*}
On the other hand, if $l_1\ominus_x p$ is null, then
\begin{align*}
T(l_1\ominus_xp, l_2,\ldots, l_n) &= T(l_2,\ldots, l_n)\le D(\delta_2,\ldots,\delta_n)\, .
\end{align*}
Note that this is true even if $n=1$, in which case $(\delta_2,\ldots, \delta_n)$ is the empty sequence, because by condition (a) in the statement of the theorem, $T(\emptyset)=0$, and $D(\emptyset)=0$ by definition. Combining the two cases listed above, we get
\begin{align}\label{tt4}
T(l_1\ominus_xp, l_2,\ldots, l_n)\le D(\delta_1+4, \delta_2,\ldots,\delta_n) + D(\delta_2,\ldots,\delta_n)\, .
\end{align}
The same bound holds for $T(l_1\oplus_x p, l_2,\ldots, l_n)$. 

For each $k\ge 1$, define two maps $\theta_k$ and $\eta_k$ from $\Delta^+$ into $\Delta$ as 
\begin{align*}
\theta_k(\delta_1,\ldots,\delta_n) &:= (\delta_1-k,k,\delta_2,\ldots, \delta_n)\, ,\\
\eta_k(\delta_1,\ldots,\delta_n) &:= (\delta_1-k-1, k-1,\delta_2,\ldots,\delta_n)\, .
\end{align*}
Additionally, define maps $\alpha$ and $\gamma$ from $\Delta^+$ into $\Delta$ as 
\begin{align*}
\alpha(\delta_1,\ldots,\delta_n) &:= (\delta_1+4,\delta_2,\ldots, \delta_n)\, ,\\
\gamma(\delta_1,\ldots,\delta_n) &:= (\delta_2,\ldots,\delta_n)\, .
\end{align*}
(When $n=1$, $\gamma(\delta_1)=\emptyset$.) Then by \eqref{tt1}, \eqref{tt2}, \eqref{tt3}, \eqref{tt4} and the fact that $|\cp^+(e)|\le2(d-1)$, 
\begin{align*}
T(s) &\le 4\sum_{k=1}^\infty D(\theta_k(\delta)) + 4\sum_{k=1}^\infty D(\eta_k(\delta)) + 4|\beta| d D(\alpha(\delta)) + 4|\beta| d D(\gamma(\delta))\, .
\end{align*}
Since this bound holds for every non-null $s$ such that $\delta(s)\le \delta$, therefore the right-hand side is an upper bound for $D(\delta)$. Thus,
\begin{align}
F(\lambda) &\le \sum_{\delta\in \Delta^+} \lambda^{\iota(\delta)}\biggl(4\sum_{k=1}^\infty D(\theta_k(\delta)) + 4\sum_{k=1}^\infty D(\eta_k(\delta)) + 4|\beta| d D(\alpha(\delta)) + 4|\beta| d D(\gamma(\delta))\biggr)\, .\label{ff1}
\end{align}
Now note that the maps $\theta_1,\theta_2,\ldots$ are all injective, and their ranges are disjoint since the second component of any vector that is in the image of $\theta_k$ must be equal to $k$. Moreover, $\iota(\theta_k(\delta)) = \iota(\delta) - 1$ for any  $\delta\in \Delta^+$. Thus,
\begin{align}
\sum_{\delta\in \Delta^+}\sum_{k=1}^\infty \lambda^{\iota(\delta)} D(\theta_k(\delta)) &= \sum_{\delta\in \Delta^+}\sum_{k=1}^\infty \lambda^{\iota(\theta_k(\delta))+1} D(\theta_k(\delta))\nonumber\\
&\le \sum_{\delta\in \Delta} \lambda^{\iota(\delta)+1} D(\delta) = \lambda F(\lambda)\, . \label{ff2}
\end{align}
Similarly, note that the maps $\eta_1,\eta_2,\ldots$ are all injective, and their ranges are disjoint. Moreover, $\iota(\eta_k(\delta)) = \iota(\delta) -3$. Thus,
\begin{align}
\sum_{\delta\in \Delta^+}\sum_{k=1}^\infty \lambda^{\iota(\delta)} D(\eta_k(\delta)) &= \sum_{\delta\in \Delta^+}\sum_{k=1}^\infty \lambda^{\iota(\eta_k(\delta))+3} D(\eta_k(\delta))\nonumber\\
&\le \sum_{\delta\in \Delta} \lambda^{\iota(\delta)+3} D(\delta) = \lambda^3 F(\lambda)\, .\label{ff3}
\end{align}
Next, note that the map $\alpha$ is injective and $\iota(\alpha(\delta)) = \iota(\delta) + 4$. Thus,
\begin{align}
\sum_{\delta\in \Delta^+} \lambda^{\iota(\delta)} D(\alpha(\delta)) &= \sum_{\delta\in \Delta^+} \lambda^{\iota(\alpha(\delta))-4} D(\alpha(\delta))\nonumber\\
&\le \sum_{\delta\in \Delta} \lambda^{\iota(\delta)-4} D(\delta) =\lambda^{-4} F(\lambda)\, .\label{ff4}
\end{align}
Finally, note that for any $(\delta_1,\ldots,\delta_n)\in \Delta^+$, 
\[
\gamma^{-1}(\delta_1,\ldots, \delta_n) \subseteq\{(k, \delta_1,\ldots,\delta_n) : k\ge 1\}\, .
\]
Moreover, for any $\delta\in \Delta^+$, either $\gamma(\delta)=\emptyset$ or $\gamma(\delta)\in \Delta^+$. Thus, 
\begin{align}
\sum_{\delta\in \Delta^+} \lambda^{\iota(\delta)} D(\gamma(\delta)) &= \sum_{\delta'\in \Delta^+} \sum_{\delta\in \gamma^{-1}(\delta')}\lambda^{\iota(\delta)} D(\delta')\nonumber\\
&\le \sum_{\delta'\in \Delta^+} \sum_{k=1}^\infty \lambda^{\iota(\delta') + k -1} D(\delta')= \frac{F(\lambda)}{1-\lambda}\, .\label{ff5}
\end{align} 
Combining \eqref{ff1}, \eqref{ff2}, \eqref{ff3}, \eqref{ff4} and \eqref{ff5}, we get
\begin{align*}
F(\lambda)&\le \biggl(4\lambda^3 + 4\lambda + \frac{4|\beta| d}{\lambda^4} + \frac{4|\beta| d}{1-\lambda}\biggr)F(\lambda)\, .
\end{align*}
Suppose that $\lambda$ is so small that $\lambda < (2L)^{-4/3}$ and $4\lambda^3 + 4\lambda < 1$. Then it is possible to choose $\beta_0$ small enough so that if $|\beta| \le \beta_0$, then the coefficient of $F(\lambda)$ on the right is strictly less than~$1$. Since $F(\lambda)$ is nonnegative and finite (since $\lambda <(2L)^{-4/3}$), this would imply that $F(\lambda)=0$. Finally, to complete the proof, note that since $\delta(s)\in \Delta^+$ for any non-null $s$, therefore $T(s)\le D(\delta(s)) \le \lambda^{-\iota(\delta(s))} F(\lambda)=0$. This proves the uniqueness claim in the statement of the theorem. To prove the convergence of $\phi_{\Lambda_N, N, \beta}(s)$, recall that these functions are uniformly bounded  by $1$ in absolute value. Therefore an application of the uniqueness part of the theorem with $L=1$, together with a simple subsequence argument, completes the proof. 
\end{proof}
We end this section with a symmetrized version of the limiting master loop equation. This will be useful for certain purposes later. %Given a loop sequence $s$, let $\fs^+(s)$ be the set of positive splittings of $s$ and let $\fs^-(s)$ be the set of negative splittings of $s$. Similarly, let $\fd^+(s)$ be the set of positive deformations of $s$ and let $\fd^-(s)$ be the set of negative deformations of $s$. 
%If $(l_1,\ldots, l_n)$ is the minimal representation of $s$, let $\fs^+_k(s)$ be the set of loop sequences produced by splitting $l_k$, and define $\fs^-_k(s)$, $\fd^+_k(s)$ and $\fd^-_k(s)$ similarly. 
\begin{thm}\label{master3}
Let $\beta_0(d)$ and $\phi_\beta$ be as in Theorem \ref{master2}. Then for any non-null loop sequence $s$ and $|\beta|\le \beta_0(d)$, \begin{align*}
|s|\phi_\beta(s) &=  \sum_{s'\in \fs^-(s)} \phi_\beta(s')-  \sum_{s'\in \fs^+(s)} \phi_\beta(s') + \beta \sum_{s'\in \fd^-(s)} \phi_\beta(s')- \beta \sum_{s'\in \fd^+(s)} \phi_\beta(s')\, .
\end{align*}
\end{thm}
\begin{proof}
This is proved by simply dividing both sides in Theorem \ref{mastersymm} by $N$ and letting $N$ tend to infinity. 
\end{proof}

\section{Power series expansion}\label{power}
The goal of this section is to show that the function $\phi_\beta$ of Theorem \ref{master2} has a convergent power series expansion in $\beta$. This is a technical step that is required for the proof of Theorem \ref{mainthm}. We will start by defining a collection of real numbers $a_k(s)$, where $k$ runs over nonnegative integers and $s$ runs over loop sequences. The definition will involve a nested induction: Having defined $a_i(s)$ for all $i<k$ and all $s$, we will define $a_k(s)$ by induction over $\iota(s)$. Note that we can do induction over the index because by Lemma \ref{iota1}, the index of a non-null loop sequence is always a positive integer.

First, define $a_0(\emptyset)=1$ and $a_0(s) = 0$ for every non-null $s$. Next take any $k\ge 1$ and suppose that $a_i(s)$ has been defined for every $i< k$ and every $s$. Let $a_k(\emptyset)=0$. Take any non-null $s$ with minimal representation $(l_1,\ldots, l_n)$. Let $e$, $m$, $A_1$, $B_1$ and $C_1$ be as in the statement of Theorem~\ref{mastern}. Suppose that $a_k(s')$ has been defined for every $s'$ with $\iota(s')<\iota(s)$. In particular, Lemma \ref{iota2}  shows that if $s'$ is obtained by splitting $s$, then $\iota(s') < \iota(s)$ and hence $a_k(s')$ has already been defined. Moreover, if $s$ has the smallest possible index, then again by Lemma~\ref{iota2}, it cannot be split. Therefore the following definition makes sense:
\begin{align*}
a_k(s) &:= \frac{1}{m}\sum_{x\in A_1, \, y\in B_1} a_k(\times_{x,y}^1 l_1, \times_{x,y}^2 l_1,l_2,\ldots,l_n) + \frac{1}{m}\sum_{x\in B_1, \, y\in A_1} a_k(\times_{x,y}^1 l_1, \times_{x,y}^2 l_1,l_2,\ldots,l_n) \\
&\quad  - \frac{1}{m}\sum_{\substack{x,y\in A_1\\ x\ne y}} a_k(\times^1_{x,y} l_1,\times^2_{x,y} l_1,l_2,\ldots,l_n) - \frac{1}{m}\sum_{\substack{x,y\in B_1\\ x\ne y}} a_k(\times^1_{x,y} l_1,\times^2_{x,y} l_1,l_2,\ldots,l_n)  \\
&\quad + \frac{1}{m} \sum_{p\in \cp^+(e)}\sum_{x\in C_1}a_{k-1}(l_1 \ominus_{x}p,l_2,\ldots,l_n)  - \frac{1}{m} \sum_{p\in \cp^+(e)}\sum_{x\in C_1}a_{k-1}(l_1 \oplus_{x}p,l_2,\ldots,l_n) \, .
%\frac{2}{|s|}\sum_{s'\in \fs^-(s)} a_k(s') - \frac{2}{|s|}\sum_{s'\in \fs^+(s)} a_k(s') + \frac{1}{|s|}\sum_{s'\in \fd^-(s)} a_{k-1}(s') - \frac{1}{|s|}\sum_{s'\in \fd^+(s)} a_{k-1}(s')\, .
\end{align*}
A crucial part of the argument involves the use of Catalan numbers. This is inspired by a similar use of Catalan numbers in \cite{collinsetal09}. Recall the definition of the Catalan numbers: $C_0=1$, and for~$i\ge 1$, 
\[
C_i = \frac{1}{i+1}{2i\choose i} = {2i\choose i} - {2i \choose i+1}\,.
\]
We will use a well known recursion relation for Catalan numbers: For each $i\ge 0$, 
\begin{equation}\label{catalan}
C_{i+1} = \sum_{j=0}^i C_j C_{i-j}\,.
\end{equation}
We will also use the facts that $C_i$ is increasing in $i$ and that 
\begin{equation}\label{catalan2}
C_{i+1}\le 4C_i
\end{equation}
for each $i\ge 0$.
\begin{lmm}\label{aklmm}
There is a constant $K(d)$ such that if $s$ is a loop sequence and $\delta = (\delta_1,\ldots,\delta_n)$ is its degree vector, then
\[
|a_k(s)| \le K(d)^{5k+\iota(\delta)} C_{\delta_1-1}\cdots C_{\delta_n-1}\, ,
\]
where $C_i$ is the $i^{\mathrm{th}}$ Catalan number. The product of Catalan numbers is interpreted as $1$ when~$s=\emptyset$.
\end{lmm}
\begin{proof}
The number $K=K(d)$ will be chosen at the end of the proof. Assuming that $K\ge 1$ is given, we will prove the claim by the same nested induction that was used to define $a_k(s)$.

Clearly the claim is true when $k=0$, irrespective of $s$. It also holds for any $k$ when $s=\emptyset$. Take any $k\ge 1$ and non-null $s$, and suppose that we have proved that the bound on $|a_i(s)|$ holds for all $i<k$ and $s\in \cs$, and the bound on $|a_k(s')|$ holds for all $s'$ with $\iota(s') <\iota(s)$. As usual let $(l_1,\ldots, l_n)$ be the minimal representation of $s$ and let $e$, $m$, $A_1$, $B_1$ and $C_1$ be as in the statement of Theorem~\ref{mastern} (not to be confused with the Catalan number $C_1$, which will be clear from the context).

By the induction hypothesis, Lemma \ref{split2}, the identity \eqref{catalan} and the monotonicity of Catalan numbers,
\begin{align}
&\frac{1}{m}\sum_{x\in A_1,\, y\in B_1} |a_k(\times_{x,y}^1 l_1,\times_{x,y}^2l_1,l_2,\ldots, l_n) |\nonumber\\
&\le \frac{1}{m}\sum_{x\in A_1,\, y\in B_1}K^{5k+\iota(\delta)-3} C_{\delta_1 - |x-y|-2} C_{|x-y|-2}C_{\delta_2-1}\cdots C_{\delta_n-1}\nonumber\\
&\le \frac{2}{m}\sum_{x\in A_1} \sum_{r=2}^{\delta_1-2} K^{5k+\iota(\delta)-3}C_{\delta_1 - r-2} C_{r-2}C_{\delta_2-1}\cdots C_{\delta_n-1}\nonumber\\
&\le 2K^{5k+\iota(\delta)-3}  C_{\delta_1-1}\cdots C_{\delta_n-1}\, .\label{cc1}
\end{align}
Similarly, by the induction hypothesis, Lemma \ref{split1} and the identity \eqref{catalan}, 
\begin{align}
&\frac{1}{m}\sum_{\substack{x,y\in A_1\\ x\ne y}} |a_k(\times_{x,y}^1 l_1,\times_{x,y}^2l_1,l_2,\ldots, l_n) |\nonumber\\
&\le \frac{1}{m}\sum_{\substack{x,y\in A_1\\ x\ne y}}K^{5k+\iota(\delta)-1} C_{\delta_1 - |x-y|-1} C_{|x-y|-1}C_{\delta_2-1}\cdots C_{\delta_n-1}\nonumber\\
&\le \frac{2}{m}\sum_{x\in A_1} \sum_{r=1}^{\delta_1-1} K^{5k+\iota(\delta)-1}C_{\delta_1 - r-1} C_{r-1}C_{\delta_2-1}\cdots C_{\delta_n-1}\nonumber\\
&\le 2K^{5k+\iota(\delta)-1} C_{\delta_1-1}\cdots C_{\delta_n-1}\, .\label{cc2}
\end{align}
The same bounds holds if $A_1$ and $B_1$ are swapped. Next, take any $p\in \cp^+(e)$ and $x\in C_1$. Let $s' := (l_1\oplus_x p, l_2,\ldots, l_n)$ and let $\delta'$ be the degree vector of $s'$. By the induction hypothesis, Lemma~\ref{deform}, the monotonicity of Catalan numbers, and the inequality \eqref{catalan2}, 
\begin{align}
|a_{k-1}(l_1\oplus_{x} p , l_2,\ldots, l_n)| &\le K^{5(k-1)+\iota(\delta')} C_{\delta_1+3} C_{\delta_2-1}\cdots C_{\delta_n-1}\nonumber\\
&\le K^{5k+ \iota(\delta)-1}4^{4}C_{\delta_1-1}\cdots C_{\delta_n-1}\, .\label{cc3} 
\end{align}
The same bound holds if $\oplus$ is replaced by $\ominus$. Using \eqref{cc1}, \eqref{cc2} and \eqref{cc3} and the definition of $a_k(s)$, we get 
\begin{align*}
|a_k(s)| &\le (4K^{-3} + 4K^{-1} + 1024 d K^{-1} ) K^{5k + \iota(\delta)} C_{\delta_1-1}\cdots C_{\delta_n-1}\, .
\end{align*} 
Choosing $K$ so large that the term inside the bracket is $\le 1$,  we are done. 
\end{proof}
\begin{lmm}\label{a0}
Take any non-null loop sequence $s$ with minimal representation $(l_1,\ldots, l_n)$. Let $e$, $m$, $A_1$, $B_1$ and $C_1$ be as in the statement of Theorem \ref{mastern}. Then 
\begin{align*}
a_0(s) &= \frac{1}{m}\sum_{x\in A_1, \, y\in B_1} a_0(\times_{x,y}^1 l_1, \times_{x,y}^2 l_1,l_2,\ldots,l_n) + \frac{1}{m}\sum_{x\in B_1, \, y\in A_1} a_0(\times_{x,y}^1 l_1, \times_{x,y}^2 l_1,l_2,\ldots,l_n) \\
&\quad - \frac{1}{m}\sum_{\substack{x,y\in A_1\\ x\ne y}} a_0(\times^1_{x,y} l_1,\times^2_{x,y} l_1,l_2,\ldots,l_n) - \frac{1}{m}\sum_{\substack{x,y\in B_1\\ x\ne y}} a_0(\times^1_{x,y} l_1,\times^2_{x,y} l_1,l_2,\ldots,l_n)\, .
\end{align*}
\end{lmm}
\begin{proof}
Recall that $a_0(s)=0$ for every non-null $s$. By Lemmas \ref{split1} and \ref{split2}, every splitting of $s$ produces a non-null sequence. Thus, both sides of the claimed identity are equal to zero. (If $s$ cannot be split, then the right-hand side is an empty sum and hence zero by convention.)
\end{proof}
The following theorem is the main result of this section. 
\begin{thm}\label{akthm}
Let $\beta_0(d)$ and $\phi_\beta$ be as in Theorem \ref{master2}. There exists $0<\beta_1(d)\le  \beta_0(d)$ such that for any $s$, the series 
\[
\sum_{k=0}^\infty a_k(s)\beta^k
\]
converges absolutely when $|\beta|\le \beta_1(d)$ and is equal to $\phi_\beta(s)$. Moreover, when $|\beta|\le \beta_1(d)$, 
\[
\sum_{k=0}^\infty |a_k(s)||\beta|^k \le C(d)^{|s|},
\]
where $C(d)$ depends only on $d$. 
\end{thm}
\begin{proof}
Let $K = K(d)$ be as in Lemma \ref{aklmm}. If $|\beta|<K^{-5}$, then by Lemma \ref{aklmm}, the series 
\[
\psi_\beta(s) := \sum_{k=0}^\infty a_k(s)\beta^k
\]
converges absolutely for any $s$. Take a non-null loop sequence $s$ with minimal representation $(l_1,\ldots, l_n)$. Let $e$, $m$, $A_1$, $B_1$ and $C_1$ be as in the statement of Theorem \ref{mastern}. Then by the definition of $a_k(s)$, Lemma \ref{a0}, and the absolute convergence of the series defining $\psi_\beta$, we get
\begin{align*}
\psi_\beta(s) &= \frac{1}{m}\sum_{x\in A_1, \, y\in B_1} \sum_{k=0}^\infty a_k(\times_{x,y}^1 l_1, \times_{x,y}^2 l_1,l_2,\ldots,l_n)\beta^k \\
&\quad +\frac{1}{m}\sum_{x\in B_1, \, y\in A_1} \sum_{k=0}^\infty a_k(\times_{x,y}^1 l_1, \times_{x,y}^2 l_1,l_2,\ldots,l_n)\beta^k  \\
&\quad - \frac{1}{m}\sum_{\substack{x,y\in A_1\\ x\ne y}}\sum_{k=0}^\infty a_k(\times^1_{x,y} l_1,\times^2_{x,y} l_1,l_2,\ldots,l_n)\beta^k\\
&\quad  - \frac{1}{m}\sum_{\substack{x,y\in B_1\\ x\ne y}} \sum_{k=0}^\infty a_k(\times^1_{x,y} l_1,\times^2_{x,y} l_1,l_2,\ldots,l_n) \beta^k \\
&\quad + \frac{1}{m} \sum_{p\in \cp^+(e)}\sum_{x\in C_1}\sum_{k=1}^\infty a_{k-1}(l_1 \ominus_{x}p,l_2,\ldots,l_n)\beta^{k}  \\
&\quad - \frac{1}{m} \sum_{p\in \cp^+(e)}\sum_{x\in C_1}\sum_{k=1}^\infty a_{k-1}(l_1 \oplus_{x}p,l_2,\ldots,l_n)\beta^{k}\, .
\end{align*}
Applying the definition of $\psi_\beta$ on the right-hand side, this gives 
\begin{align*}
\psi_\beta(s) &= \frac{1}{m}\sum_{x\in A_1, \, y\in B_1} \psi_\beta(\times_{x,y}^1 l_1, \times_{x,y}^2 l_1,l_2,\ldots,l_n) + \frac{1}{m}\sum_{x\in B_1, \, y\in A_1} \psi_\beta(\times_{x,y}^1 l_1, \times_{x,y}^2 l_1,l_2,\ldots,l_n) \\
&\quad   - \frac{1}{m}\sum_{\substack{x,y\in A_1\\ x\ne y}} \psi_\beta(\times^1_{x,y} l_1,\times^2_{x,y} l_1,l_2,\ldots,l_n) - \frac{1}{m}\sum_{\substack{x,y\in B_1\\ x\ne y}} \psi_\beta(\times^1_{x,y} l_1,\times^2_{x,y} l_1,l_2,\ldots,l_n)   \\
&\quad + \frac{\beta}{m} \sum_{p\in \cp^+(e)}\sum_{x\in C_1}\psi_\beta(l_1 \ominus_{x}p,l_2,\ldots,l_n) - \frac{\beta}{m} \sum_{p\in \cp^+(e)}\sum_{x\in C_1}\psi_\beta(l_1 \oplus_{x}p,l_2,\ldots,l_n) 
\end{align*}
In other words, $\psi_\beta$ satisfies the master loop equation of Theorem \ref{master}. 

Now recall that $a_0(\emptyset)=1$ and $a_k(\emptyset)=0$ for every $k\ge 1$. Thus, $\psi_\beta(\emptyset)=1$. For any non-null $s$, $a_0(s)=0$. Therefore by Lemma \ref{aklmm} and the inequality \eqref{catalan2},  if $|\beta|\le (2K)^{-5}$ then for any non-null $s$ with degree vector $\delta$, 
\begin{align*}
|\psi_\beta(s)| \le \sum_{k=0}^\infty |a_k(s)||\beta|^k &\le \sum_{k=1}^\infty K^{5k + \iota(\delta)} 4^{|\delta|}|\beta|^k\\
&\le \sum_{k=1}^\infty (K^5 |\beta|)^k (4K)^{|\delta|} \le (4K)^{|\delta|}= (4K)^{|s|}\, .
\end{align*}
Thus, if $|\beta|$ is small enough, then we have verified all three conditions of Theorem \ref{master2}, proving that $\psi_\beta=\phi_\beta$. The above inequality also proves the second assertion of the theorem. 
\end{proof}

The following corollary of Theorem \ref{akthm} gives a symmetrized version of the recursion relation for~$a_k$.
\begin{cor}\label{akcor}
For any non-null $s$ and any $k\ge 1$,
\begin{align*}
a_k(s) = \frac{1}{|s|}\sum_{s'\in \fs^-(s)} a_k(s') - \frac{1}{|s|}\sum_{s'\in \fs^+(s)} a_k(s') + \frac{1}{|s|}\sum_{s'\in \fd^-(s)} a_{k-1}(s') - \frac{1}{|s|}\sum_{s'\in \fd^+(s)} a_{k-1}(s')\, .
\end{align*}
\end{cor}
\begin{proof}
Consider the two sides of the identity given by Theorem \ref{master3}. For $|\beta|$ small enough, both sides may be expanded as convergent power series in $\beta$ by Theorem \ref{akthm}. The claimed identity is obtained by equating the coefficients of the power series on the two sides.
\end{proof}

Our last task in this section is to prove Proposition \ref{algoprop}.

\begin{proof}[Proof of Proposition \ref{algoprop}]
In each term on the right-hand side of the recursion, either $k$ is replaced by $k-1$ or $s$ is replaced by a loop sequence of smaller index (by Lemma \ref{iota2}). Since the index is always a nonnegative integer and $a_k(s)$ is defined a priori whenever $k=0$ or $s=\emptyset$, this proves that the recursion must terminate. Since the recursive relation that defines the $a_k(s)$'s of Proposition \ref{algoprop} is the same as the one for the $a_k(s)$'s defined at the beginning of this section, with the same initial values, therefore the $a_k(s)$'s of Proposition \ref{algoprop} are the same as those of Theorem \ref{akthm}, which clearly are the same as those in Corollary \ref{series}. 
\end{proof}

\section{Absolute convergence of the sum over trajectories}\label{conv}
Recall the definitions of $\mx(s)$ and $w_\beta(X)$ from Section \ref{not1}. The goal of this section is to prove the following theorem.
\begin{thm}\label{convthm}
There exists $\beta_2(d)> 0$ such that if $|\beta|\le \beta_2(d)$, then for any non-null loop sequence~$s$,
\[
\sum_{X\in \mx(s)} |w_\beta(s)|<\infty\, .
\]
\end{thm}
The proof of Theorem \ref{convthm} requires some preparation. First, recall the definitions of $\fs(s)$ and $\fd(s)$ that were stated prior to the statement of Theorem~\ref{master3}. Inductively define a collection of real numbers $b_k(s)$ where $k\ge 0$ and $s\in \cs$, as follows. Let $b_0(\emptyset)=1$ and $b_0(s)=0$ for every non-null $s$. For $k\ge 1$, let $b_k(\emptyset)=0$. Having defined $b_i(s)$ for every $i<k$ and every $s$, and also $b_k(s')$ for every $s'$ with $\iota(s')<\iota(s)$, define 
\[
b_k(s) := \frac{1}{|s|}\sum_{s'\in \fs(s)} b_k(s')  + \frac{1}{|s|}\sum_{s'\in \fd(s)} b_{k-1}(s')\, .
\]
The following lemma gives an analog of Lemma \ref{aklmm} for $b_k(s)$. 
\begin{lmm}\label{bklmm}
There is a universal constant $K$ such that if $s$ is a loop sequence and $\delta = (\delta_1,\ldots,\delta_n)$ is its degree vector, then
\[
0\le b_k(s) \le K^{5k+\iota(\delta)} C_{\delta_1-1}\cdots C_{\delta_n-1}\, ,
\]
where $C_i$ is the $i^{\mathrm{th}}$ Catalan number. The product of Catalan numbers is interpreted as $1$ when~$s=\emptyset$.
\end{lmm}
\begin{proof}
Let $K=K(d)$ be the same number as in Lemma \ref{aklmm}. We will prove the claim by the same nested induction that was used to define $b_k(s)$. The fact that $b_k(s)\ge 0$ is evident from the recursive definition. So the only challenge is to prove the upper bound. 

The claim is true when $k=0$, irrespective of $s$, and it also holds for any $k$ when $s=\emptyset$. Take any $k\ge 1$ and non-null $s$, and suppose that we have proved that the bound on $b_i(s)$ holds for all $i<k$ and $s\in \cs$, and the bound on $b_k(s')$ holds for all $s'$ with $\iota(s') <\iota(s)$. Let $(l_1,\ldots, l_n)$ be the minimal representation of $s$. Let $\delta_r = |l_r|$. % and let $e$, $m$, $A_1$, $B_1$ and $C_1$ be as in the statement of Theorem~\ref{master}. 

By the induction hypothesis, Lemma \ref{split1} and the identity \eqref{catalan},
\begin{align*}
&\frac{1}{|s|} \sum_{s'\in \fs^+(s)} b_k(s')
= \frac{1}{|s|} \sum_{r=1}^n \sum_{s'\in \fs^+_r(s)} b_k(s')\\
&\le \frac{1}{|s|} \sum_{r=1}^n \sum_{1\le x\ne y\le \delta_r} K^{5k + \iota(\delta)-1} C_{\delta_1-1}\cdots C_{\delta_{r-1}-1} C_{\delta_r - |x-y|-1} C_{|x-y|-1} C_{\delta_{r+1}-1} \cdots C_{\delta_n-1}\\
&\le \frac{2}{|s|} \sum_{r=1}^n \sum_{x=1}^{\delta_r}\sum_{p=1}^{\delta_r-1} K^{5k + \iota(\delta)-1} C_{\delta_1-1}\cdots C_{\delta_{r-1}-1} C_{\delta_r - p-1} C_{p-1} C_{\delta_{r+1}-1} \cdots C_{\delta_n-1}\\
&\le \frac{2}{|s|} \sum_{r=1}^n \delta_r K^{5k + \iota(\delta)-1} C_{\delta_1-1}\cdots C_{\delta_{r-1}-1} C_{\delta_r-1}  C_{\delta_{r+1}-1} \cdots C_{\delta_n-1}\\
&= 2 K^{5k + \iota(\delta)-1}C_{\delta_1-1} \cdots C_{\delta_n-1}\, .
\end{align*}
Similarly, by the induction hypothesis, Lemma \ref{split2}, the identity \eqref{catalan} and the monotonicity of Catalan numbers,
\begin{align*}
&\frac{1}{|s|} \sum_{s'\in \fs^-(s)} b_k(s')
\le \frac{1}{|s|} \sum_{r=1}^n \sum_{s'\in \fs^-_r(s)} b_k(s')\\
&\le \frac{1}{|s|} \sum_{r=1}^n \sum_{1\le x\ne y\le \delta_r} K^{5k + \iota(\delta)-3} C_{\delta_1-1}\cdots C_{\delta_{r-1}-1} C_{\delta_r - |x-y|-2} C_{|x-y|-2} C_{\delta_{r+1}-1} \cdots C_{\delta_n-1}\\
&\le \frac{2}{|s|} \sum_{r=1}^n \sum_{x=1}^{\delta_r}\sum_{p=2}^{\delta_r-2} K^{5k + \iota(\delta)-3} C_{\delta_1-1}\cdots C_{\delta_{r-1}-1} C_{\delta_r - p-2} C_{p-2} C_{\delta_{r+1}-1} \cdots C_{\delta_n-1}\\
&\le \frac{2}{|s|} \sum_{r=1}^n \delta_r K^{5k + \iota(\delta)-3} C_{\delta_1-1}\cdots C_{\delta_{r-1}-1} C_{\delta_r-3}  C_{\delta_{r+1}-1} \cdots C_{\delta_n-1}\\
&\le 2 K^{5k + \iota(\delta)-3}C_{\delta_1-1} \cdots C_{\delta_n-1}\, .
\end{align*}
Next, note that by the induction hypothesis, Lemma \ref{deform}, and the inequality \eqref{catalan2},
\begin{align*}
\frac{1}{|s|} \sum_{s'\in \fd^+(s)} b_k(s') &\le \frac{1}{|s|} \sum_{r=1}^n \sum_{s'\in \fd^+_r(s)} b_{k-1}(s') \\
&\le \frac{1}{|s|} \sum_{r=1}^n |\fd^+_r(s)| K^{5(k-1) + \iota(\delta)+4} C_{\delta_1-1}\cdots C_{\delta_{r-1}-1} C_{\delta_r + 3} C_{\delta_{r+1}-1} \cdots C_{\delta_n-1}\\
&\le \frac{1}{|s|} \sum_{r=1}^n 2d\delta_r K^{5(k-1) + \iota(\delta)+4}4^4 C_{\delta_1-1}\cdots C_{\delta_{r-1}-1} C_{\delta_r-1} C_{\delta_{r+1}-1} \cdots C_{\delta_n-1}\\
&\le 512d K^{5k+\iota(\delta)-1} C_{\delta_1-1}\cdots C_{\delta_n-1}\, .
\end{align*}
The same bound holds when $\fd^+(s)$ is replaced by $\fd^-(s)$. Combining all of the above bounds and substituting in the definition of $b_k(s)$, we get
\begin{align*}
b_k(s) &\le (2K^{-3} + 2K^{-1} + 1024 dK^{-1}) K^{5k + \iota(\delta)} C_{\delta_1-1}\cdots C_{\delta_n-1}\, .
\end{align*} 
By our choice of $K$, the term inside the bracket is $\le 1$. This completes the proof of the lemma. 
\end{proof}
Given a loop sequence $s$, recall that $\mx(s)$ denotes the set of all vanishing trajectories that start at $s$. Let $\mx_k(s)$ denote the set of vanishing trajectories starting at $s$ that have $k$ deformations. Given a trajectory $X = (s_1,\ldots, s_n)$ and a loop sequence $s_0$, let $(s_0, X)$ denote the trajectory $(s_0,s_1,\ldots, s_n)$, provided that it is a valid trajectory. %For a vanishing trajectory $X$, let $\delta(X)$ denote the number of deformation steps in $X$ and let $\chi(X)$ denote the number of splitting steps in $X$. 
\begin{lmm}\label{mx}
For any non-null loop sequence $s$ and any $k\ge 0$, $\mx_k(s)$ is a finite set. Moreover, $\mx_0(s)$ is empty.
\end{lmm}
\begin{proof}
By Lemmas \ref{split1} and \ref{split2}, splittings of non-null loops can never give rise to null loops. Therefore a vanishing trajectory must contain at least one deformation step, proving that $\mx_0(s)$ is empty for any non-null $s$.

Next, note that by Lemmas \ref{split1} and \ref{split2}, a splitting always reduces the index of a loop sequence, and by Lemma \ref{deform}, a deformation can increase the index by at most four. Moreover by Lemma~\ref{iota1}, the index of any loop sequence is nonnegative. Therefore, if a vanishing trajectory has $k$ deformations, the maximum number of splittings it can have is bounded by a finite number that depends only on $k$ and the initial loop sequence. Therefore the length of any vanishing trajectory with $k$ deformations is bounded a finite number depending only on $k$ and the initial loop sequence. Since there can only be a finite number of vanishing trajectories of a given length starting from a given loop sequence, this shows that $\mx_k(s)$ is a finite set. 
\end{proof}
\begin{lmm}\label{replmm}
For any $\beta$, $k$ and $s$, let 
\[
S_{\beta, k}(s) := \sum_{X\in \mx_k(s)} |w_\beta(X)|\, .
\]
Then there exists $\beta_3(d)>0$ such that if $|\beta|\le \beta_3(d)$, then for any non-null loop sequence $s$,
\[
S_{\beta, k}(s) = b_k(s) |\beta|^k\, .
\]
\end{lmm}
\begin{proof}
The proof goes by our usual route of nested induction. First, note that $b_0(s)=0$ for any non-null $s$ and $\mx_0(s)$ is empty by Lemma \ref{mx}. Therefore 
\[
S_{\beta, 0}(s) = 0 = b_0(s)\, .
\]
Next, suppose that the claim has been proved for every $k'$ smaller than $k$. We will prove it for $k$ by induction on $\iota(s)$. First, take any $s$ with the smallest possible $\iota(s)$. Then by Lemmas \ref{split1} and~\ref{split2}, $\fs(s)$ is empty. Thus, 
\[
b_k(s)|\beta|^k = \frac{|\beta|}{|s|} \sum_{s'\in \fd(s)}b_{k-1}(s')|\beta|^{k-1}\, .
\]
But for any $s'\in \fd(s)$, the induction hypothesis gives
\[
b_{k-1}(s')|\beta|^{k-1} = S_{\beta, k-1}(s') = \sum_{X'\in \mx_{k-1}(s')} |w_\beta(X')|\,.
\]
Since $\fs(s)$ is empty, therefore any element of $\mx_k(s)$ may be uniquely obtained as $(s, X')$ where $X'\in \mx_{k-1}(s')$ for some $s'\in \fd(s)$. Therefore by the last two displays,
\begin{align*}
b_k(s)|\beta|^k &= \frac{|\beta|}{|s|}\sum_{s'\in \fd(s)}\sum_{X'\in \mx_{k-1}(s')} |w_\beta(X')|\\
&= \sum_{s'\in \fd(s)}\sum_{X'\in \mx_{k-1}(s')} |w_\beta(s,s')||w_\beta(X')|\\
&= \sum_{s'\in \fd(s)}\sum_{X'\in \mx_{k-1}(s')} |w_\beta(s,X')|\\
&= \sum_{X\in \mx_k(s)} |w_\beta(X)|\, .
\end{align*}
Now take any non-null $s$ and assume that claim has been for proved for all $b_k(s')$ where $\iota(s')<\iota(s)$. Then note that
\[
b_k(s)|\beta|^k = \frac{1}{|s|} \sum_{s'\in \fs(s)} b_k(s')|\beta|^k+\frac{|\beta|}{|s|} \sum_{s'\in \fd(s)}b_{k-1}(s')|\beta|^{k-1}\, .
\]
By the  induction hypotheses, this gives
\begin{align*}
b_k(s)|\beta|^k &= \sum_{s'\in \fs(s)} \sum_{X'\in \mx_k(s')} |w_\beta(s,s')|| w_\beta(X')| + \sum_{s'\in \fd(s)} \sum_{X'\in \mx_{k-1}(s')} |w_\beta(s,s') ||w_\beta(X')|\\
&= \sum_{s'\in \fs(s)} \sum_{X'\in \mx_k(s')} |w_\beta(s,X')| + \sum_{s'\in \fd(s)} \sum_{X'\in \mx_{k-1}(s')} |w_\beta(s,X') |\, .
\end{align*}
To complete the proof, note that any $X\in \mx_k(s)$ may either be obtained uniquely as $(s,X')$ for either some $X'\in \mx_k(s')$ where $s'\in\fs(s)$, or some $X'\in \mx_{k-1}(s')$ where $s'\in \fd(s)$.
\end{proof}

We are now ready to prove Theorem \ref{convthm}.

\begin{proof}[Proof of Theorem \ref{convthm}]
Since a vanishing trajectory $X$ starting at a non-null loop sequence $s$ must have finite length, therefore $X\in\mx_k(s)$ for some $k$. Therefore by Lemma \ref{replmm}, 
\begin{align*}
\sum_{X\in \mx(s)} |w_\beta(X)| &= \sum_{k=0}^\infty \sum_{X\in \mx_k(s)} |w_\beta(X)| = \sum_{k=0}^\infty b_k(s) |\beta|^k\, .
\end{align*}
Lemma \ref{bklmm} shows that the last expression converges if $|\beta|$ is small enough (depending only on $d$). This completes the proof of the theorem. 
\end{proof}

\section{Proof of Theorem \ref{mainthm} (Gauge-string duality)}\label{mainthmsec}
Finally, we are ready to prove Theorem \ref{mainthm}. We will freely use the notations and variables introduced in all preceding sections. The proof is very similar to the proof of Lemma~\ref{replmm}. For any $\beta$, $k$ and $s$, let 
\[
T_{\beta, k}(s) := \sum_{X\in \mx_k(s)} w_\beta(X)\, .
\]
Since $\mx_k(s)$ is finite by Lemma \ref{mx}, $T_{\beta, k}(s)$ is well-defined. We claim that there exists $\beta_4(d)>0$ such that if $|\beta|\le \beta_4(d)$, then for any non-null loop sequence $s$,
\begin{equation}\label{tkclaim}
T_{\beta, k}(s) = a_k(s) \beta^k\, .
\end{equation}
The proof is by nested induction. First, note that $a_0(s)=0$ for any non-null $s$ and $\mx_0(s)$ is empty by Lemma \ref{mx}. Therefore 
\[
T_{\beta, 0}(s) = 0 = a_0(s)\, .
\]
Next, suppose that the claim has been proved for every $k'$ smaller than $k$. We will prove it for $k$ by induction on $\iota(s)$. First, take any $s$ with the smallest possible $\iota(s)$. Then by Lemmas \ref{split1} and~\ref{split2}, $\fs(s)$ is empty. Thus, by Corollary \ref{akcor}, 
\[
a_k(s)\beta^k = \frac{\beta}{|s|} \sum_{s'\in \fd^-(s)}a_{k-1}(s')\beta^{k-1} -\frac{\beta}{|s|} \sum_{s'\in \fd^+(s)}a_{k-1}(s')\beta^{k-1}\, .
\]
But for any $s'\in \fd(s)$, the induction hypothesis implies that 
\[
a_{k-1}(s')\beta^{k-1} = T_{\beta, k-1}(s') = \sum_{X'\in \mx_{k-1}(s')} w_\beta(X')\,.
\]
Since $\fs(s)$ is empty, therefore any element of $\mx_k(s)$ may be uniquely obtained as $(s, X')$ where $X'\in \mx_{k-1}(s')$ for some $s'\in \fd(s)$. Therefore by the last two displays,
\begin{align*}
a_k(s)\beta^k &= \frac{\beta}{|s|}\sum_{s'\in \fd^-(s)}\sum_{X'\in \mx_{k-1}(s')} w_\beta(X') - \frac{\beta}{|s|}\sum_{s'\in \fd^+(s)}\sum_{X'\in \mx_{k-1}(s')} w_\beta(X') \\
&= \sum_{s'\in \fd(s)}\sum_{X'\in \mx_{k-1}(s')} w_\beta(s,s')w_\beta(X')\\
&= \sum_{s'\in \fd(s)}\sum_{X'\in \mx_{k-1}(s')} w_\beta(s,X')\\
&= \sum_{X\in \mx_k(s)} w_\beta(X)\, .
\end{align*}
Now take any non-null $s$ and assume that the identity for $a_k(s')\beta^k$  has been proved for  all $s'$ with smaller index. By Corollary \ref{akcor},
\begin{align*}
a_k(s)\beta^k &=  \frac{1}{|s|} \sum_{s'\in \fs^-(s)} a_k(s')\beta^k - \frac{1}{|s|} \sum_{s'\in \fs^+(s)} a_k(s')\beta^k\\
&\quad +\frac{\beta}{|s|} \sum_{s'\in \fd^-(s)}a_{k-1}(s')\beta^{k-1}-\frac{\beta}{|s|} \sum_{s'\in \fd^+(s)}a_{k-1}(s')\beta^{k-1}\, .
\end{align*}
By the  induction hypotheses, this gives
\begin{align*}
a_k(s)\beta^k &= \sum_{s'\in \fs(s)} \sum_{X'\in \mx_k(s')} w_\beta(s,s') w_\beta(X') + \sum_{s'\in \fd(s)} \sum_{X'\in \mx_{k-1}(s')} w_\beta(s,s') w_\beta(X')\\
&= \sum_{s'\in \fs(s)} \sum_{X'\in \mx_k(s')} w_\beta(s,X') + \sum_{s'\in \fd(s)} \sum_{X'\in \mx_{k-1}(s')} w_\beta(s,X')\, .
\end{align*}
To complete the proof of \eqref{tkclaim}, note that any $X\in \mx_k(s)$ may either be obtained uniquely as $(s,X')$ for either some $X'\in \mx_k(s')$ where $s'\in\fs(s)$, or some $X'\in \mx_{k-1}(s')$ where $s'\in \fd(s)$.

Having proved \eqref{tkclaim}, the proof of Theorem \ref{mainthm} is completed by observing that if $|\beta|$ is small enough (depending only on $d$), then by the identity \eqref{tkclaim} and Theorem \ref{akthm},
\[
\phi_\beta(s) = \sum_{k=0}^\infty a_k(s)\beta^k = \sum_{k=0}^\infty \sum_{X\in \mx_k} w_\beta(X)\, ,
\]
and by Theorem \ref{convthm}, the sum on the right may be reorganized to simply a sum over all $X\in \mx(s)$.

\section{Proof of Corollary \ref{factor} (Factorization of Wilson loops)}\label{factorsec}
The following simple lemma will be used in this section.
\begin{lmm}\label{trivial}
If $s$ is a non-null loop sequence and $s'$ is a splitting or deformation of $s$, then $s\ne s'$.
\end{lmm}
\begin{proof}
If $s'$ is a splitting of $s$, then by Lemmas \ref{split1} and \ref{split2}, $s'$ has more component loops than $s$, and therefore $s'\ne s$.

Now take any two non-null loops $l$ and $l'$ which can be merged at locations $x$ and $y$. If $l$ and $l'$ have the same edge $e$ at the two locations, write $l = aeb$ and $l'=ced$. Then $l\oplus_{x,y}l' = [aedceb]$. Since $l$ and $l'$ are loops, it is easy to see that the closed path $aedceb$ does not have any backtracks, and is therefore the same as $[aedceb]$. Thus, $|l\oplus_{x,y}l'|=|l|+|l'|$, and  in particular, $l\oplus_{x,y} l'\ne l$. 

Next, note that if $l'$ does not have $e$ in any other location than $y$, then the closed path $ac^{-1}d^{-1} b$ has at least one less occurrence of $e$ than the loop $l$. Since $l\ominus_{x,y} l' = [ac^{-1} d^{-1} b]$, therefore $l\ominus_{x,y} l'$ has at least one less occurrence of $e$ than the loop $l$. In particular, $l\ominus_{x,y} l'\ne l$.

In a similar manner, one can show that  the same conclusions can be drawn when the edge at the $x^{\mathrm{th}}$ location of $l$ is the inverse of the edge at the $y^{\mathrm{th}}$ location of $l'$. 

From the above discussion it follows that if $l$ is a loop and $l'$ is a deformation of $l$, then $l\ne l'$. Thus, if $s$ is a loop sequence and $s'$ is obtained by deforming the $i^{\mathrm{th}}$ component in the minimal representation of $s$, then either the $i^{\mathrm{th}}$ component of the minimal representation of $s'$ is different than that of $s$, or $s'$ has a smaller number of components in its minimal representation than $s$ (which happens when the deformation results in a null loop). In either case, $s\ne s'$. This completes the proof of the lemma. 
\end{proof}
We will also need the following general fact.
\begin{lmm}\label{general}
Let $n$ and $m$ be two positive integers. Let $a_0,a_1,\ldots, a_{n-1}$ and $b_0,\ldots, b_{m-1}$ be positive real numbers, and $a_n=b_m=0$. Let $\ma(n,m)$ be the set of all nondecreasing functions $\alpha:\{0,1,\ldots, n+m\}\ra\{0,1,\ldots, n\}$ such that $\alpha(0)=0$, $\alpha(n+m)=n$ and $\alpha(i+1)-\alpha(i)\le 1$ for each $i<n$. Then 
\begin{align*}
\sum_{\alpha \in \ma(n,m)}\prod_{i=0}^{n+m-1}\frac{1}{a_{\alpha(i)}+b_{i-\alpha(i)}} = \frac{1}{a_0a_1\cdots a_{n-1} b_0b_1\cdots b_{m-1}}\, .
\end{align*}
%where empty products stand for $1$. 
\end{lmm}
\begin{proof}
Let $L$ denote the left-hand side.  First, suppose that $n=m=1$. Then either $\alpha(0)=0=\alpha(1)$ and $\alpha(2)=1$, or $\alpha(0)=0$ and $\alpha(1)=\alpha(2)=1$. In the first case,
\begin{align*}
L &= \frac{1}{(a_0 + b_0)(a_0+b_1)} = \frac{1}{(a_0 + b_0)a_0}\, ,
\end{align*}
where the second identity holds because $b_1=0$. In the second case,
\begin{align*}
L &= \frac{1}{(a_0+ b_0)(a_1+b_0)} = \frac{1}{(a_0 + b_0)b_0}\, ,
\end{align*}
Summing the two cases gives the desired result.

Next, suppose that $n=1$ and $m >1$. Then 
\begin{align*}
L &= \sum_{i=0}^{m} \frac{1}{(a_0+b_0)(a_0+b_1)\cdots (a_0+b_i)b_ib_{i+1}\cdots b_{m-1}}\,,
\end{align*}
where the product $b_ib_{i+1}\cdots b_{m-1}$ is understood to be equal to $1$ when $i=m$. 
Suppose that the desired identity has been proved for smaller values of $m$. Then  
\begin{align*}
L &= \frac{1}{(a_0+b_0)b_0b_1\cdots b_{m-1}} + \frac{1}{(a_0+b_0)} \sum_{i=1}^{m} \frac{1}{(a_0+b_1)\cdots (a_0+b_i)b_ib_{i+1}\cdots b_{m-1}}\\
&= \frac{1}{(a_0+b_0)b_0b_1\cdots b_{m-1}} + \frac{1}{(a_0+b_0)a_0b_1\cdots b_{m-1}}\\
&= \frac{1}{a_0b_0b_1\cdots b_{m-1}}\,.
\end{align*}
The proof goes through in a similar manner when $n>1$ and $m=1$.

Finally, take $n>1$ and $m>1$ and assume that the result has been proved for all smaller values of $n+m$. Let $\ma_0(n,m)$ consist of all $\alpha\in \ma(n,m)$ with $\alpha(1)=0$, and let $\ma_1(n,m)$ consist of all $\alpha\in \ma(n,m)$ with $\alpha(1)=1$. Clearly, these two sets are disjoint and their union is $\ma(n,m)$. Take any $\alpha\in \ma_0(n,m)$ and let $\alpha'(i) := \alpha(i+1)$ for $i=0,1,\ldots, n+m-1$. It is easy to see that the map $\alpha \mapsto \alpha'$ gives a bijection between $\ma_0(n,m)$ and $\ma(n,m-1)$. Thus, if we let $b'_i = b_{i+1}$ for $i=0,1,\ldots, m-1$, then by the induction hypothesis,
\begin{align*}
\sum_{\alpha \in \ma_0(n,m)}\prod_{i=0}^{n+m-1}\frac{1}{a_{\alpha(i)}+b_{i-\alpha(i)}}&= \sum_{\alpha' \in \ma(n,m-1)}\frac{1}{(a_0+b_0)}\prod_{i=1}^{n+m-1}\frac{1}{a_{\alpha'(i-1)}+b_{i-\alpha'(i-1)}}\\
&= \frac{1}{(a_0+b_0)}\sum_{\alpha' \in \ma(n,m-1)}\prod_{i=0}^{n+m-2}\frac{1}{a_{\alpha'(i)}+b'_{i-\alpha'(i)}}\\
&= \frac{1}{(a_0+b_0)a_0a_1\cdots a_{n-1} b'_0b'_1\cdots b'_{m-2}}\\
&= \frac{1}{(a_0+b_0)a_0a_1\cdots a_{n-1} b_1b_2\cdots b_{m-1}}
\end{align*}
Next, let $\alpha''(i) := \alpha(i+1)-1$ for $i=0,1,\ldots, n+m-1$.  Again, it is easy to see that the map $\alpha \mapsto \alpha''$ gives a bijection between $\ma_1(n,m)$ and $\ma(n-1, m)$. Let $a_i' := a_{i+1}$. Then by the induction hypothesis,
\begin{align*}
\sum_{\alpha \in \ma_1(n,m)}\prod_{i=0}^{n+m-1}\frac{1}{a_{\alpha(i)}+b_{i-\alpha(i)}}&= \sum_{\alpha'' \in \ma(n,m-1)}\frac{1}{(a_0+b_0)}\prod_{i=1}^{n+m-1}\frac{1}{a_{\alpha''(i-1)+1}+b_{i-\alpha''(i-1)-1}}\\
&= \frac{1}{(a_0+b_0)}\sum_{\alpha'' \in \ma(n,m-1)}\prod_{i=0}^{n+m-2}\frac{1}{a'_{\alpha''(i)}+b_{i-\alpha''(i)}}\\
&= \frac{1}{(a_0+b_0)a_0'a_1'\cdots a'_{n-2} b_0b_1\cdots b_{m-1}}\\
&= \frac{1}{(a_0+b_0)a_1a_2\cdots a_{n-1} b_0b_1b_2\cdots b_{m-1}}
\end{align*}
Summing the last two displays gives the desired result. 
\end{proof}

If $s = (l_1,\ldots, l_n)$ and $s' = (l_1',\ldots, l_m')$ are two loop sequences, we will denote by $(s,s')$ the concatenated loop sequence $(l_1,\ldots, l_n, l_1',\ldots,l_m')$.

Let $X= (s_0,s_1,\ldots, s_n)$ and $X' = (s_0',s_1',\ldots, s_m')$ be two vanishing trajectories. Let $\alpha: \{0,1,\ldots, n+m\}\ra \{0,1,\ldots, n\}$ be a nondecreasing function such that $\alpha(0)=0$, $\alpha(n+m)=n$, and $\alpha(i+1)-\alpha(i)\le 1$ for all $i<n+m$. Define the `merging' of $X$ and $X'$ by $\alpha$ as the trajectory $\alpha(X,X')$ whose $i^{\mathrm{th}}$ component is the concatenated loop sequence  $(s_{\alpha(i)}, s'_{i-\alpha(i)})$. Let $\ma(X,X')$ be the set of all $\alpha$ as above. 
\begin{lmm}\label{bijec}
Take any two non-null loops $l$ and $l'$. For any $Y\in \mx(l,l')$, there exist unique $X\in \mx(l)$, $X'\in \mx(l')$ and $\alpha \in \ma(X,X')$ such that $Y = \alpha(X, X')$. Conversely, for any $X\in \mx(l)$, $X'\in \mx(l')$ and $\alpha \in \ma(X,X')$, $\alpha(X,X')\in \mx(l,l')$. 
\end{lmm}
\begin{proof}
%It is an easily verifiable fact that any splitting or deformation of a loop sequence is not equal to the original loop sequence. This will be used several times without explicit mention in this proof. 
The first element of $Y$ is $(l,l')$. Since loops are not allowed to merge in a trajectory, it is easy to see that any component of $Y$ is a loop sequence of the form $(s,s')$ where $s$ is a descendant of $l$ and $s'$ is a descendant of $l'$, where `descendant' means a loop sequence that may be obtained by successive deformations and splittings. Following this convention, write the $i^{\mathrm{th}}$ component of $Y$ as $(s_i, s_i')$, with $i$ running from $0$ to some finite number $k$, where $s_k=s_k'=\emptyset$. 

Since only one loop is allowed to be split or deformed at each step, therefore for each $i$, either $s_{i+1}$ is a deformation or splitting of $s_i$ and $s_{i+1}'= s_i'$, or $s_{i+1}'$ is a deformation or splitting of $s_i'$ and $s_{i+1}=s_i$. Define a sequence $\alpha(0), \alpha(1),\ldots$ inductively as $\alpha(0)=0$, and for each $i\ge 0$, 
\begin{equation}\label{arecur}
\alpha(i+1) =
\begin{cases}
\alpha(i)+1 &\text{ if $s_{i+1}\ne s_i$,}\\
\alpha(i) &\text{ if $s_{i+1}=s_i$.}
\end{cases}
\end{equation}
Similarly, define $\gamma(0)=0$, and for each $i\ge 0$,
\begin{equation}\label{grecur}
\gamma(i+1) =
\begin{cases}
\gamma(i)+1 &\text{ if $s_{i+1}'\ne s_i'$,}\\
\gamma(i) &\text{ if $s'_{i+1}=s'_i$.}
\end{cases}
\end{equation}
Note that for each $i\ge 0$, by Lemma \ref{trivial},
\[
\alpha(i+1)+\gamma(i+1)=\alpha(i)+\gamma(i)+1\, ,
\]
and therefore
\begin{equation}\label{ag}
\alpha(i) +\gamma(i) = i\,.
\end{equation}
Let $n := \alpha(k)$ and $m:= k-n$. For $0\le j\le n$, let 
\begin{equation}\label{tjdef}
t_j := s_{\alpha^{-1}(j)}\, ,
\end{equation}
where 
\[
\alpha^{-1}(j) := \min\{0\le i\le k: \alpha(i)=j\}. 
\]
Similarly, for $0\le j\le m$, let 
\begin{equation}\label{tjpdef}
t_j':= s'_{\gamma^{-1}(j)}\, .
\end{equation}
By definition of $\alpha$, it follows that if $\alpha(i)=\alpha(i')$ then $s_i = s_{i'}$. Similarly if $\gamma(i)=\gamma(i')$ then $s_i' = s'_{i'}$. Therefore for any $0\le i\le n+m$, 
\[
t_{\alpha(i)} = s_{\alpha^{-1}(\alpha(i))} = s_i\, ,
\]
and by \eqref{ag},
\[
t'_{i-\alpha(i)} = t'_{\gamma(i)} = s'_{\gamma^{-1}(\gamma(i))} = s'_i\, .
\]
Thus, if $X = (t_0, \ldots, t_n)$ and $X' = (t_0',\ldots, t_m')$, then $Y = \alpha(X,X')$. 

To prove uniqueness of the representation, take any $X = (t_0,\ldots, t_n)$, $X' = (t_0',\ldots, t_m')$ and $\alpha\in \ma(X,X')$, and let $Y = \alpha(X,X')$. Write the $i^{\mathrm{th}}$ component of $Y$ as $(s_i, s_i')$. It is clear from the definition of $\alpha(X,X')$ and Lemma \ref{trivial} that the function $\alpha$ must satisfy the recursion \eqref{arecur} and the function $\gamma(i)=i-\alpha(i)$ must satisfy the recursion \eqref{grecur}, and that $t_j$ and $t_j'$ are given by \eqref{tjdef} and \eqref{tjpdef}. This completes the proof of the one-to-one correspondence between $Y\in \mx(l,l')$ and triples $(X,X', \alpha)$ where $X\in \mx(l)$, $X'\in \mx(l')$ and $\alpha \in \ma(X,X')$.  
\end{proof}

We are now ready to prove Corollary \ref{factor}. For any trajectory $X$, let $\delta^+(X)$ and $\delta^-(X)$ be the number of positive and negative splittings of $X$, and let $\chi^+(X)$ and $\chi^-(X)$ be the number of positive and negative deformations of $X$. Note that if $X = (s_0,\ldots, s_n)$, where $s_n= \emptyset$, then
\begin{equation}\label{weight}
w_\beta(X) = \frac{(-1)^{\delta^+(X)}  (-\beta)^{\chi^+(X)} \beta^{\chi^-(X)}}{|s_0||s_1|\cdots |s_{n-1}|}\, .
\end{equation}
Let $l$ be a non-null loop. By Theorem \ref{convthm} and Lemma \ref{bijec},
\begin{align}\label{yid1}
\sum_{Y\in \mx(l,l)} w_\beta(Y) &= \sum_{X, X'\in \mx(l)}\sum_{\alpha\in \ma(X,X')} w_\beta(\alpha(X,X'))\, .
\end{align}
Fix $X = (s_0,\ldots, s_n)$ and $X' = (s_0',\ldots, s_m')$ in $\mx(l)$, where $s_0=s_0'=l$, and $s_n=s_m'=\emptyset$. Take any $\alpha \in \ma(X,X')$ and let $Y = \alpha(X,X')$. Let $t_i = (s_{\alpha(i)}, s'_{i-\alpha(i)})$ be the $i^{\mathrm{th}}$ component of $Y$.  Then  by equation \eqref{weight},
\begin{align*}
w_\beta(\alpha(X,X')) &= \frac{(-1)^{\delta^+(Y)}(-\beta)^{\chi^+(Y)} \beta^{\chi^-(Y)}}{|t_0||t_1|\cdots |t_{n+m-1}|}\\
&= \frac{(-1)^{\delta^+(X)+ \delta^+(X')} (-\beta)^{\chi^+(X)+\chi^+(X')} \beta^{\chi^-(X)+\chi^-(X')}}{|t_0||t_1|\cdots |t_{n+m-1}|}\, .
\end{align*}
Note that the numerator does not depend on $\alpha$. The dependence on $\alpha$ comes only through the term
\begin{align*}
T(\alpha) &:= \frac{1}{|t_0||t_1|\cdots |t_{n+m-1}|}\\
&= \prod_{i=0}^{n+m-1}\frac{1}{|s_{\alpha(i)}| + |s'_{i-\alpha(i)}|}\, .
\end{align*}
By Lemma \ref{general},  
\begin{equation*}%\label{tclaim}
\sum_{\alpha\in \ma(X,X')} T(\alpha) = \frac{1}{|s_0|\cdots |s_{n-1}||s'_0|\cdots |s'_{m-1}|}\, .
\end{equation*}
Combining all of the above calculations, we obtain 
\begin{align}\label{yid2}
\sum_{\alpha\in \ma(X,X')} w_\beta(\alpha(X,X')) = w_\beta(X)w_\beta(X')\, .
\end{align}
Thus, by Theorem \ref{mainthm} and the identities \eqref{yid1} and \eqref{yid2},
\begin{align*}
\lim_{N\ra\infty} \frac{\smallavg{W_{l}}^2}{N^2} &= \sum_{X, X'\in \mx(l)} w_\beta(X)w_\beta(X')\\
&= \sum_{X,X'\in \mx(l)} \sum_{\alpha \in \ma(X,X')} w_\beta(\alpha(X,X'))\\
&= \sum_{Y\in \mx(l,l)} w_\beta(Y) = 
\lim_{N\ra\infty} \frac{\smallavg{W_{l}^2}}{N^2}\, .
\end{align*}
Thus, for any $l_1,\ldots, l_n$, using the inequality $|W_l|\le N$ and the Cauchy--Schwarz inequality we get
\begin{align*}
& N^{-n}|\smallavg{W_{l_1}\cdots W_{l_n}} - \smallavg{W_{l_1}\cdots W_{l_{n-1}} } \smallavg{W_{l_n}}|\\
&\le  N^{-n}|\smallavg{W_{l_1}\cdots W_{l_{n-1}}(W_{l_n}- \smallavg{W_{l_n}})}|\\
&\le  N^{-1}\smallavg{|W_{l_n}-\smallavg{W_{l_n}}|}\\
&\le  N^{-1}\smallavg{(W_{l_n}-\smallavg{W_{l_n}})^2}^{1/2} = N^{-1}(\smallavg{W_{l_n}^2}-\smallavg{W_{l_n}}^2)^{1/2}\, .
\end{align*}
By the previous display, the last term tends to zero as $N\ra\infty$. The proof of Corollary \ref{factor} can now be easily completed using induction on $n$.

\section{Proof of Corollary \ref{area} (Area law upper bound)}\label{areasec}
We will use the following two lemmas.
\begin{lmm}\label{rlmm}
For any non-null loop $l$, $\textup{area}(l)\le $ the minimum number of deformations in a vanishing trajectory starting from $l$. 
\end{lmm}
\begin{proof}
Let $r$ be the map defined prior to the statement of Corollary \ref{area}. Extend the definition of $r$ as follows. Let $r(\emptyset)=0$, and for a loop sequence $s= (l_1,\ldots, l_n)$, let 
\[
r(s) := r(l_1)+\cdots +r(l_n)\, .
\]
Note that if $\rho$ is a path and $\rho'$ is obtained from $\rho$ by a backtrack erasure, then $r(\rho)=r(\rho')$. Thus, for any closed path $l$, 
\begin{equation}\label{ll}
r(l)=r([l])\, .
\end{equation}
Using this it is easy to see that if $s'$ is a splitting of $s$, then $r(s)=r(s')$. 

Next, take any two loops $l$ and $l'$ and locations $x$ in $l$ and $y$ in $l'$ such that $l$ and $l'$ can be merged at $x$ and $y$. First, suppose that the $x^{\mathrm{th}}$ edge of $l$ is the same as the $y^{\mathrm{th}}$ edge of $l'$. Then it follows easily from definition and the identity \eqref{ll}  that
\[
r(l\oplus_{x,y}l') = r(l)+r(l')
\]
and
\[
r(l\ominus_{x,y}l') = r(l)-r(l')\,.
\]
On the other hand, if the $x^{\mathrm{th}}$ edge of $l$ is the inverse of the $y^{\mathrm{th}}$ edge of $l'$, then
\[
r(l\oplus_{x,y}l') = r(l)-r(l')
\]
and
\[
r(l\ominus_{x,y}l') = r(l)+r(l')\,.
\]
Combining the above observations, it follows that if $X= (s_0,s_1,\ldots,s_n)$ is a vanishing trajectory with $s_n=\emptyset$ and $k$ deformations, and $p_1,\ldots, p_k$ are the plaquettes involved in the deformations, then
\[
r(s_0) = r(s_0)-r(s_n)=\sigma_1r(p_1)+\cdots +\sigma_k r(p_k)
\]
for some $\sigma_1,\ldots, \sigma_k\in \{-1,1\}$. Now let $x$ be the lattice surface
\[
x := \sigma_1p_1+\cdots +\sigma_k p_k\, .
\]
Then note that 
\begin{align*}
\textup{area}(x) &\le |\sigma_1|+\cdots +|\sigma_k| = k\, ,
\end{align*}
and 
\begin{align*}
\delta(x) &= \sigma_1\delta(p_1)+\cdots +\sigma_k \delta(p_k)\\
&= \sigma_1r(p_1)+\cdots +\sigma_k r(p_k) = r(s_0)\, .
\end{align*}
Thus, if $s_0$ is equal to a single loop $l$, then $x$ is a surface of area $\le k$ and with boundary $l$. This completes the proof of the lemma. 
\end{proof}
\begin{lmm}\label{newlmm}
If $l$ is a non-canceling loop, then $|l|\le 4\,\textup{area}(l)$.
\end{lmm}
\begin{proof}
Take any $2$-chain $x$ such that $\delta(x)=r(l)$, in the notation introduced immediately preceding the statement of Corollary \ref{area}. Take any edge $e$ in $l$. Recall that $\cp^+(e)$ is the set of all positively oriented plaquettes that contain either $e$ or $e^{-1}$. Let $m(e)$ be the number of occurrences of $e$ in $l$. The non-canceling nature of $l$ implies that
\[
\sum_{\substack{p\in \cp^+(e)\\ e\in p}} n_p - \sum_{\substack{p\in \cp^+(e)\\ e^{-1}\in p}} n_p= m(e), 
%\begin{cases}
%m(e) &\text{ if $e\in E^+$,}\\
%-m(e) &\text{ if $e\in E^-$.} 
%\end{cases}
\]
irrespective of whether $e$ is positively or negatively oriented. In particular,
\[
\sum_{p\in \cp^+(e)} |n_p|\ge m(e).
\]
Summing over all distinct $e\in l$, and noting that each plaquette contains at most four distinct edges of $l$, gives the desired result.
\end{proof}
We are now ready to finish the proof of Corollary \ref{area}. It is easy to give an argument that is purely derived from the statement of Theorem \ref{mainthm}, but to save space we will use some facts that have already been proved while proving Theorem~\ref{mainthm}. Let $a_k(l)$ be the coefficient defined in Section~\ref{power}. 
Then by Theorem \ref{akthm},
\begin{align*}
\lim_{N\ra\infty} \frac{\smallavg{W_l}_{\Lambda_N, N, \beta}}{N} &= \sum_{k=0}^\infty a_k(l)\beta^k\, .
\end{align*}
Moreover by second inequality in Theorem \ref{akthm}, there is a constant $K$, depending only on $d$, such that 
\[
|a_k(l)|\le K^{k+|l|}
\]
for all $k$. By equation~\eqref{tkclaim} and Lemma \ref{rlmm}, $a_k(l)=0$ when $k<\textup{area}(l)$. Thus,
\[
\lim_{N\ra\infty} \frac{|\smallavg{W_l}_{\Lambda_N, N, \beta}|}{N}\le \sum_{k\ge \textup{area}(l)} K^{k+|l|} |\beta|^k\,.
\]
By Lemma \ref{newlmm}, this completes the proof of Corollary~\ref{area}. % when $|\beta|$ is small enough. Since $|W_l|\le N$ for all $l$, we can simply increase the value of $C(d)$ to render the statement valid for all $\beta$. 

\section{Proof of Corollary \ref{part} (Limit of partition function)}\label{partsec}
Let $\Lambda_N'$ be the same as $\Lambda_N$, except that vertices on opposite faces are connected by edges to give it the graph structure of a torus. Define $SO(N)$ lattice gauge theory on $\Lambda_N'$ just as on $\Lambda_N$, except that there are some additional plaquettes due to presence of the extra edges. Let $\cp_N'$ be the set of plaquettes in $\Lambda_N'$ and let $Z_{\Lambda_N', N,\beta}$ be the partition function of this new theory.

Let $K$ be a compact set equipped with its Borel sigma algebra, and let $\mu$ be a finite positive measure on $K$. If $H$ and $H'$ are two bounded measurable real-valued functions on $K$, then
\begin{align*}
\log \int_K e^{H(x)} d\mu(x) - \log \int_K e^{H'(x)} d\mu(x) &= \log \frac{\int_K e^{H(x)-H'(x)} e^{H'(x)} d\mu(x)}{\int_K e^{H'(x)} d\mu(x)}\\
&\le \log \sup_{x\in K}e^{H(x)-H'(x)} = \sup_{x\in K}(H(x)-H'(x))\, .
\end{align*}
Since the number of additional plaquettes in $\cp_N'$ is of order $N^{d-1}$ and $|\tr(Q_p)|\le N$ for all $p$, the above inequality can be used to show that 
\[
|\log Z_{\Lambda_N, N, \beta} - \log Z_{\Lambda_N', N, \beta}| \le C(d)|\beta| N^{d+1}\, ,
\]
where $C$ is a  constant that depends only on $d$. In particular,
\[
\lim_{N\ra\infty} \frac{\log Z_{\Lambda_N, N, \beta} - \log Z_{\Lambda_N', N, \beta}}{N^2|\Lambda_N|} = 0\, .
\]
Since $Z_{\Lambda_N', N, 0} = 1$, therefore for any $\beta_1 > 0$, 
\begin{align*}
\log Z_{\Lambda_N', N, \beta_1} &= \int_0^{\beta_1} \fpar{}{\beta} \log Z_{\Lambda_N', N, \beta}\,d\beta\\
&= \sum_{p\in\cp_N'} N\int_0^{\beta_1} \smallavg{W_p}_{\Lambda_N', N, \beta}\, d\beta
\end{align*}
The main useful feature of $\Lambda_N'$ is its symmetry: each term in the above sum has the same value. Thus, for any fixed plaquette $p$,
\begin{align*}
\log Z_{\Lambda_N', N, \beta_1} &=  |\cp_N'| N\int_0^{\beta_1} \smallavg{W_p}_{\Lambda_N', N, \beta}\, d\beta\, .
\end{align*}
Following the steps in the proof of Theorems \ref{mastern} and \ref{master} with $\Lambda_N'$ instead of $\Lambda_N$, it can be shown that any limit point of $(N^{-1}\smallavg{W_s}_{\Lambda_N', N, \beta})_{s\in \cs}$ satisfies the master loop equation of Theorem \ref{master}. Therefore by Theorem \ref{master2} and Theorem \ref{mainthm}, it follows that if $|\beta|$ is sufficiently small, then
\[
\lim_{N\ra\infty} \frac{\smallavg{W_p}_{\Lambda_N', N, \beta} }{N} = \lim_{N\ra\infty} \frac{\smallavg{W_p}_{\Lambda_N, N, \beta} }{N} = \sum_{X\in \mx(p)} w_\beta(X)\, .
\]
It is easy to see that the proof of Theorem \ref{convthm} actually gives something slightly stronger: 
\[
\sup_{0\le \beta\le \beta_1} \sum_{X\in \mx(p)} |w_\beta(X)|<\infty\, .
\]
Therefore by the bounded convergence theorem for integrals, we get
\begin{align*}
\lim_{N\ra\infty}\int_0^{\beta_1} \frac{\smallavg{W_p}_{\Lambda_N', N, \beta}}{N}\, d\beta &= \sum_{X\in \mx(p)}\int_0^{\beta_1}w_\beta(X)\, d\beta\,.
\end{align*}
Now, $w_\beta(X) = v(X) \beta^{\delta(X)}$ where $v(X)$ does not depend on $\beta$. Thus,
\[
\int_0^{\beta_1}w_\beta(X)\, d\beta = \frac{\beta_1 w_{\beta_1}(X)}{\delta(X)+1}\, .
\]
The proof is now completed by observing that $|\cp_N'|/|\Lambda_N|\ra d(d-1)/2$ as $N\ra\infty$. The proof is similar when $\beta_1< 0$.

\section{Proof of Corollary \ref{series} (Real analyticity at strong coupling)}\label{seriessec}
Just like Corollary \ref{area}, Corollary \ref{series} can be derived purely from the statement of Theorem \ref{mainthm}. However, for the sake of saving space, we will prove Corollary \ref{series} using facts that have been proven while proving Theorem \ref{mainthm}. 

The proof of the first power series expansion in the statement of the corollary is already implicit in the proof of Theorem \ref{mainthm}. Specifically, it follows easily from equation \eqref{tkclaim} and Theorem~\ref{akthm}. For the second series, one simply needs to combine Corollary \ref{part} with equation \eqref{tkclaim} and Theorem~\ref{convthm}.

\section{Proof of Lemma \ref{core} (Uniqueness of nonbacktracking core)}\label{coresec}
Write $\rho = e_1e_2\cdots e_n$. Then $n$ is the length of $\rho$ and $\rho'$. The lemma will be proved by induction on $n$. If $n=0$, then $\rho$ and $\rho'$ are null paths with no backtracks. So the only possibility is that $\rho_1=\rho = \rho' = \rho_2$.

Suppose that the claim has been proved for all paths with length less than $n$. If $\rho$ has no backtracks, then neither does $\rho'$. In this case there is nothing to prove, since the only possibility is that $\rho_1=\rho$ and $\rho_2=\rho'$.

So assume that $\rho$ has at least one backtrack, so that $\rho'$ also has at least one backtrack. Since $\rho_1$ and $\rho_2$ are nonbacktracking, therefore $\rho_1\ne \rho$ and $\rho_2\ne \rho'$. Let $\tau_1$ be the path produced after the first backtrack erasure in the sequence that produces $\rho_1$. Let $\tau_2$ be the path produced after the first backtrack erasure in the sequence that produces $\rho_2$.

If $\tau_1\sim\tau_2$, then the induction hypothesis implies that $\rho_1\sim\rho_2$, and there is nothing more to prove. So assume that $\tau_1\not \sim \tau_2$. Let $\tau_1$ be produced by erasing a backtrack at a location $i$ of $\rho$ and $\tau_2$ be produced by erasing a backtrack at a location $i'$ of $\rho'$. Since $\rho$ is a cyclic permutation of $\rho'$, there is a location $j$ in $\rho$ that corresponds to the location $i'$ in $\rho'$. 

If $j=i$, then clearly $\tau_1\sim \tau_2$, so this cannot be the case. Therefore assume that $j\ne i$. We claim that $j\ne i+1$. This is proved by contradiction. Suppose that $j=i+1$. There are two cases. First, if $n>j$, then $e_{i+2}=e_{i+1}^{-1} = e_i$ since $j=i+1$ and $\rho$ has backtracks at $i$ and $j$. Thus, 
\[
\tau_1 = e_1\cdots e_{i-1}e_{i+2}e_{i+3}\cdots e_n = e_1 \cdots e_{i-1} e_i e_{i+3}\cdots e_n \sim \tau_2\, ,
\] 
which is a contradiction. Next, if $n=j$, then $e_1 = e_n^{-1} = e_{n-1}$. Thus,
\[
\tau_1 = e_1e_2\cdots e_{n-2} = e_{n-1}e_2e_3\cdots e_{n-2} \sim e_2e_3\cdots e_{n-1}\sim\tau_2\, ,
\] 
which, again, is a contradiction. Therefore we have established the claim that $j\ne i+1$. A similar argument shows that $j\ne i-1$. 

Suppose that $i< j$. Since $j\ne i+1$ and $\rho$ has backtracks at $i$ and $j$, the following is a well-defined closed path:
\[
\tau_3 := e_1e_2\cdots e_{i-1}e_{i+2}\cdots e_{j-1} e_{j+2} \cdots e_n\, .
\] 
Note that $\tau_3$ can be obtained by a single backtrack erasure from $\tau_1$, and is cyclically equivalent to a path $\tau_4$ that can be obtained by a single backtrack erasure from $\tau_2$. Continue erasing backtracks from $\tau_3$ and $\tau_4$ to arrive at   nonbacktracking  closed paths $\rho_3$ and $\rho_4$ respectively. Applying the induction hypothesis to the pair $(\tau_3, \tau_4)$, we get $\rho_3\sim \rho_4$. 

Now, $\rho_1$ and $\rho_3$ are both obtained by erasing backtracks starting from $\tau_1$, possibly in different orders. Therefore, by the induction hypothesis, $\rho_1\sim \rho_3$. Similarly, $\rho_2$ and $\rho_4$ are both obtained by erasing backtracks starting from $\tau_2$. Thus, by the induction hypothesis, $\rho_4\sim \rho_2$.  Since $\rho_3\sim \rho_4$ as observed in the previous paragraph, these two observations show that $\rho_1\sim \rho_2$. This completes the proof of the lemma when $i< j$. If $i>j$, the proof goes through similarly after defining $\tau_3 =  e_1e_2\cdots e_{j-1}e_{j+2}\cdots e_{i-1} e_{i+2} \cdots e_n$.

\section{Open problems}\label{opensec}
There are many open problems in the mathematics of lattice gauge theories. Here is a tentative list of some problems that are most closely related to this paper:
\begin{enumerate}[1.]
\item Can there be a simplification of the formula given in Theorem \ref{mainthm}, for example along the lines of the reduction proposed by Eguchi and Kawai \cite{eguchikawai82}?
\item Do the Wilson loop expectations, suitably rescaled, converge as $N\ra\infty$ for any value of $\beta$? Does the rescaled log-partition function converge?
\item Assuming that Wilson loop expectations converge, is the string representation of Theorem \ref{mainthm} valid for all $\beta$? If not, where does it break down? What is the formula for the limiting Wilson loop expectation when $\beta$ is large?
\item Is the limit of the rescaled log-partition function, assuming it exists, a real analytic function of $\beta$ with infinite radius of convergence? If not, where does analyticity break down? The answer to this question is known in two dimensions: Gross and Witten \cite{grosswitten80} and Wadia \cite{wadia} showed using arguments that are almost rigorous, that in two dimensional $U(N)$ lattice gauge theory, the limit of the log-partition function is not analytic, and identified the exact point of phase transition. Theorem \ref{mainthm} gives hope that such a phase transition may be proved in higher dimensions. 
\item Is there any way to build a similar theory when the lattice scaling is taken to zero?
\item Is it possible to exploit the finite $N$ master loop equation (Theorem \ref{mastersymm}) to prove interesting results --- such as area law or a formula for Wilson loop expectations --- without taking $N$ to infinity?
\item Is it possible to prove an area law lower bound for general loops, either for finite $N$, or in the limit $N\ra\infty$, using the techniques of this paper? This would generalize a result of Seiler \cite{seiler78}, who proved it for rectangles. 
\item Is it possible to extend the techniques developed here to other settings, such as AdS/CFT?
\end{enumerate}

\vskip.2in

\noindent{\bf Acknowledgments.} I thank  Amir Dembo, Persi Diaconis, Bruce Driver, Len Gross, Alice Guionnet, Jafar Jafarov, Todd Kemp, Herbert Neuberger, Steve Shenker, Lior Silberman, Tom Spencer and Akshay Venkatesh for many helpful discussions and comments.  I am grateful to the referee for a number of useful suggestions, and to H.-T.~Yau for his enthusiasm about getting this paper published in CMP.%, and Jafar Jafarov for carefully checking the proof. %Thanks to Abdelmalek Abdesselam for pointing out the reference \cite{levy11}.

\end{document}